\documentclass{amsart}

\usepackage{enumerate, amsmath, amsfonts, amssymb, amsthm, wasysym, graphics, graphicx, url, hyperref, hypcap, a4wide, pdflscape, multido, xargs, colortbl, multicol, multirow, calc, shuffle, hvfloat, dsfont, mathrsfs, float, pifont, xparse}
\hypersetup{colorlinks=true, citecolor=darkblue, linkcolor=darkblue}
\usepackage{tikz}
\usepackage[dvipsnames]{xcolor}
\usepackage{pgfplots}
\usetikzlibrary{trees, decorations, decorations.markings, shapes, arrows, matrix, calc, fit, intersections, patterns, angles, cd, backgrounds,pgfplots.fillbetween}
\usepackage{tikz-qtree}
\usepackage{comment}
\usepackage{etex}
\usepackage{ulem}
\usepackage[noabbrev,capitalise]{cleveref}
\graphicspath{{figures/}{figures/}}
\makeatletter
\def\input@path{{figures/}{figures/}}
\makeatother

\title[Interval hypergraphic polytopes, Tamari interval posets, and weeping willows]{\mbox{Interval hypergraphic polytopes (or deformed associahedra),} \\ Tamari interval posets, and weeping willows}

\author[J.~Bastidas]{Jose Bastidas}
\address[J.~Bastidas]{LACIM, Université du Québec à Montréal, Canada}
\email{bastidas.math@proton.me}
\urladdr{https://bastidas-jose.codeberg.page/}

\author[F.~Gélinas]{Félix Gélinas}
\address[F.~Gélinas]{York University, Canada}
\email{felixgel@yorku.ca}
\urladdr{https://felixgelinas.github.io/}

\author[V.~Pilaud]{Vincent Pilaud}
\address[V.~Pilaud]{Universitat de Barcelona \& Centre de Recerca Matemàtica, Barcelona}
\email{vincent.pilaud@ub.edu}
\urladdr{https://www.ub.edu/comb/vincentpilaud/}

\author[G.~Poullot]{Germain Poullot}
\address[G.~Poullot]{Osnabr\"uck Universit\"at, Germany}
\email{germain.poullot@uni-osnabrueck.de}
\urladdr{https://embrunforestier.github.io/germainpoullot.github.io/}

\author[A.~Sack]{Andrew Sack}
\address[A.~Sack]{University of Michigan, United States of America}
\email{asack@umich.edu}
\urladdr{https://andrewsack.com/}

\author[E.~Tzanaki]{Eleni Tzanaki}
\address[E.~Tzanaki]{University of Crete, Greece}
\email{etzanaki@uoc.gr}
\urladdr{https://sites.google.com/view/tzanakel/homepage}

\thanks{VP was supported by the Spanish project PID2022-137283NB-C21 of MCIN/AEI/10.13039/501100011033 / FEDER, UE, by the Spanish--German project COMPOTE (AEI PCI2024-155081-2 \& DFG 541393733), and by the Severo Ochoa and María de Maeztu Program for Centers and Units of Excellence in R\&D (CEX2020-001084-M)}

\newtheorem{theorem}{Theorem}[section]
\newtheorem{corollary}[theorem]{Corollary}
\newtheorem{proposition}[theorem]{Proposition}
\newtheorem{lemma}[theorem]{Lemma}
\newtheorem{conjecture}[theorem]{Conjecture}
\crefname{conjecture}{Conjecture}{Conjectures}
\newtheorem{problem}[theorem]{Problem}
\crefname{problem}{Problem}{Problems}
\newtheorem{theoremA}{Theorem}

\crefname{theoremA}{Theorem}{Theorems}
\newtheorem{propositionA}[theoremA]{Proposition}

\crefname{propositionA}{Proposition}{Propositions}

\theoremstyle{definition}
\newtheorem{definition}[theorem]{Definition}
\newtheorem{example}[theorem]{Example}
\newtheorem{remark}[theorem]{Remark}

\crefname{notation}{Notation}{Notations}

\AddToHook{env/theorem/begin}{\crefalias{section}{theorem}}
\AddToHook{env/proposition/begin}{\crefalias{theorem}{proposition}}
\AddToHook{env/corollary/begin}{\crefalias{theorem}{corollary}}
\AddToHook{env/lemma/begin}{\crefalias{theorem}{lemma}}
\AddToHook{env/conjecture/begin}{\crefalias{theorem}{conjecture}}
\AddToHook{env/definition/begin}{\crefalias{theorem}{definition}}
\AddToHook{env/example/begin}{\crefalias{theorem}{example}}
\AddToHook{env/remark/begin}{\crefalias{theorem}{remark}}
\AddToHook{env/observation/begin}{\crefalias{theorem}{observation}}
\AddToHook{env/question/begin}{\crefalias{theorem}{question}}
\AddToHook{env/notation/begin}{\crefalias{theorem}{notation}}
\AddToHook{env/assumption/begin}{\crefalias{theorem}{assumption}}
\AddToHook{env/convention/begin}{\crefalias{theorem}{convention}}
\AddToHook{env/theoremA/begin}{\crefalias{theorem}{theorem}}
\AddToHook{env/propositionA/begin}{\crefalias{theoremA}{proposition}}

\newcommand{\R}{\mathbb{R}} 
\newcommand{\I}{\mathbb{I}} 
\renewcommand{\b}[1]{\boldsymbol{#1}} 

\renewcommand{\c}[1]{{\mathcal{#1}}} 

\newcommand{\set}[2]{\left\{ #1 \;\middle|\; #2 \right\}} 
\newcommand{\bigset}[2]{\big\{ #1 \;|\; #2 \big\}} 
\newcommand{\ssm}{\smallsetminus} 
\newcommand{\eqdef}{\mbox{\,\raisebox{0.2ex}{\scriptsize\ensuremath{\mathrm:}}\ensuremath{=}\,}} 
\newcommand{\defeq}{\mbox{~\ensuremath{=}\raisebox{0.2ex}{\scriptsize\ensuremath{\mathrm:}} }} 
\newcommand{\simplex}{\triangle} 

\DeclareMathOperator{\conv}{conv} 
\DeclareMathOperator{\inv}{inv} 

\newcommand{\ie}{\textit{i.e.}~} 
\newcommand{\eg}{\textit{e.g.}~} 
\newcommand{\aka}{\textit{aka.}~} 
\newcommand{\para}[1]{\medskip\noindent\uline{#1}} 
\definecolor{darkblue}{rgb}{0,0,0.7} 
\definecolor{green}{RGB}{57,181,74} 
\definecolor{violet}{RGB}{147,39,143} 
\newcommand{\red}{\color{red}} 
\newcommand{\blue}{\color{blue}} 
\newcommand{\orange}{\color{orange}} 
\newcommand{\green}{\color{green}} 
\newcommand{\darkblue}{\color{darkblue}} 
\newcommand{\defn}[1]{\textsl{\darkblue #1}} 
\newcommand{\mathdefn}[1]{{\darkblue #1}} 
\renewcommand{\emptyset}{\varnothing}

\usepackage{todonotes}

\newcommand{\OEIS}[1]{\href{https://oeis.org/#1}{\cite[#1]{OEIS}}} 

\newcommand{\fS}{\mathfrak{S}} 

\newcommand{\Or}{\mathcal O}  
\newcommand{\BB}{\mathbb B}  
\newcommand{\GG}{\mathbb G}  
\newcommand{\FF}{\mathbb F}  
\newcommand{\HH}{\mathbb H}  
\newcommand{\II}{\mathbb I} 
\newcommand{\JJ}{\mathbb J} 
\newcommand{\KK}{\mathbb K} 

\newcommand{\intervalHypergraph}[2]{
	\begin{tikzpicture}[baseline=0]
		\foreach \x in {1,...,#1} {
			\node (\x) at (\x*.5,-.3) [inner sep = -1pt] {$\scriptstyle \x$};
		}
		\newcount{\y} \y=0
		\foreach \a/\b in {#2} {
			\draw [thick,{Bar[width=3pt]}-{Bar[width=3pt]}] (\a*.5,\y*.2)--(\b*.5,\y*.2);
			\global\advance\y by 1
		}
		\node at (.5,0) {\phantom{$\bullet$}};
		\node at (#1*.5,0) {\phantom{$\bullet$}};
	\end{tikzpicture}
}

\newcommand{\acyclicOrientation}[2]{
	\begin{tikzpicture}[baseline=0]
		\foreach \x in {1,...,#1} {
			\node (\x) at (\x*.5,-.3) [inner sep = -1pt] {$\scriptstyle \x$};
		}
		\newcount{\y} \y=0
		\foreach \a/\b/\c in {#2} {
			\draw [thick,{Bar[width=3pt]}-{Bar[width=3pt]}] (\a*.5,\y*.2)--(\b*.5,\y*.2); \node at (\c*.5,\y*.2) {$\bullet$};
			\global\advance\y by 1
		}
		\node at (.5,0) {\phantom{$\bullet$}};
		\node at (#1*.5,0) {\phantom{$\bullet$}};
	\end{tikzpicture}
}

\newcommand{\weepingWillow}[3]{
	\begin{tikzpicture}[decoration={markings, mark=at position 0.7 with {\arrow{>}}}, baseline=0, yscale={#3}]
		\foreach \x in {1,...,#1} {
			\node (\x) at (\x*.5,-.2) [inner sep = -1pt] {$\scriptstyle \x$};
		}
		\foreach \a/\b in {#2} {
			\ifthenelse{\a<\b}{\draw[postaction={decorate}] (\a*.5,0) arc (180:0:\b*.25-\a*.25);}{\draw[postaction={decorate}] (\a*.5,0) arc (0:180:\a*.25-\b*.25);}
		}	
	\end{tikzpicture}
}

\newcommand{\SchroderWeepingWillow}[4]{\begin{tikzpicture}[decoration={markings, mark=at position 0.7 with {\arrow{>}}}, baseline=0, scale={#4}]
	\foreach \x in {1,...,#1} {
		\node (\x) at (\x*.5,-.2) [inner sep = -1pt] {$\scriptstyle \x$};
	}

	\foreach \a/\b in {#2} {
		\ifthenelse{\a<\b}{\draw[postaction={decorate}] (\a*.5,0) arc (180:0:\b*.25-\a*.25);}{\draw[postaction={decorate}] (\a*.5,0) arc (0:180:\a*.25-\b*.25);}
	}

	\foreach \liste in {#3} {
		\def\prev{}%
		\edef\mypath{}%
		\def\len{}%
		\def\first{}%
		\def\last{}%
		\foreach \x [count=\i] in \liste {%
			\ifnum\i=1
				\xdef\mypath{(0.5*\x,0)}
				\xdef\first{\x}
				\xdef\prev{\x}
			\else
				\xdef\mypath{\mypath arc (180:0:{0.25*(\x-\prev)})}
				\xdef\prev{\x}
				\xdef\len{\i}
				\xdef\last{\x}
			\fi
		}
		\draw[name path=A] \mypath;
		\pgfmathsetmacro{\ys}{ifthenelse(\len==2,1.5,1)}
		\draw[yscale=\ys, name path=B] (0.5*\first,0) arc (180:0:{0.25*(\last-\first)});
		\tikzfillbetween[of=A and B]{pattern=north west lines, pattern color=gray};
	}
\end{tikzpicture}}

\newcommand{\TamariIntervalPoset}[3]{
	\begin{tikzpicture}[decoration={markings, mark=at position 0.7 with {\arrow{>}}}, baseline=0, yscale={#3}]
		\foreach \x in {1,...,#1} {
			\node (\x) at (\x*.5,-.2) [inner sep = -1pt] {$\scriptstyle \x$};
		}
		\foreach \a/\b in {#2} {
			\ifthenelse{\a<\b}{\draw[postaction={decorate}] (\a*.5,0) arc (180:0:\b*.25-\a*.25);}{\draw[postaction={decorate}] (\a*.5,-.4) arc (360:180:\a*.25-\b*.25);}
		}	
	\end{tikzpicture}
}

\newcommand{\TamariIntervalPreposet}[4]{\begin{tikzpicture}[decoration={markings, mark=at position 0.7 with {\arrow{>}}}, baseline=0, scale={#4}]
	\foreach \x in {1,...,#1} {
		\node (\x) at (\x*.5,-.2) [inner sep = -1pt] {$\scriptstyle \x$};
	}

	\foreach \a/\b in {#2} {
		\ifthenelse{\a<\b}{\draw[postaction={decorate}] (\a*.5,0) arc (180:0:\b*.25-\a*.25);}{\draw[postaction={decorate}] (\a*.5,-.4) arc (360:180:\a*.25-\b*.25);}
	}

	\foreach \liste in {#3} {
		\def\prev{}%
		\edef\mypath{}%
		\def\len{}%
		\def\first{}%
		\def\last{}%
		\foreach \x [count=\i] in \liste {%
			\ifnum\i=1
				\xdef\mypath{(0.5*\x,0)}
				\xdef\first{\x}
				\xdef\prev{\x}
			\else
				\xdef\mypath{\mypath arc (180:0:{0.25*(\x-\prev)})}
				\xdef\prev{\x}
				\xdef\len{\i}
				\xdef\last{\x}
			\fi
		}
		\draw[name path=A] \mypath;
		\pgfmathsetmacro{\ys}{ifthenelse(\len==2,1.5,1)}
		\draw[yscale=\ys, name path=B] (0.5*\first,0) arc (180:0:{0.25*(\last-\first)});
		\tikzfillbetween[of=A and B]{pattern=north west lines, pattern color=gray};
	}
\end{tikzpicture}}

\NewDocumentCommand{\drawSubset}{O{blue} O{white} O{white} m}{
	\begin{tikzpicture}[scale=.3]

	\path[use as bounding box] (0,0) rectangle (1.4,0.8660254);

	\foreach \i/\x/\y in {
	  1/0/0,
	  2/0.25/0.4330127,
	  3/0.5/0.8660254,
	  4/0.5/0,
	  5/0.75/0.4330127,
	  6/1/0
	}{
	  \coordinate (v\i) at (\x,\y);
	}

	\draw[line width=2.5mm, line join=round, line cap=round, color=#2] (v1) -- (v6) -- (v3) -- cycle;
	\draw[line width=2mm, line join=round, line cap=round, color=#3!40] (v1) -- (v6) -- (v3) -- cycle;

	\foreach \i in {1,2,3,4,5,6}{
	  \draw[color=#1] (v\i) circle (0.2);
	}

	\foreach \i in {#4}{
	  \fill[color=#1] (v\i) circle (0.2);
	}
	
	\end{tikzpicture}
}

\newcommand{\less}{\vartriangleleft} 
\renewcommand{\lessdot}{\mathbin{\vartriangleleft\hspace{-7pt}\cdot\,}} 
\newcommand{\more}{\vartriangleright} 
\newcommand{\moredot}{\mathbin{\vartriangleright\hspace{-9pt}\cdot\,\,}} 
\newcommand{\bless}{\blacktriangleleft} 
\newcommand{\bmore}{\blacktriangleright} 

\newcommand{\ww}{{W\!\!W}} 
\newcommandx{\WW}[1][1=n]{\mathbb{W}\!\mathbb{W}\!_{#1}} 
\newcommand{\sww}{S\!W\!\!W}
\DeclareMathOperator{\tour}{tour}
\newcommandx{\canoArc}[2][1=A,2=B]{#1\smallfrown#2}

\newcommandx{\Perm}[1][1=n]{\mathds{P}\mathrm{erm}(#1)} 
\newcommandx{\Asso}[1][1=n]{\mathds{A}\mathrm{sso}(#1)} 
\newcommand{\poly}[1]{\mathds{#1}} 
\newcommand{\pol}{\poly{P}} 

\begin{document}

\begin{abstract}
For a hypergraph~$\HH$ on~$[n]$, the hypergraphic polytope~$\simplex_\HH$ is the Minkowski sum of the standard simplices~$\simplex_H$ for all~$H \in \HH$.
We focus here on interval hypergraphs, where all hyperedges are intervals of~$[n]$.
They are precisely the deformations of Loday's associahedron.
Their vertex posets are Tamari interval posets, and we describe which Tamari interval poset appears as a vertex poset in which interval hypergraphic polytope.
We also characterize the interval hypergraphs~$\II$ for which the hypergraphic polytope~$\simplex_\II$ is simple, and we study their vertex posets, which we call weeping willows.
\end{abstract}

\maketitle

\begin{figure}[h]
    \centering


\definecolor{myproperpurple}{RGB}{224, 176, 255}

\begin{tikzpicture}[scale=1.575, yscale=0.73]
\draw[blue, line width=2pt, rounded corners=6pt] (-3.85,-4.88) rectangle (3.85, 4);
\node at (0,3.8){\color{blue}HYPERGRAPHIC POLYTOPES};

\draw[myproperpurple,  line width=2pt,rounded corners=6pt] (-3.75, 3.5) rectangle (1, 0.13);
\node[fill=white, fill opacity=0.7, inner sep=1pt] at ({-3.75/2 + 0.5}, 3.25){\color{myproperpurple}NESTOHEDRA (building sets)};

\draw[cyan,  line width=2pt,rounded corners=6pt] (-1.75, 3) rectangle (3.75, -3.75);
\node[fill=white, fill opacity=0.7, inner sep=1pt] at ({-1.75/2 + 3.75/2}, 2.8){\color{cyan}INTERVAL HYPERGRAPHIC POLYTOPES};

\draw[red,  line width=2pt,rounded corners=6pt] (-3.75, -4.75) rectangle (1, 0);
\node[fill=white, fill opacity=0.7, inner sep=1pt] at ({0.5 - 3.75/2}, -0.2){\color{red}UNIFORM HYPERGRAPHIC POLYTOPES};

\draw[orange,  line width=2pt, rounded corners=6pt] (-3.65, -0.5) -- (-0.5, -0.5) -- (-0.5, -1.5) -- (-1.85, -1.5) -- (-1.85, -3.75) -- (-3.65, -3.75) -- cycle;
\node[fill=white, fill opacity=0.7, inner sep=1pt] at ({-0.25 - 1.85}, -0.7){\color{orange}GRAPHICAL ZONOTOPES};

\node[anchor = west] at (-3.2, -1){Chordful};
\node[anchor = west] at (-3.2, -1.25){\begin{scriptsize}Prop.\,\ref{prop:chordful}\end{scriptsize}};
\node[anchor = west] at (-1.7, -1){Path zono.};
\node[anchor = west] at (-1.7, -1.25){\begin{scriptsize}Ex.\,\ref{exm:orientationPath},\,\ref{exm:cube}\end{scriptsize}};

\node[anchor = west] at (-1.66, 2.2){Associahedra};
\node[anchor = west] at (-1.66, 1.95){\begin{scriptsize}Ex.\,\ref{exm:binarySearchTrees},\,\ref{exm:associahedron}\end{scriptsize}};

\node[anchor = west] at (-1.66, 1.45){Fertilitopes};
\node[anchor = west] at (-1.66, 1.2){\begin{scriptsize}Ex.\,\ref{ex:fertilitopes}\end{scriptsize}};

\node[anchor = west] at (-1.66, .7){Pitman--Stanley pol.};
\node[anchor = west] at (-1.66, 0.45){\begin{scriptsize}Ex.\,\ref{exm:PitmanStanleyTrees},\,\ref{exm:PitmanStanley}\end{scriptsize}};

\node[anchor = west] at (1.75, 1.66){Freehedra};
\node[anchor = west] at (1.75, 1.4){\begin{scriptsize}Ex.\,\ref{exm:freehedronTrees},\,\ref{exm:freehedron}\end{scriptsize}};

\node[anchor = west] at (1.05, -2.25){Capped unit interval pol.};
\node[anchor = west] at (1.05, -2.5){\begin{scriptsize}Sec.\,\ref{subsubsec:cappedUnitIntervalPosets}, Ex.\,\ref{exm:cappedUnitIntervalHypergraphicPolytope}\end{scriptsize}};

\node[anchor = west] at (-1.35, -2.25){Uniform interval pol.};
\node[anchor = west] at (-1.35, -2.5){\begin{scriptsize}Sec.\,\ref{subsubsec:uniformIntervalPosets}, Ex.\,\ref{exm:uniformIntervalHypergraphicPolytope}\end{scriptsize}};

\draw[green, ultra thick] (-4.5, -1.62) -- (4.2, -1.62);

\draw[green] (-4.1, 1) node{\rotatebox{90}{\textbf{Simple}}};
\draw[green] (-4.1, -3.2) node{\rotatebox{90}{\textbf{Non-simple}}};
\end{tikzpicture}

    \caption{The main families of hypergraphic polytopes discussed in this paper.}
    \label{fig:WhereIsWhat}
\end{figure}
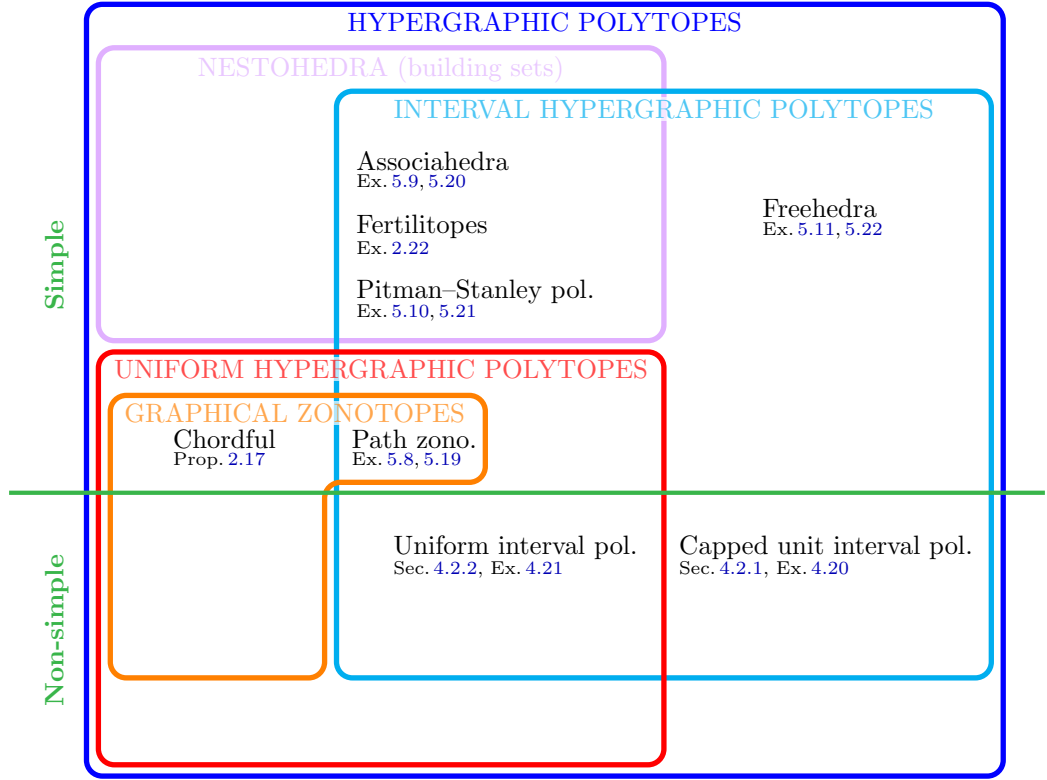

\pagebreak
\tableofcontents



\pagebreak
\section{Introduction}
\label{sec:introduction}

\para{Hypergraphic polytopes.}
We fix an integer~$n \ge 1$, and denote by~$(\mathdefn{\b{e}_i})_{i \in [n]}$ the standard basis of~$\R^n$. 
The \defn{hypergraphic polytope} of a hypergraph~$\mathdefn{\HH}\subseteq 2^{[n]}$ on~$[n]$ is the Minkowski sum~$\mathdefn{\simplex_\HH} \eqdef \! \sum_{H\in \HH} \simplex_H$,
where $\mathdefn{\simplex_H}$ is the simplex given by the convex hull of the points $\b{e}_h \in \R^n$ for~$h \in H$ (\cref{subsec:D_H}).
The face lattice of~$\simplex_\HH$ was described combinatorially in terms of acyclic orientations of~$\HH$ in~\cite{BenedettiBergeronMachacek} (\cref{subsec:acyclicOrientations,subsec:acyclicPreorientations}).

Hypergraphic polytopes (or some special cases) have been studied in~\cite{FeichtnerSturmfels, PostnikovReinerWilliams, Postnikov, AgnarssonMorris, Agnarsson, BenedettiBergeronMachacek, PadrolPilaudPoullot-deformationConesHypergraphicPolytopes, Rehberg, AguiarArdila, CardinalHoangMerinoMickaMutze, CardinalSteiner, ABGPS, BergeronPilaud} among others.
Important examples of hypergraphic polytopes include the permutahedron (when~$\HH = \binom{[n]}{2}$ is the complete graph), the associahedron of~\cite{ShniderSternberg,Loday} (when~${\HH = \set{[i,j]}{1 \le i \le j \le n}}$ is the complete interval hypergraph), graphical zonotopes (when~$\HH \subseteq \binom{[n]}{2}$ is a graph), graph associahedra~\cite{CarrDevadoss}, nestohedra~\cite{FeichtnerSturmfels,Postnikov}, multiplihedra~\cite{Stasheff-HSpaces, SaneblidzeUmble-diagonals, Forcey-multiplihedra, ArdilaDoker, ChapotonPilaud, PilaudPoullot2025PivotPolytope}, constrainahedra~\cite{BottmanPoliakova, ChapotonPilaud, PilaudPoullot2025PivotPolytope}, and other more specific examples given below (\cref{subsec:exmHP}).

Hypergraphic polytopes belong to the more general class of \defn{deformed permutahedra} (\aka \defn{generalized permutahedra}~\cite{Postnikov}, or \defn{polymatroids}~\cite{Edmonds}).
The later are all deformations of the permutahedron, \ie polytopes whose normal fans coarsen the braid arrangement (\cref{subsec:deformedPermutahedra}).
They are all obtained as Minkowski sums and differences of faces of the standard simplex $\simplex_{[n]}$.
Within the class of deformed permutahedra, the hypergraphic polytopes are precisely those which can be constructed using only Minkowski \emph{sums} of faces of the standard simplex.

The edge directions of deformed permutahedra are all of the form~$\b{e}_i - \b{e}_j$.
Hence, each vertex $\b{v}$ of a deformed permutahedron $\pol$ defines a directed graph, called its \defn{vertex digraph}, with nodes~$[n]$ and where there is an arc $i \to j$ if $\pol$ has a neighboring vertex $\b{w}$ of $\b{v}$ such that $\b{w} - \b{v}$ is a positive multiple of~$\b{e}_j - \b{e}_i$.
The corresponding \defn{vertex poset} is the transitive closure of the vertex digraph.
Similarly, each face of a deformed permutahedron defines a preposet (\ie a reflexive and transitive relation), called its \defn{face preposet}.
In this paper, we will study the combinatorics of these vertex posets and face preposets for certain hypergraphic polytopes.


\para{Interval hypergraphic polytopes.}
Following~\cite{BergeronPilaud}, we study in this paper the case of \defn{interval hypergraphs}~$\II$, \ie when all hyperedges of~$\II$ are intervals of~$[n]$, namely $\II \subseteq \set{[i,j]}{1 \le i \le j \le n}$.
Note that the family of interval hypergraphic polytopes includes
\begin{itemize}
\item the classical associahedron of~\cite{ShniderSternberg,Loday} when~$\II$ contains all intervals of~$[n]$,
\item the Pitman--Stanley polytope~\cite{PitmanStanley} when~$\II$ is the set of all initial intervals~$[i]$~for~${i \in [n]}$,
\item the freehedron of~\cite{Saneblidze-freehedron} when~$\II$ is the set of all initial intervals~$[i]$ for~${i \in [n]}$ and all final intervals~$[n] \ssm [i]$~for~${i \in [n-1]}$,
\item the fertilitopes of~\cite{Defant-fertilitopes} when any two intervals of~$\II$ are either nested or disjoint.
\end{itemize}
See \cref{subsec:examplesWeepingWillows}.
The permutahedron, graphical zonotopes, graph associahedra, nestohedra, multiplihedra and constrainahedra are not interval hypergraphic polytopes (except for trivial examples).

Two combinatorially interesting families of interval hypergraphs are those closed under intersection, and those closed under union (of intersecting hyperedges).
The former corresponds to interval hypergraphic lattices \cite[Thm.~A]{BergeronPilaud} while the latter corresponds to interval nestohedra (in the sense of~\cite{FeichtnerSturmfels,Postnikov}).
An elementary size preserving bijection between these two families of interval hypergraphs was described in~\cite{Pilaud-MarioLuigi}, passing through certain permutations known to be counted by median Genocchi numbers.
We reproduce this argument (\Cref{subsec:Genocchi}) in order to place it in a more formal context, noting that interval nestohedra form an important class of simple interval hypergraphic polytopes.

By definition, the interval hypergraphic polytopes are deformations of the classical associahedron, \ie all polytopes whose normal fan coarsens the sylvester fan.
It actually follows from~\cite{BazierMatteChapelierLaguetDouvilleMousavandThomasYildirim,PadrolPaluPilaudPlamondon,PadrolPilaudPoullot-deformedNestohedra} that the converse also holds when we allow for positive scaling of the Minkowski summands (which preserves the normal fan).
Namely, any deformation of the associahedron is an interval hypergraphic polytope~$\sum_{1 \le i \le j \le n} \lambda_{ij} \simplex_{[i,j]}$ with~$\lambda_{ij} \ge 0$ for all~$1 \le i < j \le n$ and~$\lambda_{ii} \in \R$ for~$i \in [n]$.

\para{Tamari interval (pre)posets.}
\enlargethispage{.2cm}
Interval hypergraphic polytopes are also closely related to the Tamari interval posets of~\cite{ChatelPons}.
These are the posets~$\less$ on~$[n]$ such that~$a \less c$ implies~$a \less b$ while~$a \more c$ implies~$b \more c$ for all~$1 \le a < b < c \le n$ (\cref{subsec:TamariIntervalPosets,subsec:examplesTamariIntervalPosets}).
They are in correspondence with intervals of the Tamari lattice, are counted by remarkable product formulas~\cite{Chapoton1,Chapoton2}, and are also in bijection with relevant families of planar maps~\cite{BernardiBonichon,FangFusyNadeau} (\cref{subsec:countingTamariIntervalPosets}).
The following connects Tamari interval posets and interval hypergraphic polytopes (\cref{subsec:TIPIHP}).

\begin{propositionA}[{\Cref{prop:TamariIntervalPosets1,prop:TamariIntervalPosets2}}]
\label{prop:TamariIntervalPosets}
The interval hypergraphic polytopes are precisely the deformed permutahedra whose vertex posets are Tamari interval posets.
Moreover, any Tamari interval poset is a vertex poset of some interval hypergraphic polytope.
\end{propositionA}

In more detail (\cref{subsec:finerDescriptionTamariIntervalPoset}), we characterize which Tamari interval posets appear as vertex posets of which interval hypergraphic polytopes (\Cref{thm:TamariIntervalPosetInIntervalHypergraphic}).
We then consider, for a given Tamari interval poset~$\less$, the inclusion poset~$\c{I}_\less$ of interval hypergraphs $\II$ such that~$\less$ is a vertex poset of~$\simplex_\II$.
This poset~$\c{I}_\less$ is an order convex subposet of the boolean lattice, always admits a maximum, and we characterize when it admits a minimum (\Cref{prop:inclusionPosetIntervalHypergraphsContainingTamariIntervalPoset}).

In this paper, we also consider the face preposets of interval hypergraphic polytopes.
We prove (\cref{subsec:TIPPIHP}) that they are precisely the preposets~$\bless$ on $[n]$ such that $a \bless c$ implies~$a \bless b$ while~$a \bmore c$ implies~$b \bmore c$ for all~$1 \le a < b < c \le n$, which we call \defn{Tamari interval preposets} (\cref{subsec:TamariIntervalPreposets}).
We again describe which Tamari interval preposets appear as face preposets of which interval hypergraphic polytopes (\Cref{thm:TamariIntervalPreposetInIntervalHypergraphic}) and study the inclusion poset~$\c{I}_\bless$ of interval hypergraphs $\II$ such that~$\bless$ is a face preposet of~$\simplex_\II$ (\cref{prop:inclusionPosetIntervalHypergraphsContainingTamariIntervalPreposet}).

\para{Simple interval hypergraphic polytopes.}
A polytope is \defn{simple} if each vertex is incident to dimension-many edges (or, equivalently, facets).
For a generalized permutahedron, this means that its vertex digraphs are forests (actually trees if the polytope is of maximal dimension $n-1$).
Among the above-mentioned examples of hypergraphic polytopes, the permutahedron, associahedron, graph associahedra, and nestohedra are all simple, while the multiplihedra and constrainahedra are not (except in low dimension).
As observed in~\cite[Rem.~6.2]{Kim} (see also \cite[Prop.~5.2]{PostnikovReinerWilliams} and \cite[Prop.~53]{Pilaud-acyclicReorientationLattices}), the graphical zonotope of a graph~$G$ is simple if and only if~$G$ is \defn{chordful} (meaning that any cycle induces a clique).
In contrast, it is unclear whether simple hypergraphic polytopes admit an elementary characterization (\Cref{subsec:towardsSimpleHypergraphicPolytopes}).
Going back to interval hypergraphs, we prove the following characterization of simple interval hypergraphic polytopes.

\begin{theoremA}[{\Cref{thm:characterizationSimpleIntervalHypergraphicPolytopesProof}}]
\label{thm:characterizationSimpleIntervalHypergraphicPolytopes}
The hypergraphic polytope $\simplex_\II$ of an interval hypergraph~$\II$ is simple if and only if the following two conditions hold for all $I, J \in \II$:
\begin{enumerate}[(i)]
\item if $I \cap J \ne \varnothing$ and there exists $K \in \II$ with $I \cup J \subseteq K$, then $I \cup J \in \II$, and 
\item if $|I \cap J| \ge 2$, then $I \cup J \in \II$ or there exists $K \in \II$ with $I \cap J \subseteq K$ and $I \not\subseteq K$ and $J \not\subseteq K$. 
\end{enumerate}
\end{theoremA}

\para{(Schröder) weeping willows.}
\enlargethispage{.2cm}
We then consider the vertex trees of these simple interval hypergraphic polytopes.
We call \defn{weeping willow} any directed tree~$\ww$ on~$[n]$ such that for all~${1 \le a < b < c \le n}$, if~$\ww$ contains the arc~$(a,c)$ (resp.~$(c,a)$), then it contains a directed path from~$a$ to~$b$ (resp.~from~$c$ to~$b$) (\cref{subsec:weepingWillows}).
In other words, these are precisely the Hasse diagrams of the Tamari interval posets which are trees.
We first observe that these trees have interesting counting formulas (\cref{subsec:countingWeepingWillows}).
We then specialize \cref{prop:TamariIntervalPosets} to the following statement.

\begin{propositionA}[{\Cref{prop:weepingWillows1,prop:weepingWillows2}}]
\label{prop:weepingWillows}
The simple interval hypergraphic polytopes are precisely the deformed permutahedra whose vertex digraphs are weeping willows.
Moreover, any weeping willow is a vertex tree of some simple interval hypergraphic polytope.
\end{propositionA}

This result can even be extended to characterize interval nestohedra (\ie when the underlying interval hypergraph is closed under union of intersecting hyperedges).

\begin{propositionA}[{\Cref{prop:rootedWeepingWillows1,prop:rootedWeepingWillows2}}]
\label{prop:rootedWeepingWillows}
The interval nestohedra are precisely the deformed permutahedra whose vertex digraphs are rooted weeping willows.
Moreover, any rooted weeping willow is a vertex tree of some interval nestohedron.
\end{propositionA}

Again, using our characterization of which Tamari interval posets appear as vertex posets of which interval hypergraphic polytopes, we describe the posets of simple hypergraphic polytopes and of interval nestohedra containing a given weeping willow (\Cref{prop:inclusionPosetIntervalHypergraphsContainingWeepingWillow,prop:inclusionPosetIntervalBuildingSetsContainingRootedWeepingWillow}).
We finally extend these results to face preposets of simple interval hypergraphic polytopes, which we call \defn{Schröder weeping willows} (\cref{sec:SchroderWeepingWillows}).


\section{Preliminaries}
\label{sec:preliminaries}


\subsection{Deformed permutahedra}
\label{subsec:deformedPermutahedra}

The \defn{Minkowski sum} of two polytopes~$\poly{P}, \poly{Q} \subseteq \R^n$ is the polytope~$\mathdefn{\poly{P} + \poly{Q}} \eqdef \set{\b p+\b q}{\b p \in \poly{P}, \b q \in \poly{Q}}$.
A \defn{deformation} of a polytope~$\poly{P}$ is a polytope~$\poly{Q}$ which satisfies the following equivalent conditions:
\begin{itemize}
\item $\poly{Q}$ is a weak Minkowski summand of~$\poly{P}$, meaning that there exists~$\lambda > 0$ and a polytope~$\poly{R}$ such that~$\lambda \poly{P} = \poly{Q} + \poly{R}$,
\item the normal fan of~$\poly{Q}$ coarsens the normal fan of~$\poly{P}$,
\item $\poly{Q}$ can be obtained from~$\poly{P}$ by parallelly translating its facets without moving past vertices,
\item $\poly{Q}$ can be obtained from~$\poly{P}$ by moving its vertices in such a way that all edge directions and orientations are preserved.
\end{itemize}

The \defn{permutahedron~$\Perm$} is the polytope in~$\R^n$ obtained equivalently as:
\begin{itemize}
\item the convex hull of the points~$\sum_{i \in [n]} i \, \b{e}_{\sigma(i)}$ for all permutations~$\sigma$ of~$[n]$, see~\cite{Schoute},
\item the intersection of the hyperplane~$\poly{H} = \smash{\bigset{\b{x} \in \R^n}{\sum_{i \in [n]} x_i = \binom{n+1}{2}}}$ with the half-spaces $\smash{\bigset{\b{x} \in \R^n}{\sum_{i \in I} x_i \ge \binom{|I|+1}{2}}}$ for all~${\varnothing \ne I \subsetneq [n]}$, see \cite{Rado},
\item (a translation of) the Minkowski sum of the segments~$[\b{e}_i, \b{e}_j]$ for~$1\leq i < j \leq n$.
\end{itemize}
See \cref{fig:exmHypergraphicPolytopes}.
The (outer) normal fan of the permutahedron~$\Perm$ is the \defn{braid fan}, defined by the (type~$A$) Coxeter arrangement formed by the hyperplanes~$\set{\b{x} \in \R^n}{x_i = x_j}$ for all~${1 \le i < j \le n}$.
Each permutation~$\sigma$ of~$[n]$ corresponds to a maximal cone~$\mathdefn{\poly{C}_\sigma} \eqdef \set{\b{x} \in \R^n}{x_{\sigma(1)} \le \dots \le x_{\sigma(n)}}$ of the braid fan, consisting of all points whose coordinates are ordered by the permutation~$\sigma$.

\begin{figure}[b]
	\centerline{\includegraphics[width=\linewidth]{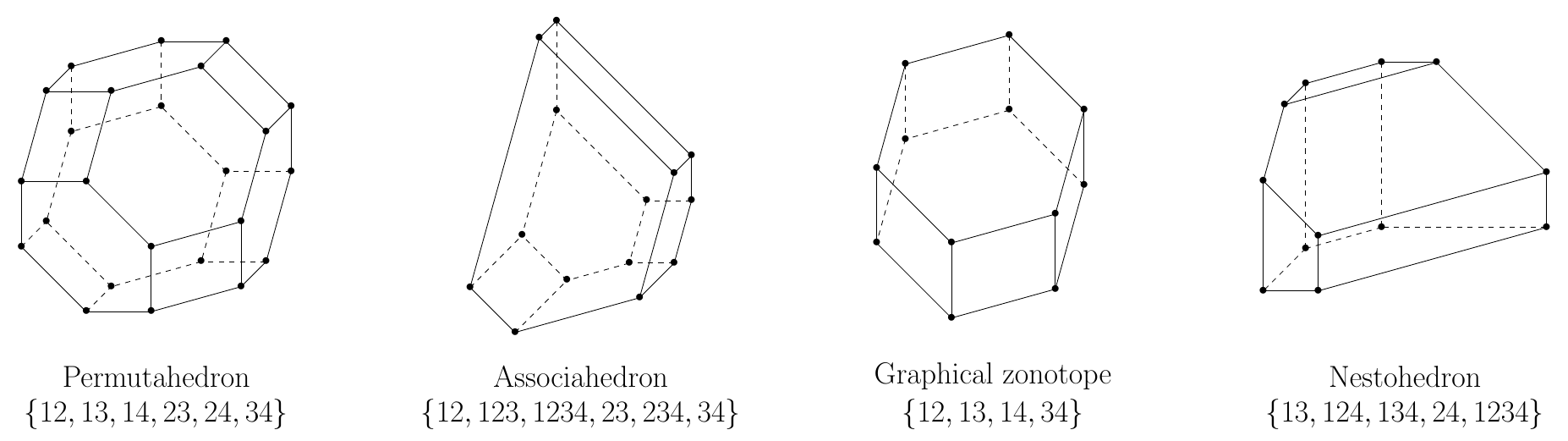}}
	\caption{Examples of 3-dimensional hypergraphic polytopes.}
	\label{fig:exmHypergraphicPolytopes}
\end{figure}

A \defn{deformed permutahedron} (\aka \defn{generalized permutahedron}~\cite{Postnikov}, or \defn{polymatroid}~\cite{Edmonds}) is a deformation of the permutahedron~$\Perm$.
A classical example of a deformed permutahedron is the \defn{associahedron~$\Asso$}, defined equivalently as
\begin{itemize}
\item the convex hull of the points~$\sum_{i \in [n]} \ell(T,i) \, r(T,i) \, \b{e}_i$ for all binary trees~$T$ on~$n$ nodes, where $\ell(T,i)$ and~$r(T,i)$ respectively denote the numbers of leaves in the left and right subtrees of the $i$-th node of~$T$ in infix labeling, see~\cite{Loday},
\item the intersection of the hyperplane~$\poly{H}$ with the halfspaces~${\bigset{\b{x} \in \R^n}{\sum_{i \le \ell \le j} x_\ell \ge \binom{j-i+2}{2}}}$ for all~$1 \le i \le j \le n$, see~\cite{ShniderSternberg},
\item the Minkowski sum of the faces~$\simplex_{[i,j]}$ of the standard simplex~$\simplex_{[n]}$ for ${1 \le i \le j \le n}$, see~\cite{Postnikov}.
\end{itemize}
See \cref{fig:exmHypergraphicPolytopes}.
The (outer) normal fan of the associahedron is the \defn{sylvester fan}.
Each binary tree~$T$ with $n$ nodes corresponds to a maximal cone~$\mathdefn{\poly{C}_T} \eqdef \set{\b{x} \in \R^n}{x_i \le x_j \text{ for $i$ child of $j$ in~$T$}}$ of the sylvester fan.
Other examples of deformed permutahedra include matroid polytopes, graphical zonotopes, graph associahedra~\cite{CarrDevadoss}, nestohedra~\cite{FeichtnerSturmfels}, quotientopes~\cite{PilaudSantos-quotientopes,PadrolPilaudRitter}, brick polytopes~\cite{PilaudSantos-brickPolytope}, multiplihedra~\cite{Stasheff-HSpaces, SaneblidzeUmble-diagonals, Forcey-multiplihedra, ArdilaDoker, ChapotonPilaud}, constrainahedra~\cite{BottmanPoliakova, ChapotonPilaud}, and many others.

Following~\cite{PostnikovReinerWilliams}, we now recall the dictionary between vertices (resp.~faces) of deformed permutahedra and integer posets (resp.~preposets).
Note that our convention for vertex poset and face preposets is reversed from that of~\cite{PostnikovReinerWilliams} in order to fit with that of~\cite{BenedettiBergeronMachacek, BergeronPilaud}.
In particular, several of our arguments are formulated in terms of minimizing directions rather than maximizing ones.

For a vertex~$\b{v}$ of a deformed permutahedron~$\poly{Q}$, we call \defn{vertex digraph} of~$\b{v}$ in~$\poly{Q}$ the directed graph on $[n]$ with an arc~$i \to j$ for each neighbor~$\b{w}$ of~$\b{v}$~$\poly{Q}$ such that~$\b{w}-\b{v}$ is a positive multiple of~$\b{e}_j - \b{e}_i$, and \defn{vertex poset $\less$} of~$\b{v}$ in~$\poly{Q}$ its transitive closure (note that vertex digraphs are acyclic, otherwise it would contradict $\b v$ being a vertex).
In other words, the (outer) normal cone of~$\b{v}$ in~$\poly{Q}$ is given by~$\set{\b{x} \in \R^n}{x_i \ge x_j \text{ for all } i \less j}$.
In particular, the linear extensions of the opposite of the vertex poset of~$\b{v}$ in~$\poly{Q}$ are precisely the permutations $\sigma$ such that the cone~$\poly{C}_\sigma$ of the braid arrangement is contained in the normal cone of~$\b{v}$ in~$\poly{Q}$.
Moreover, note that $\b{v}$ is a simple vertex of~$\poly{Q}$ if and only if its vertex digraph is a forest.

More generally, for a face~$\poly{F}$ of~$\poly{Q}$, we call \defn{face preposet~$\bless$} of~$\poly{F}$ in~$\poly{Q}$ the preposet on~$[n]$ such that the (outer) normal cone of~$\poly{F}$ in~$\poly{P}$ is given by~$\set{\b{x} \in \R^n}{x_i \ge x_j \text{ for all } i \bless j}$.
Recall that a \defn{preposet}~$\bless$ is a reflexive and transitive relation, and defines an equivalence relation~$\set{(a,b)}{a \bless b \text{ and } b \bless a}$ and a poset~$\set{(A,B)}{a \bless b \text{ for some (or all) } a \in A \text{ and } b \in B}$ on the equivalence classes of this equivalence relation.
We will always denote by~$\less$ the poset associated to the preposet~$\bless$.
We call \defn{face digraph} of~$\poly{F}$ in~$\poly{Q}$ the Hasse diagram of the poset~$\less$ (its vertex set being a partition of~$[n]$).
If~$\b{v}$ is a simple vertex, then the face digraphs of the faces containing~$\b{v}$ are precisely the forests obtained by iterative edge contractions in the vertex digraph of~$\b{v}$.

For instance, the vertex (resp.~face) digraphs of the associahedron~$\Asso$ are precisely the binary (resp.~Schröder) trees with $n+1$ leaves, labeled in inorder, and oriented towards their leaves, see \Cref{exm:binarySearchTrees} (resp.~\cref{exm:faceLatticeAssociahedron}).
Recall that a \defn{Schröder tree} is a rooted plane tree where each node has at least two children (or equivalently a tree obtained from a binary tree by iterative edge contractions).



\subsection{Hypergraphic polytopes}
\label{subsec:D_H}

A \defn{hypergraph} $\HH$ on~$[n]$ is a collection of subsets of~$[n]$.
The \defn{hypergraphic polytope}~$\simplex_\HH$ is the Minkowski sum
\[
\mathdefn{\simplex_\HH} \eqdef \sum_{H\in \HH} \simplex_H\,,
\]
where $\mathdefn{\simplex_H} \eqdef \conv\set{\b{e}_h}{h \in H}$ is the simplex given by the convex hull of the points $\b{e}_h \in \R^n$ for~$h \in H$.
Note that~$\simplex_\HH$ is always a deformed permutahedron, as it is a Minkowski sum of faces of the standard simplex.

\begin{example}
\label{exm:DH1}
For the hypergraph $\HH=\{ {\orange 123}, {\green 134} \}$,
we  have
\[
\begin{array}{ccc}
\begin{tikzpicture}[scale=1.5,baseline=.5cm]
	\coordinate (v1) at (.6,.75);
	\coordinate (v2) at (0,0);
	\coordinate (v3) at (1.2,0);
	\draw[white, fill=orange!7] (v1) -- (v2) -- (v3) -- cycle;
    \node (1) at (.6,.75) {$\scriptstyle \orange 1$};
	\node (2) at (0,0) {$\scriptstyle \orange 2$};
	\node (3) at (1.2,0) {$\scriptstyle \orange 3$};
	\draw (1) -- (2) -- (3) -- (1);
\end{tikzpicture}
\quad + & \quad
 \begin{tikzpicture}[scale=1.5,baseline=.5cm]
    \coordinate (v1) at (.6,.75);
	\coordinate (v3) at (1.2,0);
	\coordinate (v4) at (2.2,.5);
	\draw[white, fill=green!7] (v1) -- (v3) -- (v4) -- cycle;
    \node (1) at (.6,.75) {$\scriptstyle \green 1$};
	\node (3) at (1.2,0) {$\scriptstyle \green 3$};
	\node (4) at (2.2,.5) {$\scriptstyle \green 4$};
    \draw (1) -- (3) -- (4) -- (1);
\end{tikzpicture}
\quad = &
\begin{tikzpicture}[scale=1.5,baseline=.5cm]
		\coordinate (v23) at (0,0);
		\coordinate (v21) at (-.6,.75);
		\coordinate (v24) at (1,.5);
        \coordinate (v34) at (2.2,.5);
		\coordinate (v14) at (1.6,1.25);
        \draw[white, fill=green!7] (v23) -- (v21) -- (v24) -- cycle;
        \draw[white, fill=orange!7] (v24) -- (v34) -- (v14) -- cycle;
        \node (23) at (0,0) {$\scriptstyle ({\orange 2}, {\green 3})$}; 		
		\node (21) at (-.6,.75) {$\scriptstyle ({\orange 2}, {\green 1})$};	
		\node (24) at (1,.5) {$\scriptstyle ({\orange 2}, {\green 4})$}; 	
		\node (34) at (2.2,.5) {$\scriptstyle ({\orange 3}, {\green 4})$};
		\node (14) at (1.6,1.25) {$\scriptstyle ({\orange 1}, {\green 4})$};	
		\node (33) at (1.2,0) {$\scriptstyle ({\orange 3}, {\green 3})$};	
		\node (11) at (0,1.5) {$\scriptstyle ({\orange 1}, {\green 1})$};		
		\path (11) edge  (21);
		\path[dotted] (11) edge (33);
		\path (11) edge (14);
		\path (14) edge  (34) ;
		\path (14) edge (24)  ;
		\path (24) edge (34)  ;
		\path (21) edge (24);
		\path (21) edge (23);
		\path (24) edge (23);
		\path (23) edge (33);
		\path (33) edge (34);
\end{tikzpicture}\\
\orange{\simplex_{123}} & \green{\simplex_{134}} & \simplex_\HH = {\orange \simplex_{123}} + {\green \simplex_{134}}\\
\end{array}
\]
which is a 3-dimensional polytope sitting in $\R^4$.
We simplify notation, writing~$123$ for~$\{1,2,3\}$.
\end{example}

\begin{remark}
\label{rem:singletons}
Note that the singleton hyperedges are irrelevant for our purposes.
Namely, adding to~$\HH$ the hyperedge~$\{i\}$ for some~$i \in [n]$ just translates the polytope~$\simplex_\HH$ in the direction~$\b{e}_i$, which does not affect the face structure (nor the normal fan) of the polytope.
For some statements, it is convenient to assume that $\{i\} \in \HH$ for all $i \in [n]$.
However, we ignore the singletons in the figures to simplify the drawings.
\end{remark}

Recall that the face of a Minkowski sum~$\sum_i \poly{P}_i$ minimizing a direction~$\b{c}$ is the Minkowski sum of the faces of the summands~$\poly{P}_i$ minimizing~$\b{c}$.
This standard fact immediately yields combinatorial descriptions of the vertices and faces of the hypergraphic polytope~$\triangle_\HH$, which we recall in the next two sections.
Our descriptions are similar to (but slightly different from) those of \cite[Thm.~2.18]{BenedettiBergeronMachacek}.


\subsection{Acyclic orientations and vertices of hypergraphic polytopes} 
\label{subsec:acyclicOrientations}

We first recall a combinatorial model for the vertices of~$\simplex_\HH$, which will be extended to all faces of~$\triangle_\HH$ later in \cref{subsec:acyclicPreorientations}.

\begin{definition}
\label{def:acyclicOrientation}
An \defn{orientation} of~$\HH$ is a map~$O: \HH \to [n]$ such that~$O(H) \in H$ for all~${H \in \HH}$.
The orientation~$O$ is \defn{acyclic} if there is no~$H_1, \dots, H_k$ with~$k \ge 2$ such that~$O(H_{i+1}) \in H_i \ssm \{O(H_i)\}$ for~$i \in [k-1]$ and~$O(H_1) \in H_k \ssm \{O(H_k)\}$.
\end{definition}

\begin{example}
\label{exm:DH2}
The hypergraph $\HH=\{ {\orange 123}, {\green 134} \}$ of~\cref{exm:DH1} has $9$ orientations, $7$ of which are acyclic as displayed in~\cref{exm:DH1} (we represent each orientation~$O$ by the pair~$({\orange O(123)}, {\green O(134)})$). 
For instance, the orientation $({\orange O(123)}, {\green O(134)}) = ({\orange 1}, {\green3})$ is cyclic since ${\orange O(123)} = {\orange 1} \in {\green 134}$ and ${\green O(134)} = {\green 3} \in {\orange 123}$ yields a cycle with $k = 2$.
\end{example}

\begin{proposition}[{\cite[Thm.~2.18]{BenedettiBergeronMachacek}}]
\label{prop:acyclicOrientations}\label{def:less_A}
The acyclic orientations of~$\HH$ are in bijection with the vertices of~$\simplex_\HH$.
More precisely, an acyclic orientation~$A$ of~$\HH$~corresponds~to
\begin{itemize}
\item the vertex $\sum_{H \in \HH} \b{e}_{A(H)}$ of~$\simplex_\HH$,
\item the vertex poset~$\mathdefn{\less_A}$ of~$\simplex_\HH$ defined as the transitive closure of~$\bigset{A(H) \le h}{h \in H \in \HH}$,
\item the maximal cone~$\mathdefn{\poly{C}_A} \eqdef \set{\b{x} \in \R^n}{x_{A(H)} \ge x_{h} \text{ for all } h \in H \in \HH}$ in the normal fan of~$\simplex_\HH$
\end{itemize}
\end{proposition}

%
%
%

\begin{definition}
\label{def:surjection}
For a permutation~$\sigma$ of~$[n]$, we define an orientation~$\mathdefn{\Or_\sigma}$ of~$\HH$ by
\[
\Or_\sigma(H) \eqdef \sigma\big(\min\set{j}{\sigma(j)\in H}\big).
\]
Equivalently, $\Or_\sigma(H)$ is the element of $H$ that first appears in the word $\sigma(1) \sigma(2) \ldots \sigma(n)$.
\end{definition}

\begin{proposition}[{\cite[Lem.~2.9]{BenedettiBergeronMachacek}}]
The map~$\Or$ is a surjection from the permutations of~$[n]$ to the acyclic orientations of~$\HH$.
The pre-image $\Or^{-1}(A) \eqdef \set{\sigma}{\Or_\sigma = A}$ is the set of linear extensions of~$\less_A$.
\end{proposition}

\begin{example}
\label{exm:DH3}
Continuing with the hypergraph $\HH=\{{\orange 123}, {\green 134} \}$ of \cref{exm:DH1,exm:DH2}, let $A$ be the acyclic orientation $({\orange 2}, {\green 4})$.
The order~$\less_A$ is the following
\[
	{\less_A} = \text{Transitive closure }\Bigg(
	\begin{tikzpicture}[scale=1,baseline=.2cm]
		\node (2) at (0,0) {$\scriptstyle 2$};
		\node (1) at (-.4,.7) {$\scriptstyle 1$};
		\node (3) at (.4,.7) {$\scriptstyle 3$};
		\draw [thick] (2)--(1); 
		\draw [thick] (2)--(3); 
	\end{tikzpicture} 
	\cup
	\begin{tikzpicture}[scale=1,baseline=.2cm]
		\node (4) at (0,0) {$\scriptstyle 4$};
		\node (1) at (-.4,.7) {$\scriptstyle 1$};
		\node (3) at (.4,.7) {$\scriptstyle 3$};
		\draw [thick] (4)--(1); 
		\draw [thick] (4)--(3); 
	\end{tikzpicture} 
	\Bigg)
	=
	\begin{tikzpicture}[scale=1,baseline=.2cm]
		\node (4) at (.4,0) {$\scriptstyle 4$};
		\node (2) at (-.4,0) {$\scriptstyle 2$};
		\node (1) at (-.4,.7) {$\scriptstyle 1$};
		\node (3) at (.4,.7) {$\scriptstyle 3$};
		\draw [thick] (4)--(1); 
		\draw [thick] (4)--(3); 
		\draw [thick] (2)--(1); 
		\draw [thick] (2)--(3); 
	\end{tikzpicture}.
\]
The linear extensions of this order are the permutations $2413$, $2431$, $4213$ and $4231$.
\end{example}


\subsection{Acyclic preorientations and faces of hypergraphic polytopes} 
\label{subsec:acyclicPreorientations}

We now recall a combinatorial model for all faces of~$\simplex_\HH$, similar to (but slightly different from) that of \cite[Thm.~2.18]{BenedettiBergeronMachacek}.

\begin{definition}
\label{def:acyclicPreorientation}
A \defn{preorientation} of~$\HH$ is a map~$O: \HH \to 2^{[n]}$ such that~$\varnothing \ne O(H) \subseteq H$ for all~${H \in \HH}$.
The preorientation~$O$ is \defn{acyclic} if there is no~$H_1, \dots, H_k$ with~$k \ge 2$ such that~${O(H_{i+1}) \cap H_i \ne \varnothing}$ for~$i \in [k-1]$ and~$O(H_1) \cap (H_k \ssm \{O(H_k)\}) \ne \varnothing$.
\end{definition} 

\begin{proposition}[{\cite[Thm.~2.18]{BenedettiBergeronMachacek}}]
\label{prop:acyclicPreorientations}
The acyclic preorientations of~$\HH$ are in bijection with the faces of~$\simplex_\HH$.
More precisely, an acyclic preorientation~$A$ of~$\HH$ corresponds to
\begin{itemize}
\item the face $\sum_{H \in \HH} \triangle_{A(H)}$ of~$\simplex_\HH$,
\item the face preposet~$\mathdefn{\bless_A}$ of~$\simplex_\HH$ defined as the transitive closure of~${\bigset{a \le h}{h \in H \in \HH, \, a \in A(H)}}$,
\item the cone~$\mathdefn{\poly{C}_A} \eqdef \set{\b{x} \in \R^n}{x_a \ge x_h \text{ for all } h \in H \in \HH, a \in A(H)}$ in the normal fan of~$\simplex_\HH$.
\end{itemize}
\end{proposition}


%
%

\begin{definition} \label{surj}
For an ordered partition (\aka set composition) $X = (X_1,\dots,X_k)$ of $[n]$, we define a preorientation \mathdefn{$\mathcal{O}_X$} of $\HH$ by
\[
\mathcal{O}_X(H) \eqdef H \cap X_{\min\set{i \in [k]}{H \cap X_i \neq \varnothing}}.
\]
\end{definition}

\begin{proposition}[{\cite[Lem.~2.9]{BenedettiBergeronMachacek}}]
The map~$\Or$ is a surjection from the ordered partitions of~$[n]$ to the acyclic preorientations of~$\HH$.
\end{proposition}

\begin{remark}
Note that an (acyclic) orientation of~$\HH$ is essentially a (acyclic) preorientation of~$\HH$ whose images are all singletons.
In this case, the orientation just records the element while the preorientation records the singleton.
It thus slightly simplifies the notation to work with orientations when focusing on vertices of~$\simplex_\HH$.
\end{remark}

\subsection{Three families of hypergraphic polytopes}
\label{subsec:exmHP}

We briefly recall three specific families of hypergraphic polytopes, which are illustrated in \cref{fig:exmHypergraphicPolytopes}.


\subsubsection{Graphical zonotopes}\label{sssec:GraphicalZono}

We start with graphical zonotopes, which are extensively studied in the literature, see \eg~\cite{Stanley-acyclicOrientations, Greene, GreeneZaslavsky, PostnikovReinerWilliams, Postnikov, Pilaud-acyclicReorientationLattices,PadrolPilaudPoullot-deformedGraphicalZonotopes,Poullot2025RaysDeformationConesGraphical,PadrolPoullot2025IndecomposabilityAndBeyond}.

\begin{definition}
A \defn{graphical zonotope} is the hypergraphic polytope~$\simplex_\GG$ of a graph~$\GG$, \ie a hypergraph such that $|H| = 2$ for all~$H \in \GG$.
\end{definition}

\begin{example}\,
\begin{itemize}
\item The graphical zonotope of a tree on $n$ nodes is a cube of dimension $n-1$.
\item The graphical zonotope of the complete graph on $n$ nodes is the permutahedron $\Perm$, see Figure~\ref{fig:exmHypergraphicPolytopes}.
\end{itemize}
\end{example}

Note that the normal fan of the graphical zonotope~$\simplex_\GG$ is the fan defined by the graphical arrangement of~$\GG$, that is, the collection of hyperplanes~$\set{\b{x} \in \R^n}{x_i = x_j}$ for all arcs~$\{i,j\}$ of~$\GG$.
The following description of the vertex posets of~$\simplex_\GG$ is classical, and traces back to the work of C.~Greene~\cite{Greene} (see also~\cite[Lem.~7.1]{GreeneZaslavsky}).

\begin{proposition}[\cite{Greene}, {\cite[Lem.~7.1]{GreeneZaslavsky}}]
The vertex posets of a graphical zonotope~$\simplex_\GG$ are the transitive closures of the acyclic orientations of~$\GG$.
Moreover, any poset is the vertex poset of some graphical zonotope.
\end{proposition}

Finally, simple graphical zonotopes were characterized in~\cite[Rem.~6.2]{Kim} (see also \cite[Prop.~5.2]{PostnikovReinerWilliams} and \cite[Prop.~53]{Pilaud-acyclicReorientationLattices}).
We say that a graph~$\GG$ is \defn{chordful} (\aka a \defn{block graph}, \aka \defn{clique tree}) when the following equivalent conditions are satisfied:
\begin{enumerate}[(i)]
\item any cycle in~$\GG$ induces a clique of~$\GG$,
\item the $2$-connected components of~$\GG$ are cliques,
\item $\GG$ is the line graph of some forest~$\FF$ (the line graph of~$\FF$ is the graph whose vertices are the arcs of~$\FF$, and where two arcs of~$\FF$ are connected if they are incident to a common node).
\end{enumerate}

\begin{proposition}[{\cite[Rem.~6.2]{Kim}}]
\label{prop:chordful}
The graphical zonotope~$\simplex_\GG$ is simple if and only if~$\GG$ is chordful.
\end{proposition}


\subsubsection{Nestohedra}

Nestohedra of building sets were defined in the work of A.~Postnikov~\cite{Postnikov} and E.-M.~Feichtner and B.~Sturmfels~\cite{FeichtnerSturmfels} in connection to the wonderful compactifications of subspace arrangements of C.~De~Concini and C.~Procesi~\cite{DeConciniProcesi}.

\begin{definition}
A hypergraph $\BB$ is a \defn{building set} if~$A \cap B \ne \varnothing$ implies~$A \cup B \in \BB$ for all~$A, B \in \BB$.
The hypergraphic polytope~$\simplex_\BB$ is the \defn{nestohedron} of~$\BB$.
\end{definition}

\begin{example} \,
\begin{itemize}
\item The collection $2^{[n]}$ of all subsets of~$[n]$ is a building set. The corresponding nestohedron is (normally equivalent to) the permutahedron $\Perm$, see Figure~\ref{fig:exmHypergraphicPolytopes}.
\item The collection of all intervals of~$[n]$ is a building set. The corresponding nestohedron is the usual (Loday) associahedron $\Asso$, see Figure~\ref{fig:exmHypergraphicPolytopes}.
\item The collection of \defn{tubes} of a graph~$\GG$ (\ie subsets of vertices inducing a connected subgraph) is a building set, called the \defn{graphical building set} of $\GG$. The corresponding nestohedron is called a \defn{graphical associahedron}~\cite{CarrDevadoss}.
\item More generally, the collection of tubes of a hypergraph~$\HH$ (\ie subsets of vertices inducing a connected subhypergraph) is a building set, whose nestohedron is the \defn{hypergaph associahedron} of~$\HH$ \cite{DosenPetric}.
\item A collection $\BB$ of nested sets (\ie if $A, B\in \BB$ satisfy $A\cap B\ne \varnothing$, then either $A\subseteq B$ or~$B\subseteq A$) is a building set. The corresponding nestohedron is called a \defn{fertilitope}~\cite{Defant-fertilitopes}. 
\end{itemize}
\end{example}

\enlargethispage{.1cm}
In a rooted tree $T$ with associated order relation $\less$, the \defn{descendant set} of a node $a \in T$ is~$\mathdefn{a^\less} \eqdef \set{b \in T}{a \less b}$, and two vertices $a,b \in T$ are \defn{incomparable} if neither $a\in b^\less$ nor $b\in a^\less$.
For a building set $\BB$, a directed tree is a \defn{$\BB$-tree} if for all $a\in T$, we have $a^\less\in \BB$, and for all incomparable vertices $a_1, \dots, a_r\in T$, we have $\bigcup_{i=1}^r a_i^\less \notin \BB$. 
A rooted forest $F$ is a \defn{$\BB$-forest} if the connected components of $F$ are connected components of $\BB$ (\ie inclusion maximal elements of~$\BB$), each inducing a \defn{$\BB$-tree}.
Especially, if $\BB$ is connected (\ie $[n]\in \BB$), \mbox{then all $\BB$-forests are $\BB$-trees.}

\pagebreak

\begin{proposition}[{\cite[Prop.~7.8]{Postnikov}}]\label{prop:NestoVertexDiGraphsAreBBforests}
The vertex digraphs of the nestohedron~$\simplex_\BB$ are precisely the $\BB$-forests.
In particular, all nestohedra are simple polytopes.
Moreover, any rooted forest is the vertex digraph of a nestohedron.
\end{proposition}

\begin{proof}
The first part of the statement is \cite[Prop.~7.8]{Postnikov}.
For the last sentence of the statement, fix a rooted forest $F$, and let $\BB = \set{a^\less}{a \in F}$.
Then $\BB$ is a collection of nested sets, because $a^\less\cap b^\less\ne \varnothing$ implies $a\in b^\less$ (hence $a^\less\subseteq b^\less$) or $b\in a^\less$ (hence $b^\less\subseteq a^\less$).
Hence, $\BB$ is a building set and $F$ is a $\BB$-forest.
Thus, by the first part of the statement, $F$ is a vertex digraph of the nestohedron $\simplex_{\BB}$.
\end{proof}

In \cite{PostnikovReinerWilliams}, the authors conjectured that nestohedra are recognizable (among the deformed permutahedra) by their vertex digraphs.

\begin{conjecture}[{\cite[Question~8.3]{PostnikovReinerWilliams}}]
\label{conj:nestohedralCharacterization}
A deformed permutahedron has the normal fan of a nestohedron if and only if all its vertex digraphs are rooted forest.
\end{conjecture}

We will prove this conjecture in an upcoming paper.
In \Cref{prop:rootedWeepingWillows1}, we prove a weaker version of this conjecture, which can be phrased in two different but equivalent ways.
Namely, we prove that \cref{conj:nestohedralCharacterization} holds when we replace
\begin{itemize}
\item either ``deformed permutahedron'' by ``deformed associahedron'' (or equivalently ``interval hypergraphic polytope'', see \cref{subsec:IHP}),
\item or ``rooted trees'' by ``rooted weeping willows'' (see \cref{sec:weepingWillows}).
\end{itemize}

We conclude with a relevant family of examples of nestohedra.

\begin{example}\label{ex:fertilitopes}
A \defn{fertilitope}~\cite{Defant-fertilitopes} is a nestohedron~$\simplex_\BB$ where all elements of~$\BB$ are either nested or disjoint.
To describe its vertices, it is convenient to assume that~$\BB$ contains all singletons.
For~$B \in \BB$, let $\mathdefn{\partial B}$ be the set of inclusion maximal elements of $\set{A \in \BB}{A \subsetneq B}$ if~$|B| > 1$, and $\partial B = \{B\}$ if $|B| = 1$.
The vertices of the fertilitope $\simplex_\BB$ are in bijection with the sequences~$(x_B)_{B \in \BB} \in [n]^\BB$ such that $x_B \in \set{x_A}{A \in \partial B}$ and~$x_{\{b\}} = b$ for~$b \in [n]$.
Hence, the fertilitope $\simplex_\BB$ has $\prod_{B \in \BB} |\partial B|$ vertices and is actually combinatorially equivalent to the product of simplices~$\prod_{B \in \BB} \simplex_{|\partial B|}$.
The $\BB$-tree associated to $(x_B)_{B \in \BB}$ is formed by the arcs from $x_B$ to $x_A$ for all $A \in \partial B$ with $x_A \ne x_B$.
The $i$-th coordinate of the vertex associated to $(x_B)_{B \in \BB}$ is the number of $B \in \BB$ satisfying $x_B = i$.
\end{example}


\subsubsection{Interval hypergraphic polytopes}
\label{subsec:IHP}

Following~\cite{BergeronPilaud}, we focus in this paper on the following family of hypergraphs on~$[n]$.

\begin{definition}
An \defn{interval hypergraph~$\II$} is a hypergraph on~$[n]$ where each~$I \in \II$ is an interval of the form $I = \mathdefn{[i,j]} \eqdef [j]\ssm[i-1] = \{i, i+1, i+2, \dots, j-1, j\}$.
\end{definition}

\begin{definition}
An \defn{interval building set} is an interval hypergraph which is also a building set.
\end{definition}


\begin{example}
Our running example~$\HH=\{ {\orange 123}, {\green 134} \}$  of \cref{exm:DH1,exm:DH2,exm:DH3}, is not an interval hypergraph, as ${\green 134}$ is not an interval.
\end{example}

\begin{example}
\label{exm:intervalHypergraph}
The hypergraph
$\II = \{ {\red 123}, {\orange 1234}, {\green 23}, {\blue 234} \}$ 
is an interval hypergraph admitting an acyclic orientation $\Or_{4132} = ({\red 1}, {\orange 4}, {\green 3}, {\blue 4})$, see \Cref{fig:IntervalHypergraphExample}.
\end{example}

\begin{figure}[h]
    \centering
    \includegraphics[scale=.4]{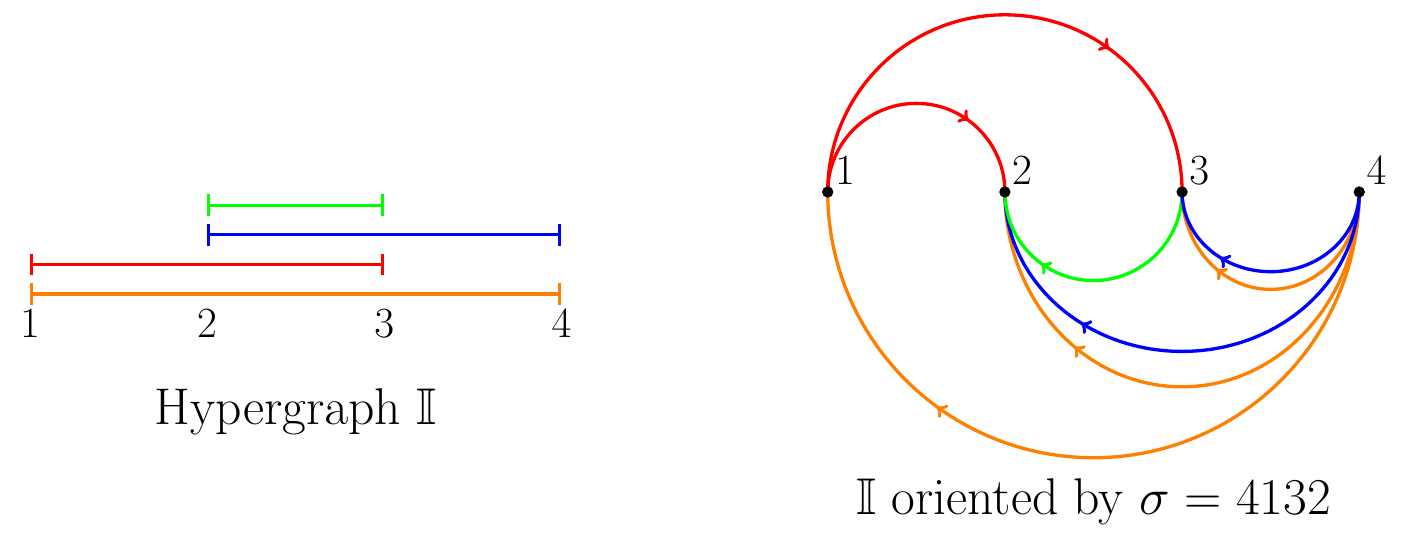}
    \caption{(Left) An interval hypergraph $\II$, (Right) and the graph formed by the directed edges $(\Or_{4132}(H),\, h)$ for each $H \in \II$ and each $h \in H \setminus \{\Or_{4132}(H)\}$.}
    \label{fig:IntervalHypergraphExample}
\end{figure}

We will prove in \cref{prop:TamariIntervalPosets1} some relevant characterizations of interval hypergraphic polytopes.
At the moment, we only need the following important observation from~\cite[Prop.~3.6]{BergeronPilaud}, and its generalization to preorientations.
They state that an orientation (resp.~preorientation) is acyclic if and only if it contains no cycle of length $2$.
\cref{prop:length2cycles} is proven in~\cite[Prop.~3.6]{BergeronPilaud} or follows from \cref{prop:prelength2cycles}.

\begin{proposition}[{\cite[Prop.~3.6]{BergeronPilaud}}]
\label{prop:length2cycles}
An orientation $O$ of an interval hypergraph~$\II$ is acyclic if and only if there is no~$I, J \in \II$ such that~$O(I) \in J \ssm \{O(J)\}$ and~$O(J) \in I \ssm \{O(I)\}$. 
\end{proposition}


%
%

\begin{proposition}
\label{prop:prelength2cycles}
A preorientation $O$ of an interval hypergraph~$\II$ is acyclic if and only if there is no~$I, J \in \II$ such that~$O(I) \cap J \ne \varnothing$ and~$O(J) \cap (I \ssm O(I)) \ne \varnothing$.
\end{proposition}

\begin{proof}
If there exist $I,J\in \II$ such that $O(I) \cap J \ne \varnothing$ and~$O(J) \cap (I \ssm O(I)) \ne \varnothing$, then $O$ is cyclic by \Cref{def:acyclicPreorientation}.
Conversely, suppose that~$O$ is cyclic, and let~$I_1,\dots,I_k \in \II$ be such that there are~$x_1 \in O(I_1) \cap (I_k \ssm \{O(I_k)\})$ and~$x_i \in O(I_i) \cap I_{i-1}$ for~$2 \le i \le k$, and that~$k \ge 3$.
We distinguish three cases, proving that there exists a smaller cycle (hence by recursion, a cycle of length $2$):
\begin{itemize}
\item If~$x_1 \le x_2 \le x_k$ (resp.~$x_k \le x_2 \le x_1$), then~$x_2 \in I_k$. If~$x_2 \in O(I_k)$, then~$I_1, I_k$ is a smaller cycle. If~$x_2 \notin O(I_k)$, then~$I_2, \dots, I_k$ is a smaller cycle.
\item If~$x_1 \le x_k \le x_2$ (resp.~$x_2 \le x_k \le x_1$), then~$x_k \in I_1$, so that~$I_1, I_k$ is a smaller cycle.
\item If~$x_2 \le x_1 \le x_k$ (resp.~$x_k \le x_1 \le x_2$), then there exists~$2 \le i < k$ such that~$x_i \le x_1 \le x_{i+1}$ (resp.~$x_{i+1} \le x_1 \le x_i$), hence~$x_1 \in I_i$. If~$x_1 \in O(I_i)$, then~$I_i, \dots, I_k$ is a smaller cycle. If~$x_1 \notin O(I_i)$, then~$I_1, \dots, I_i$ is a smaller cycle.
\qedhere
\end{itemize}
\end{proof}


\subsection{Two Genocchi families of interval hypergraphic polytopes}
\label{subsec:Genocchi}

We now consider two families of interval hypergraphs, which are both counted by the famous median Genocchi numbers (see \cref{rem:Genocchi}).
The following proposition was advertised in~\cite{Pilaud-MarioLuigi}, but we reproduce it here in a more formal environment.
In this statement, the size of a hypergraph (resp.~graph, resp.~permutation) is its number of hyperedges (resp.~arcs, resp.~simple transpositions in any reduced expression).

\begin{proposition}[\cite{Pilaud-MarioLuigi}]
\label{prop:plumbingBijections}
There are size preserving bijections between:
\begin{enumerate}[(i)]
\item the interval hypergraphs~$\II$ on~$[n]$ such that~$I,J \in \I$ and~$I \cap J \ne \varnothing$ implies~$I \cap J \in \II$,~\label{item:intersectionClosed}
\item the interval hypergraphs~$\II$ on~$[n]$ such that~$I,J \in \I$ and~$I \cap J \ne \varnothing$ implies~$I \cup J \in \II$,~\label{item:unionClosed}
\item the graphs~$([n+1], E)$ such that~$\{a,c\}, \{b,d\} \in E$ implies~$\{b,c\} \in E$ for all~${a < b < c < d}$,~\label{item:hookGraphs}
\item the graphs~$([n+1], E)$ such that~$\{a,c\}, \{b,d\} \in E$ implies~$\{a,d\} \in E$ for all~${a < b < c < d}$,~\label{item:terrainLikeGraphs}
\item the permutations~$\sigma$ of~$[2n]$ such that~$\sigma(i) \le 2i$ and~$\sigma(2n-i+1) \ge 2(n-i)+1$ for all~$i \in [n]$,~\label{item:YoshiPermutations}
\item the permutations~$\tau$ of~$[2n]$ such that~$\tau(2i-1) \ge 2i-1$ and~$\tau(2i) \le 2i$ for all~$i \in [n]$. \label{item:DumontPermutations}
\end{enumerate}
\end{proposition}

\begin{proof}
Mapping the interval~$[i,j]$ to the arc~$\{i,j+1\}$ clearly defines size preserving  bijections \mbox{\eqref{item:intersectionClosed} $\leftrightarrow$ \eqref{item:hookGraphs}} and \mbox{\eqref{item:unionClosed} $\leftrightarrow$ \eqref{item:terrainLikeGraphs}}.
Considering the permutation~$\rho$ of~$[2n]$ with~$\rho(2i-1) = n+i$ and~${\rho(2i) = i}$, the composition~$\tau \eqdef \sigma\rho$ defines a bijection \mbox{\eqref{item:YoshiPermutations} $\leftrightarrow$ \eqref{item:DumontPermutations}}.
Finally, the bijections \mbox{\eqref{item:intersectionClosed} $\leftrightarrow$ \eqref{item:YoshiPermutations} $\leftrightarrow$ \eqref{item:unionClosed}} are illustrated in \cref{fig:plumbingBijection}.
Namely, consider the word $$Q \eqdef s_n (s_{n-1} s_{n+1}) (s_{n-2} s_n s_{n+2}) (s_{n-3} s_{n-1} s_{n+1} s_{n+3}) \cdots (s_1 s_3 \cdots s_{2n-3} s_{2n-1})$$ on the simple transpositions~$s_i \eqdef (i \; i+1)$ of~$\fS_{2n}$.
Then any permutation~$\sigma$ of~$[2n]$ with~$\sigma(i) \le 2i$ and~$\sigma(2n-i+1) \ge 2(n-i)+1$ for all~$i \in [n]$ defines a non-empty subword complex~$SC(Q, \sigma)$ (see \cite{KnutsonMiller-GroebnerGeometry,KnutsonMiller-subwordComplex}) which admits unique greedy and anti-greedy facets (see~\cite{Pilaud-greedyFlipTree,PilaudStump-ELlabeling}).
The positions of the contacts in the greedy (resp.~anti-greedy facet) of~$SC(Q, \sigma)$ give an interval hypergraph satisfying~\eqref{item:intersectionClosed} (resp.~\eqref{item:unionClosed}).
\qedhere

\begin{figure}[h]
	\centerline{
		\begin{tabular}{c@{}c@{}c@{}c}
			\includegraphics[scale=1.2]{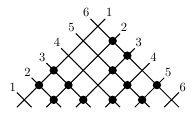} &
			\includegraphics[scale=1.2]{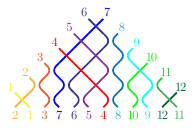} &
			\includegraphics[scale=1.2]{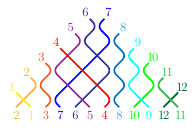} &
			\includegraphics[scale=1.2]{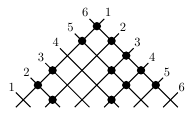} \\
			\eqref{item:intersectionClosed} $\leftrightarrow$ \eqref{item:hookGraphs} &
			greedy facet &
			anti-greedy facet &
			\eqref{item:unionClosed} $\leftrightarrow$ \eqref{item:terrainLikeGraphs} 
		\end{tabular}
	}
	\caption{Plumbing bijection.}
	\label{fig:plumbingBijection}
\end{figure}
\end{proof}

\begin{remark}
\enlargethispage{.1cm}
Note that in \cref{prop:plumbingBijections},
\begin{enumerate}[(i)]
\item the interval hypergraphs of~\eqref{item:intersectionClosed} are called \defn{intersection closed} and are precisely (up to the singletons) those whose hypergraphic poset (defined as the transitive closure of the graph of the hypergraphic polytope~$\simplex_\II$ oriented in the linear direction~$(1-n, 3-n, \dots, n-3, n-1)$) is a lattice~\cite[Thm.~A]{BergeronPilaud} (note that not all such hypergraphic polytopes are simple),
\item the interval hypergraphs of~\eqref{item:unionClosed} are called \defn{union closed} and are precisely (up to the singletons) the interval building sets in the sense of \cite{FeichtnerSturmfels, Postnikov}, and their hypergraphic polytopes are \defn{interval nestohedra} (in particular, they are simple polytopes),
\item the graphs of~\eqref{item:hookGraphs} are known as \defn{grounded rectangle graphs}~\cite{JelinekTopfer}, \defn{hook graphs}~\cite{Hixon}, \defn{max point-tolerance graphs}~\cite{CatanzaroChaplickFelsnerHalldorssonHalldorssonHixonStacho}, \defn{$p$-box graphs}~\cite{SotoCaro}, or \defn{non-jumping graphs}~\cite{AshurFiltserSababn},
\item the graphs of~\eqref{item:terrainLikeGraphs} are known as \defn{terrain-like graphs}~\cite{FroeseRenken-algo,FroeseRenken,AshurFiltserSababn},
\item the permutations of~\eqref{item:YoshiPermutations} are called \defn{Yoshi permutations} in~\cite{Pilaud-MarioLuigi},
\item the permutations of~\eqref{item:DumontPermutations} are called \defn{Dumont permutations}~\cite{DumontRandrianarivony}.
\end{enumerate}
Note that a similar bijection from \eqref{item:terrainLikeGraphs} to \eqref{item:DumontPermutations} was presented in~\cite{FroeseRenken}, but we are not aware that the bijection from \eqref{item:hookGraphs} to either \eqref{item:terrainLikeGraphs} or \eqref{item:DumontPermutations} was observed earlier, even if the two classes of graphs \eqref{item:hookGraphs} and \eqref{item:terrainLikeGraphs} were compared in~\cite{AshurFiltserSababn}.
\end{remark}

\begin{remark}
\label{rem:Genocchi}
The Dumont permutations of \cref{prop:plumbingBijections}\,\eqref{item:DumontPermutations} are counted by the famous median Genocchi numbers~\OEIS{A005439}, whose first values are
\[
1, 1, 2, 8, 56, 608, 9440, 198272, 5410688, 186043904, 7867739648, 401293838336, \dots
\]
They can be defined, for instance, using every second number~$g_{2n-1,1}$ in the leftmost column of the \defn{Seidel triangle}, which is the array of numbers~$g_{n,k}$ with~$n \ge 1$ and~$1 \le k \le \frac{n+1}{2}$, defined by~$g_{1,1} = 1$ and the induction
\[
g_{2n,k} = \sum_{i \ge k} g_{2n-1,i}
\qquad\text{and}\qquad
g_{2n+1,k} = \sum_{i \le k} g_{2n,i}.
\]
See \cref{table:Seidel}.
\begin{table}[H]
	\[
	\begin{array}{c|ccccccc}
	n \backslash k & 1 & 2 & 3 & 4 & 5 & 6 & \cdots \\
	\hline
	1 & 1 \\
	2 & 1 \\
	3 & 1 & 1 \\
	4 & 2 & 1 \\
	5 & 2 & 3 & 3 \\
	6 & 8 & 6 & 3 \\
	7 & 8 & 14 & 17 & 17 \\
	8 & 56 & 48 & 34 & 17 \\
	9 & 56 & 104 & 138 & 155 & 155 \\
	10 & 608 & 552 & 448 & 310 & 155 \\
	11 & 608 & 1160 & 1608 & 1918 & 2073 & 2073 \\
	\vdots & \vdots & \vdots & \vdots & \vdots & \vdots & \vdots & \ddots
	\end{array}
	\]
	\caption{Seidel's triangle~$g_{n,k}$.}
	\label{table:Seidel}
\end{table}
\end{remark}

\begin{remark}
The proof actually extends to arbitrary staircase polyomino~\cite[Sect.~3]{Pilaud-MarioLuigi}, but this goes beyond the focus of this paper.
\end{remark}


\section{Simple interval hypergraphic polytopes}
\label{sec:simpleIntervalHypergraphicPolytopes}

In this section, we prove the characterization of simple interval hypergraphic polytopes stated in \cref{thm:characterizationSimpleIntervalHypergraphicPolytopes}, and which we recall here.
We prove the necessity direction in \cref{subsec:necessity} and the sufficiency direction in \cref{subsec:sufficiency}.
The set of all interval hypergraphs on $4$ nodes whose hypergraphic polytope is simple is illustrated in \Cref{fig:LatticeOfSinI}.

\begin{theorem}
\label{thm:characterizationSimpleIntervalHypergraphicPolytopesProof}
The hypergraphic polytope $\simplex_\II$ of an interval hypergraph~$\II$ is simple if and only if the following two conditions hold for all $I, J \in \II$:
\begin{enumerate}[(i)]
\item if $I \cap J \ne \varnothing$ and there exists $K \in \II$ with $I \cup J \subseteq K$, then $I \cup J \in \II$, and\label{cond:simple1}
\item if $|I \cap J| \ge 2$, then $I \cup J \in \II$ or there exists $K \in \II$ with $I \cap J \subseteq K$ and $I \not\subseteq K$ and $J \not\subseteq K$. \label{cond:simple2}
\end{enumerate}
\end{theorem}

\begin{example}
We recover from \cref{thm:characterizationSimpleIntervalHypergraphicPolytopesProof} that the hypergraphic polytopes are all simple for:
	\begin{itemize}
	\item interval building sets (that is, interval hypergraphs closed under union of intersecting hyperedges): $I, J \in \II$ and~$I \cap J \ne \varnothing$ implies $I \cup J \in \II$, see~\cite{FeichtnerSturmfels,Postnikov},
	\item interval hypergraphs closed under subintervals: $I \in \II$ and~$J \subseteq I$ implies~$J \in \II$, see~\cite[Prop.~ 5.5]{CortesFrost2026DyckPathsConfigurationSpaces}.
	\end{itemize}
In particular the associahedron $\simplex_\II$ for $\II = \set{[i, j]}{1 \leq i < j \leq n}$ is simple for both reasons.
\end{example}

The following corollary emphasizes that interval nestohedra are special among simple interval hypergraphic polytopes.
See \Cref{cor:SimpleSaturatedHypergraphsContainingGroundSetAreBuildingSets} for a generalization to all hypergraphs.

\begin{corollary}
\label{coro:fullHyperedgePresent}
Among the interval hypergraphs containing the ground set~$[n]$ as a hyperedge, the only simple hypergraphic polytopes are the interval nestohedra.
\end{corollary}

\begin{proof}
Suppose $\simplex_\II$ is simple and $[n]\in \II$.
For $I, J\in \II$ with $I\cap J \ne \emptyset$, applying \Cref{thm:characterizationSimpleIntervalHypergraphicPolytopesProof}\,\eqref{cond:simple1} with $K = [n]$, we obtain that~$I\cup J\in\II$.
Hence, $\II$ is a building set.
\end{proof}

\begin{example}
\label{exm:cappedUnitIntervalHypergraphicPolytopeSimple}
For~$S \subseteq [n-1]$, the hypergraphic polytope~$\simplex_{\II_S}$ of~$\II_S \eqdef \set{[i,i+1]}{i \in S} \cup \{[n]\}$ is simple if and only if $n \le 3$ or $S$ does not contain two consecutive elements or equivalently no unit intervals intersect, see \Cref{exm:cappedUnitIntervalHypergraphicPolytope}.
\end{example}

\begin{figure}
	\includegraphics{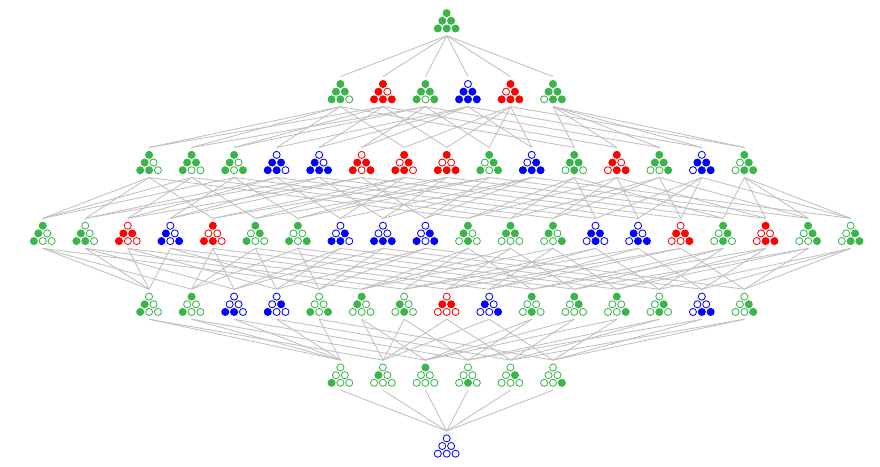}
	\caption{The boolean lattice $\c I$ of interval hypergraphs (containing all singletons) on $4$ nodes, with respect to inclusion. We encode a hypergraph by a subset of the triangle, where the presence of an interval~$[i,j]$ appears as a solid dot in diagonal~$i$ and antidiagonal~$j-1$. For instance, the interval hypergraph $\{[1,2], [2,4],[3,4]\}$ is encoded by \!\drawSubset{1,5,6}\!\!. The \textcolor{green}{green} and \textcolor{blue}{blue} hypergraphs define simple hypergraphic polytopes (\textcolor{green}{green} are the interval nestohedra), while the \textcolor{red}{red} do not.}
	\label{fig:LatticeOfSinI}
\end{figure}


\subsection{Necessity}
\label{subsec:necessity}

We first prove the forward direction of \cref{thm:characterizationSimpleIntervalHypergraphicPolytopesProof}.
We start with a standard observation that we will use repeatedly throughout the paper.
See also \cref{lem:cyclesPoset} for a discussion on the backward direction of this statement.

\begin{lemma}
\label{lem:posetTree}
If~$\less$ is the transitive closure of a directed tree, then for any~$i,j,k,\ell$,
\begin{enumerate}[(i)]
\item if~$k \less i$, $k \less j$, $i \less \ell$ and~$j \less \ell$, then~$i$ and~$j$ are comparable in~$\less$, \label{item:diamond}
\item if~$i \less k$, $i \less \ell$, $j \less k$, and~$j \less \ell$, then either~$i$ and~$j$, or~$k$ and~$\ell$ are comparable in~$\less$. \label{item:bowtie}
\end{enumerate}
\end{lemma}

We now decompose the proof of the forward direction of \cref{thm:characterizationSimpleIntervalHypergraphicPolytopesProof} into two lemmas, corresponding to the two conditions of \cref{thm:characterizationSimpleIntervalHypergraphicPolytopesProof}.
In the proof of each lemma, we assume that~$\II$ does not fulfill the condition, and we exhibit a non-simple vertex of~$\simplex_\II$.
For this, we define a permutation~$\sigma$ such that the vertex digraph of the acyclic orientation~$\Or_\sigma$ contains a cycle.

In the following figures, the  solid arrows are cover relations  $\lessdot$ and the dashed arrows are comparisons $\less$.

\begin{lemma}
\label{lem:forward1}
If an interval hypergraph~$\II$ contains~$I, J, K \in \II$ such that~$I \cap J \ne \varnothing$, $I \cup J \notin \II$ and $I \cup J \subsetneq K$, then the hypergraphic polytope~$\simplex_\II$ is not simple.
\end{lemma}

\begin{proof}
Assume by symmetry that~$i \eqdef \min(I) \in I \ssm J$ and~$j \eqdef \max(J) \in J \ssm I$, and fix arbitrary~${k \in K \ssm (I \cup J)}$ and $\ell \in I \cap J$.
Note that~$i, j, k, \ell$ exist by assumption, and are all distinct by definition.

\centerline{\includegraphics[scale=1]{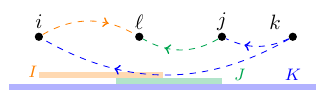}}

Let $X$ (resp.~$Y$) be the word formed by the complement of $I \cup J \cup \{k\}$ in~$[n]$ (resp.~of~$\{i, j,\ell\}$ in~$I \cup J$) written in an arbitrary order.
Consider the poset~$\less$ of the acyclic orientation~$\Or_\sigma$ defined by the permutation~$\sigma \eqdef k X ij\ell Y$.
As~$\{i,j,k\} \subseteq K \in \II$ and~$\Or_\sigma(K) = k$, we have~$k \less i$ and~$k \less j$.
Considering~$I$ and~$J$, we obtain similarly that~$i \less \ell$ and~$j \less \ell$.
Since~$\II$ contains only intervals, if~$\{i,j\} \subseteq H \in \II$, then~$I \cup J = [i, j] \subseteq H$.
Hence,~$I \cup J \subsetneq H$ because~$I \cup J \notin \II$, so $H$ contains one of the letters in the word $kX$, implying that~$\Or_\sigma(H) \notin \{i,j\}$, hence~$i$ and~$j$ do not form a cover relation of~$\less$.
As~$i$ and~$j$ are consecutive in~$\sigma$ and do not form a cover relation of~$\less$, we conclude that they are incomparable in~$\less$.
Since~$k \less i$, $k \less j$, $i \less \ell$ and $j \less \ell$, and~$i$ and~$j$ are incomparable in~$\less$, we obtain by \cref{lem:posetTree}\,\eqref{item:diamond} that the Hasse diagram of~$\less$ has a cycle.
We conclude that~$\simplex_\II$ is not simple.
\end{proof}

\begin{lemma}
\label{lem:forward2}
If an interval hypergraph~$\II$ contains~$I, J \in \II$ such that~$|I \cap J| \ge 2$, $I \cup J \notin \II$, and~$I \subseteq K$ or $J \subseteq K$ whenever~$K \in \II$ satisfies~$I \cap J \subseteq K$, then the hypergraphic polytope~$\simplex_\II$ is not simple.
\end{lemma}

\begin{proof}
As $I\cup J\notin \II$, by \cref{lem:forward1}, we can assume that there is no~$H \in \II$ with~$I \cup J \subseteq H$, otherwise $\simplex_\II$ is not simple.
Assume by symmetry that~$i \eqdef \min(I) \in I \ssm J$ and~$j \eqdef \max(J) \in J \ssm I$, and let~$k \eqdef \min(I \cap J)$ and~$\ell \eqdef \max(I \cap J)$.
Note that~$i, j, k, \ell$ exist and are all distinct by assumption.

\centerline{\includegraphics[scale=1]{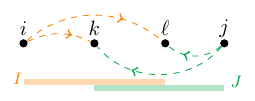}}

Let~$X$ be the word formed by the complement of $\{i,j,k,\ell\}$ in~$[n]$ written in an arbitrary order.
Consider the poset~$\less$ of the acyclic orientation~$\Or_\sigma$ defined by the permutation~$\sigma \eqdef ijk\ell X$.
Since~$\{i,k,\ell\} \subseteq I \in \II$ and~$\Or_\sigma(I) = i$, we have~$i \less k$ and~$i \less \ell$.
Considering~$J$, we obtain similarly that~$j \less k$ and~$j \less \ell$.
As we assume that there is no $H\in\II$ satisfying~$[i, j] = I \cup J \subseteq H$, and since $\II$ contains only intervals, there is no $H\in \II$ such that~$\{i,j\} \subseteq H$.
As~$i$ and~$j$ are consecutive in~$\sigma$, we conclude that they are incomparable in~$\less$.
Moreover, for any~$\{k,\ell\} \subseteq H \in \II$, we have~$I \cap J = [k, \ell] \subseteq H$, hence~$I \subseteq H$ or~$J \subseteq H$, implying that $\Or_\sigma(H) \notin\{k, \ell\}$.
As~$k$ and~$\ell$ are consecutive in~$\sigma$, we conclude that they are incomparable in~$\less$.
Since~$i \less k$, $i \less \ell$, $j \less k$, and~$j \less \ell$, and both~$\{i,j\}$ and~$\{k,\ell\}$ are incomparable pairs in~$\less$, we obtain by \cref{lem:posetTree}\,\eqref{item:bowtie} that the Hasse diagram of~$\less$ has a cycle.
We conclude that~$\simplex_\II$ is not simple.
\end{proof}

\begin{proof}[Proof of forward direction of \cref{thm:characterizationSimpleIntervalHypergraphicPolytopesProof}]
Follows from \cref{lem:forward1,lem:forward2}.
\end{proof}


\subsection{Sufficiency}
\label{subsec:sufficiency}

We now prove the backward direction of \cref{thm:characterizationSimpleIntervalHypergraphicPolytopesProof}.
We start with some observations on the vertex posets of interval hypergraphic polytopes.

\begin{lemma}
\label{lem:avoidingPatterns}
Consider an acyclic orientation of an interval hypergraph on~$[n]$, and denote by~$\less$ the associated poset on~$[n]$ from~\cref{def:less_A} and by~$\lessdot$ its cover relations.
For any~${1 \le a < b < c \le n}$,
\begin{enumerate}[(i)]
\item $a \lessdot b \moredot c$ (or $a \moredot b \lessdot c$) implies that~$a$ and~$c$ are incomparable (\ie $a\not\less c$ and $c\not\less a$), \label{item:ap1} 
\item $a \less c$ implies~$a \less b$ (and symmetrically $a \more c$ implies $b \more c$), \label{item:ap2}
\item $a \lessdot c$ or $a \moredot c$ implies that~$b \not\less c$ (and symmetrically $a \not\more b$). \label{item:ap3}
\end{enumerate}
\end{lemma}

\begin{proof}
Point~\eqref{item:ap1} obviously holds for any poset.
Indeed, if for instance~$a \less c$, then~$a \less b$ is not a cover relation as it is implied by~$a \less c$ and~$c \less b$.

Point~\eqref{item:ap2} already appeared as~\cite[Prop.~3.10]{BergeronPilaud}, we repeat the short proof for convenience.
Assume that~$a \less c$. By \cref{def:less_A}, there are~$I_1, \dots, I_k \in \II$ such that~$a = A(I_1)$, $A(I_{i+1}) \in I_{i}$ for all~$i \in [k-1]$, and~$c \in I_k$.
As $\bigcup_{i \in [k]} I_i$ is an interval containing~$a$ and~$c$ and~$a < b < c$, it also contains~$b$.
Hence, there is~$j \in [k]$ such that~$b \in I_j$, and the sequence~$I_1, \dots, I_j$ proves that~$a \less b$.

Finally, for Point~\eqref{item:ap3}, assume that $a$ and $c$ are comparable and that~$b \less c$. We distinguish:
\begin{itemize}
\item if~$a \less c$, then $a \less b$ by~\eqref{item:ap2}, so that~$a \less b \less c$ implies that~$a \less c$ is not a cover relation,
\item if~$a \more c$, then~$b \more c$ by~\eqref{item:ap2}, so that~$b \less c$ and~$b \more c$ contradict the acyclicity of the orientation.
\qedhere
\end{itemize}
\end{proof}

\begin{proof}[Proof of backward direction of \cref{thm:characterizationSimpleIntervalHypergraphicPolytopesProof}]
Assume that~$\simplex_\II$ is not simple.
Consider a non-simple vertex, and denote by~$A$ the corresponding acyclic orientation of~$\II$, by~$\less$ the associated poset on~$[n]$, and by~$\lessdot$ its cover relations.
As the vertex is not simple, $\lessdot$ contains an unoriented cycle~$\Gamma$, which contains a sink~$t \in [n]$.
Let~$x,y \in [n]$ be such that~$x \lessdot t \moredot y$ in the cycle~$\Gamma$.
By \cref{lem:avoidingPatterns}\,\eqref{item:ap3}, $x$ and~$y$ cannot both lie on the same side of~$t$, so we can assume by symmetry that~$1 \le x < t < y \le n$.
Moreover, as~$\Gamma$ is a cycle, it contains an arc~$\{u,v\}$ with~$1 \le u < t < v \le n$.
Assume by symmetry that this arc is oriented~$u \lessdot v$.
Note that~$u \less v$ implies $u \less t$ by \cref{lem:avoidingPatterns}\,\eqref{item:ap2}.
If~$x < u < t$, then $x \lessdot t$ and~$u \less t$ contradict \cref{lem:avoidingPatterns}\,\eqref{item:ap3}.
Hence, $u \le x < t$ and we distinguish two situations depending on the relative positions of~$v$ and~$y$.

Assume first that~$y \le v$, so that we have~$1 \le u \le x < t < y \le v \le n$.
As~$x \lessdot t$, $t \moredot y$, and~$u \lessdot v$, there are intervals $I, J, K \in \II$ such that~$[x,t] \subseteq I$ and~$A(I) = x$, $[t,y] \subseteq J$ and~$A(J) = y$, and~$[u,v] \subseteq K$ and $A(K) = u$.
If~$u \in I$, then we have~$u \more x$, which contradicts \cref{lem:avoidingPatterns}\,\eqref{item:ap3} since~$u \lessdot v$.
Similarly, we have~$v \notin I$, $u \notin J$, and~$v \notin J$.
We thus obtain that~$I \cup J \subset K$.
\centerline{\includegraphics[scale=1]{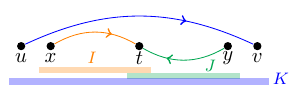}}

Assume now that~$I \cup J \in \II$.
If~$A(I \cup J) \in I$, then~$A(I \cup J) = x$ (as otherwise~$I$ and $I\cup J$ would contradict the acyclicity of~$A$).
Similarly, if~$A(I \cup J) \in J$, then~$A(I \cup J) = y$.
In both cases, we obtain that~$x$ and~$y$ are comparable for~$\less$, which contradicts \cref{lem:avoidingPatterns}\,\eqref{item:ap1} since~$x \lessdot t \moredot y$.
We conclude that~$I \cup J \notin \II$, so that our interval hypergraph~$\II$ does not satisfy Condition~\eqref{cond:simple1} in \cref{thm:characterizationSimpleIntervalHypergraphicPolytopesProof}.

Assume now that~$v < y$, so that we have~$1 \le u \le x < t < v < y \le n$.
As~$u \lessdot v$ and~$t \moredot y$, there are~$I,J \in \II$ such that~$[u,v] \subseteq I$ and~$A(I) = u$, and~$[t,y] \subseteq J$ and~$A(J) = y$.
We moreover pick these intervals~$I,J \in \II$ such that~$\min(I)$ is maximal while~$\max(J)$ is minimal among all possible intervals satisfying these properties.
Note that~$I \cap J$ contains at least~$t$ and~$v$.

\centerline{\includegraphics[scale=1]{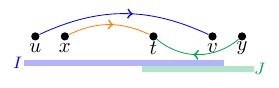}}

We now separate two cases:
\begin{itemize}
\item Assume first that~$I \cup J \in \II$. We again distinguish two cases:
	\begin{itemize}
	\item If~$A(I \cup J) \in I$, then~$A(I \cup J) = u$ by acyclicity of~$A$. We thus obtain that~$u \less y$. As~$A(J) = y$ and~$v \in [t,y] \subseteq J$, we also have~$v \more y$. We conclude that~$u \less y \less v$, contradicting that $u \lessdot v$ is a cover relation.
	\item If~$A(I \cup J) \in J$, then~$A(I \cup J) = y$ by acyclicity of~$A$. As~$x \in [u,y] \subseteq I \cup J$, this implies that~$x \more y$, which contradicts \cref{lem:avoidingPatterns}\,\eqref{item:ap1} since~$x \lessdot t \moredot y$.
	\end{itemize}

\centerline{\includegraphics[scale=1]{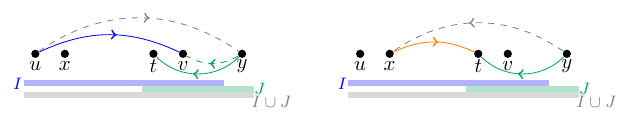}}
    
As~$A(I \cup J) \in I \cup J$, we obtain a contradiction, and conclude that~$I \cup J \notin \II$.
\item Assume now that there is~$K \in \II$ with~$I \cap J \subseteq K$ but~$I \not\subseteq K$ and~$J \not\subseteq K$. Set~$s = A(K)$.
If $s < u$, then $A(I) = u\in K$ and $A(K) = s\in I$ (otherwise $I\subseteq K$), which contradicts the acyclicity of $A$.
If $s = u$, then $[u, v] \subseteq K$ with $A(K) = u$, which contradicts the maximality of $\min(I)$.
Hence, we have $u < s$.
By a symmetric argument, we get that~$u < s < y$.
We again discuss two cases:
	\begin{itemize}
	\item If~$s < v$, we have~$s = A(K)$ and~$v \in I \cap J \subseteq K$, so that~$s \less v$, which contradicts \cref{lem:avoidingPatterns}\,\eqref{item:ap3} since~$u \lessdot v$ and~$u < s < v$.
	\item If~$t < s$, we have~$s = A(K)$ and~$t \in I \cap J \subseteq K$, so that~$s \less t$, which contradicts \cref{lem:avoidingPatterns}\,\eqref{item:ap3} since~$t \moredot y$ and~$t < s < y$.
	\end{itemize}

\centerline{\includegraphics[scale=1]{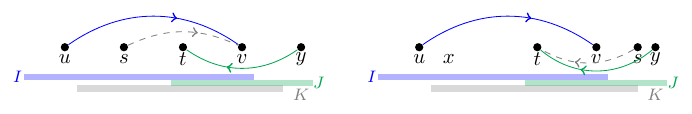}}

As~$t < v$, we conclude that there is no~$K \in \II$ with~$I \cap J \subseteq K$ but~$I \not\subseteq K$ and~$J \not\subseteq K$.
\end{itemize}
We thus obtained that~$|I \cap J| \ge 2$ but~$I \cup J \notin \II$ and there is no~$K \in \II$ with~$I \cap J \subseteq K$ but~$I \not\subseteq K$ and~$J \not\subseteq K$, so that our interval hypergraph~$\II$ does not satisfy Condition~\eqref{cond:simple2} in \cref{thm:characterizationSimpleIntervalHypergraphicPolytopesProof}.
\end{proof}


\subsection{Counting simple interval hypergraphic polytopes}
\label{subsec:countingSimpleIntervalHypergraphicPolytopes}

In this section, we discuss the problem of counting simple interval hypergraphic polytopes.
Remember that singleton hyperedges are irrelevant for hypergraphic polytopes (see \cref{rem:singletons}).
We thus ignore singletons in all our counting (taking them into account would only amount to multiply by a factor~$2^n$).

Denote by~$\mathdefn{I_n}$ (resp.~$\mathdefn{S_n}$, resp.~$\mathdefn{N_n}$, resp.~$\mathdefn{C_n}$) the number of interval hypergraphs (resp.~interval hypergraphs~$\II$ whose hypergraphic polytope~$\simplex_\II$ is simple, resp.~interval building sets, resp.~interval hypergraphs closed under subintervals) containing all singletons.
See \cref{table:relevantHypergraphicPolytopes} for the first few values of the sequences~$(I_n)_{n \ge 1}$, $(S_n)_{n \ge 1}$, $(N_n)_{n \ge 1}$ and $(C_n)_{n \ge 1}$.

\begin{table}[h]
	\centerline{
		\begin{tabular}{c|rrrrrrrrrr|c}
			$n$ & $1$ & $2$ & $3$ & $4$ & $5$ & $6$ & $7$ & $8$ & $9$ & $\cdots$ & OEIS \\
			\hline
			$I_n$ & $1$ & $2$ & $8$ & $64$ & $1\,024$ & $32\,768$ & $2\,097\,152$ & $268\,435\,456$ & $68\,719\,476\,736$ & $\cdots$ & \href{https://oeis.org/A006125}{A006125} \\
			$S_n$ & $1$ & $2$ & $8$ & $53$ & $510$ & $6\,771$ & $121\,036$ & $\dots$ &&& --- \\
			$N_n$ & $1$ & $2$ & $7$ & $38$ & $295$ & $3\,098$ & $42\,271$ & $726\,734$ & $15\,366\,679$ & $\cdots$ & \href{https://oeis.org/A000366}{A000366} \\
			$C_n$ & $1$ & $2$ & $5$ & $14$ & $42$ & $132$ & $429$ & $1\,430$ & $4\,862$ & $\cdots$ & \href{https://oeis.org/A000108}{A000108}
		\end{tabular}
	}
	\caption{The numbers~$I_n = 2^{\binom{n}{2}}$ of interval hypergraphic polytopes, $S_n$ of simple interval hypergaphic polytopes, $N_n$ of interval nestohedra, and $C_n$ of interval hypergraphs closed under subintervals on~$[n]$. Note that~$C_n < N_n < S_n < I_n$, for all $n \ge 4$.}
	\label{table:relevantHypergraphicPolytopes}
\end{table}

As any nestohedron is simple, we clearly have 
\[
\frac{1}{16\sqrt{\pi}} \left(\frac{2}{e^2\pi^2}\right)^n n^{2n+3/2} \sim N_n \le S_n \le I_n = 2^{\binom{n}{2}} \sim 2^{n^2/2}.
\]
The asymptotic expression for~$N_n$ is obtained from the standard expression for the median Genocchi numbers~\OEIS{A005439}, divided by~$2^n$ (since we ignore singletons as discussed earlier).

Finally, we have also reported in \cref{table:relevantConnectedHypergraphicPolytopes} the number~$\mathdefn{I^\star_n}$ (resp.~$\mathdefn{S^\star_n}$, resp.~$\mathdefn{N^\star_n}$, resp.~$\mathdefn{C^\star_n}$) of connected interval hypergraphs (resp.~connected interval hypergraphs~$\II$ whose hypergraphic polytope~$\simplex_\II$ is simple, resp.~connected interval building sets, resp~connected interval hypergraphs closed under subintervals) containing all singletons.
The numbers of \cref{table:relevantHypergraphicPolytopes,table:relevantConnectedHypergraphicPolytopes} are related by $\left(\sum_n I_n \, x^n\right) \left( 1 - \sum_n I_n^\star \, x^n \right) = 1$ (and similarly for~$S_n$, $N_n$, and~$C_n$).

\begin{table}[h]
	\centerline{
		\begin{tabular}{c|rrrrrrrrrr|c}
			$n$ & $1$ & $2$ & $3$ & $4$ & $5$ & $6$ & $7$ & $8$ & $9$ & $\cdots$ & OEIS \\
			\hline
			$I^\star_n$ & $1$ & $1$ & $5$ & $49$ & $893$ & $30\,649$ & $2\,030\,213$ & $264\,198\,625$ & $68\,180\,168\,717$ & $\cdots$ & \href{https://oeis.org/A348901}{A348901} \\
			$S^\star_n$ & $1$ & $1$ & $5$ & $38$ & $401$ & $5\,691$ & $106\,693$ & $\dots$ &&& --- \\
			$N^\star_n$ & $1$ & $1$ & $4$ & $25$ & $217$ & $2\,470$ & $35\,647$ & $637\,129$ & $13\,843\,948$ & $\cdots$ & \href{https://oeis.org/A218826}{A218826} \\
			$C^\star_n$ & $1$ & $1$ & $2$ & $5$ & $14$ & $42$ & $132$ & $429$ & $1\,430$ & $\cdots$ & \href{https://oeis.org/A000108}{A000108}
		\end{tabular}
	}
	\caption{The numbers~$I^\star_n$ of connected interval hypergraphic polytopes, $S^\star_n$ of connected simple interval hypergaphic polytopes, $N^\star_n$ of connected interval building sets, and $C^\star_n$ of connected interval hypergraphs closed under subintervals on~$[n]$.  Note that~$C^\star_n < N^\star_n < S^\star_n < I^\star_n$, for all $n \ge 4$.}
	\label{table:relevantConnectedHypergraphicPolytopes}
\end{table}


\subsection{Characterizing simple hypergraphic polytopes?}
\label{subsec:towardsSimpleHypergraphicPolytopes}

In this section, we discuss the extension of \cref{thm:characterizationSimpleIntervalHypergraphicPolytopes} to all hypergraphic polytopes, beyond those only consisting of intervals.
First, it is important to observe that, contrary to the case of interval hypergraphs (where the Minkowski sum $\sum_{} \lambda_{ij}\simplex_{[i,j]}$ uniquely determines the summands), the hypergraph~$\HH$ is in general not determined by the normal fan of the hypergraphic polytope~$\simplex_\HH$.

\begin{lemma}
\label{lem:completingHypergraphIntoSaturated}
If~$J \notin \HH$ and for every~$e \in \binom{J}{2}$ there exists~$H \in \HH$ such that~$e \subseteq H \subseteq J$, then~$\simplex_\HH$ and~$\simplex_{\HH \cup \{J\}}$ have the same normal fan. 
\end{lemma}

\begin{proof}
Remember that 
\begin{itemize}
\item two fans coincide if and only if they have the same walls (\ie codimension~$1$ cones), 
\item the union of the walls of the normal fan of a Minkowski sum is the union of the walls of the normal fans of the summands,
\item the normal fan of a simplex~$\simplex_J$ with~$J \subseteq [n]$ has a wall for each~$\{u,v\} \in \binom{J}{2}$, given by~$\set{\b x \in \R^n}{x_u = x_v < x_j \text{ for all } j \in J \ssm \{u,v\}}$.
\end{itemize}
Consequently, the hypergraphic polytopes~$\simplex_\HH$ and~$\simplex_{\HH \cup \{J\}}$ have the same normal fan if and only if the walls of the normal fan of the simplex~$\simplex_J$ are all contained in the union of the walls of~$\simplex_\HH$, that is, if and only if, for each~$e \in \binom{J}{2}$, there is~$H \in \HH$ such that~$e \subseteq H \subseteq J$.
\end{proof}

\begin{example}
The hypergraphic polytopes of the hypergraphs~$\{12, 23, 13\}$ and $\{12, 23, 13, 123\}$ have the same normal fan.
More generally, the braid fan is the normal fan of the hypergraphic polytope of any hypergraph containing the complete graph~$\binom{[n]}{2}$, for example of the complete hypergraph~$\sum_{X \subseteq [n]} \simplex_X$.
\end{example}

We say that a hypergraph~$\HH$ is \defn{saturated} if there is no~$J$ as in \cref{lem:completingHypergraphIntoSaturated}, that is, if~$\HH$ is inclusion maximal among all hypergraphs whose hypergraphic polytope have a given normal fan.
(Note that this differs from the notion of saturated hypergraph of~\cite[Sect.~4]{DosenPetric}, which just means building set.)
For instance, all interval hypergraphs are saturated.

To extend \cref{thm:characterizationSimpleIntervalHypergraphicPolytopes}, it is convenient to restrict our attention to saturated hypergraphs, as this entails no loss of generality.
First, \cref{lem:forward1,lem:forward2} extend straightforward to all saturated hypergraphs (we give similar proofs), which yields the following necessary conditions on~$\HH$ for~$\simplex_\HH$ to be simple.

\begin{lemma}
\label{lem:forward1_general}
If a saturated hypergraph~$\HH$ contains~$I, J, K \in \HH$ such that~$I \cap J \ne \varnothing$, $I \cup J \notin \HH$, and~$I \cup J \subsetneq K$, then the hypergraphic polytope~$\simplex_\HH$ is not simple.
\end{lemma}

\begin{proof}
As~$\HH$ is saturated and~$I \cup J \notin \HH$, there exist~$i,j \in I \cup J$ such that there is no~$H \in \HH$ with~$\{i,j\} \subseteq H \subseteq I \cup J$.
Fix moreover~$k \in K \ssm (I\cup J)$ and~$\ell \in I \cap J$.
Note that~$i, j \in (I \cup J) \ssm (I \cap J)$, so that~$i, j, k, \ell$ are all distinct.
Consider the poset~$\less$ of the acyclic orientation~$\Or_{\sigma}$ defined by a permutation~$\sigma \eqdef kXij\ell Y$ where~$X$ (resp.~$Y$) is the word formed by the complement of~$I \cup J \cup \{k\}$ in~$[n]$ (resp.~of~$\{i, j, \ell\}$ in~$I \cup J$) written in an arbitrary order.
The choice of~$i,j,k,\ell$ ensures that $k \less i \less \ell$ and $k \less j \less \ell$.
Moreover, as $i$ and $j$ are consecutive in $\sigma$ and there is no~$H \in \HH$ with~$\{i,j\} \subseteq H \subseteq I\cup J = \{i, j, \ell\} \cup Y$,  we obtain that $i$ and $j$ are incomparable in~$\less$.
We conclude by \cref{lem:posetTree}~\eqref{item:diamond} that the Hasse diagram of~$\less$ has a cycle, hence that~$\simplex_\II$ is not simple.
\end{proof}

\begin{lemma}
\label{lem:forward2_general}
If a saturated hypergraph~$\HH$ contains~$I, J \in \HH$ such that~$|I \cap J| \ge 2$, $I \cup J \notin \HH$, and there exists $k, \ell\in I\cap J$ such that there is no $H\in\HH\ssm\{I, J\}$ satisfying $\{k,\ell\}\subseteq H\subseteq I\cup J$, then the hypergraphic polytope~$\simplex_\HH$ is not simple.
\end{lemma}

\begin{proof}
As~$\HH$ is saturated and~$I \cup J \notin \HH$, there exist~$i,j \in I \cup J$ such that there is no~$H \in \HH$ with~$\{i,j\} \subseteq H \subseteq I \cup J$.
Note that~$i, j \in (I \cup J) \ssm (I \cap J)$, so that~$i, j, k, \ell$ are all distinct.
Consider the poset~$\less$ of the acyclic orientation~$\Or_{\sigma}$ defined by a permutation~$\sigma \eqdef Xijk\ell Y$ where~$X$ (resp.~$Y$) is the word formed by the complement of $I \cup J$ in~$[n]$ (resp.~of~$\{i, j, k, \ell\}$ in~$I\cup J$) written in an arbitrary order.
The choice of~$i,j,k,\ell$ ensures that $i \less k$, $i \less \ell$, $j \less k$ and~$j \less \ell$.
Moreover, as~$i$ and $j$ are consecutive  in $\sigma$ and there is no~$H \in \HH$ with~$\{i,j\} \subseteq H \subseteq I \cup J = \{i, j, k, \ell\} \cup Y$, we obtain that $i$ and $j$ are incomparable in~$\less$.
Similarly, as~$k$ and~$\ell$ are consecutive  in $\sigma$ and there is no~$H \in \HH$ with~$\{k,\ell\} \subseteq H \subseteq (I \cup J) \ssm \{i,j\} = \{k, \ell\} \cup Y$, we obtain that $k$ and $\ell$ are incomparable in~$\less$.
We conclude by \cref{lem:posetTree}\,\eqref{item:crown} that the Hasse diagram of~$\less$ has a cycle, hence that~$\simplex_\II$ is not simple.
\end{proof}

\begin{example}
The hypergraphic polytope of the saturated hypergraph $\{123, 234, 145, 12345\}$ is not simple by \cref{lem:forward1_general} with $I = 123$, $J = 234$, $K = 12345$.
The same hypergraphic polytope is not simple by~\cref{lem:forward2_general} with $I = 123$ and $J = 234$.
\end{example}

The following consequence of \cref{lem:forward1_general} generalizes~\cref{coro:fullHyperedgePresent}.
Roughly speaking, among the hypergraphic polytopes containing the full simplex~$\simplex_{[n]}$ as a summand, the simple hypergraphic polytopes are precisely the nestohedra.

\begin{corollary}\label{cor:SimpleSaturatedHypergraphsContainingGroundSetAreBuildingSets}
Consider a saturated hypergraph~$\HH$ containing the ground set~$[n]$ as a hyperedge. Then the hypergraphic polytope~$\simplex_\HH$ is simple if and only if~$\HH$ is a building set.
\end{corollary}

\begin{proof}
Suppose $\simplex_\HH$ is simple and $[n]\in \HH$.
For $A, B\in \HH$ with $A\cap B \ne \emptyset$, applying \Cref{lem:forward1_general} with $K = [n]$, we obtain that~$A\cup B\in\HH$.
Hence, $\HH$ is a building set.
\end{proof}

To prove \Cref{lem:forward1_general,lem:forward2_general}, we heavily relied on \Cref{lem:posetTree} which describes two kinds of patterns that prevent a Hasse diagram from being a directed forest.
However, more general obstructions can arise, demanding for more convoluted criteria on $\HH$ characterizing when is $\simplex_\HH$ a simple polytope.
The following classical strengthening of \cref{lem:posetTree} characterizes the posets whose Hasse diagrams are forests.


\begin{lemma}
\label{lem:cyclesPoset}
The Hasse diagram $D$ of the order relation $\less$ is a directed forest if and only if:
\begin{enumerate}[(i)]
\item there exist no $i,j,k,\ell$ with $i$ and $j$ incomparable by $\less$ such that~$k \less i$, $k \less j$, $i \less \ell$ and~$j \less \ell$,\label{item:diamond2}
\item there exist no $k_1, \dots, k_r,\ell_1,\dots, \ell_r$ for $r\geq 2$ with $k_p$ and $k_q$ (resp. $\ell_p$ and $\ell_q$) incomparable by $\less$ for $p\ne q$ such that~$k_p \less \ell_p$, $k_p \less \ell_{p+1}$ and~$k_r \less \ell_1$. \label{item:crown}
\end{enumerate}
See \cref{fig:crownPosets} for illustrations of these posets.
\begin{figure}
	\begin{tikzpicture}[scale=1,baseline=.2cm]
		\begin{scope}[shift={(0,.5)}]
			\node (1) at (0,-.7) {$k$};
			\node (2) at (-.4,0) {$i$};
			\node (3) at (.4,0) {$j$};
			\node (4) at (0,.7) {$\ell$};
			\draw [thick] (1)--(2); 
			\draw [thick] (1)--(3); 
			\draw [thick] (2)--(4); 
			\draw [thick] (3)--(4); 
		\end{scope}
		
		\begin{scope}[shift={(2,0)}]
			\node (1) at (-.4,0) {$k_1$};
			\node (2) at (.4,0) {$k_2$};
			\node (3) at (-.4,1) {$\ell_1$};
			\node (4) at (.4,1) {$\ell_2$};
			\draw [thick] (1)--(3); 
			\draw [thick] (1)--(4); 
			\draw [thick] (2)--(3); 
			\draw [thick] (2)--(4); 
		\end{scope}
		
		\begin{scope}[shift={(4.5,0)}]
			\node (1) at (-.8,0) {$k_1$};
			\node (2) at (0,0) {$k_2$};
			\node (3) at (.8,0) {$k_3$};
			\node (4) at (-.8,1) {$\ell_1$};
			\node (5) at (0,1) {$\ell_2$};
			\node (6) at (.8,1) {$\ell_3$};
			\draw [thick] (1)--(4); 
			\draw [thick] (1)--(5); 
			\draw [thick] (2)--(5); 
			\draw [thick] (2)--(6); 
			\draw [thick] (3)--(6); 
			\draw [thick] (3)--(4); 
		\end{scope}
		
		\begin{scope}[shift={(7.5,0)}]
			\node (1) at (-1.2,0) {$k_1$};
			\node (2) at (-.4,0) {$k_2$};
			\node (3) at (.4,0) {$k_3$};
			\node (4) at (1.2,0) {$k_4$};
			\node (5) at (-1.2,1) {$\ell_1$};
			\node (6) at (-.4,1) {$\ell_2$};
			\node (7) at (.4,1) {$\ell_3$};
			\node (8) at (1.2,1) {$\ell_4$};
			\draw [thick] (1)--(5); 
			\draw [thick] (1)--(6); 
			\draw [thick] (2)--(6); 
			\draw [thick] (2)--(7); 
			\draw [thick] (3)--(7); 
			\draw [thick] (3)--(8); 
			\draw [thick] (4)--(8); 
			\draw [thick] (4)--(5); 
		\end{scope}
		
		\begin{scope}[shift={(10,.5)}]
			\node {$\dots$};
		\end{scope}
	\end{tikzpicture}
	\caption{Diamond and crown posets.}
	\label{fig:crownPosets}
\end{figure}
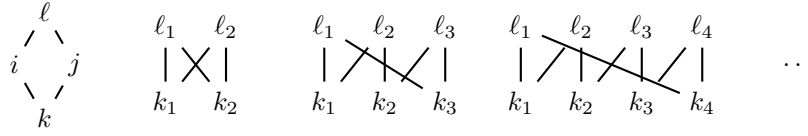
\end{lemma}

\begin{proof}
The proof is classical from the literature, we give here a very short idea.
Fix a cycle in~$D$ and consider $k_1, \dots, k_r$ to be its sources and $\ell_1, \dots, \ell_r$ its sinks, enumerated in cyclic order.
If~$r \geq 2$, then one is in the case of~\eqref{item:crown}.
If $r = 1$, then there are two directed paths from $k_1$ to $\ell_1$:
taking some $i$ on the first path and some $j$ on the second yields case~\eqref{item:diamond2}.
\end{proof}

\begin{example}
\label{exm:longCrown}
Let $C_n$ be a cycle on $n\geq 4$ vertices.
The graphical zonotope~$\simplex_{C_n}$ is not simple, even if it does not violate the necessary conditions of \cref{lem:forward1_general,lem:forward2_general}.
\end{example}

\pagebreak

However, searching in general for patterns in~$\HH$ which yield vertex digraphs of~$\simplex_\HH$ containing the cycles of \cref{lem:cyclesPoset} seems combinatorially intricate and computationally inefficient.
We thus leave the following problem open.

\begin{problem}
What is the complexity of the decision problem ``Given a hypergraph~$\HH \subseteq \binom{2^{[n]}}{m}$, is the hypergraphic polytope~$\simplex_\HH$ simple?''.
\end{problem}

\begin{remark}
Note that the problem ``Does~$\simplex_\HH$ have a non-simple vertex?'' is \textbf{NP} as one can provide as certificate a vertex poset of~$\simplex_\HH$ (given as an acyclic orientation of~$\HH$), and one of the patterns of \cref{lem:cyclesPoset}.
\end{remark}

\begin{remark}
We have seen three classes of hypergraphs $\HH$ for which determining whether $\simplex_\HH$ is simple can be done in polynomial time:
\begin{itemize}
\item for a building set $\BB$, the nestohedron $\simplex_\BB$ is always simple,
\item for a graph $\GG$, the graphical zonotope $\simplex_\GG$ is simple if and only if any cycle of $\GG$ induces a clique (this can be checked in $O(n+m)$ time),
\item for interval hypergraphs $\II$, the conditions of \Cref{thm:characterizationSimpleIntervalHypergraphicPolytopes} can be checked in $O(m^4)$ time.
\end{itemize}
Besides, note that verifying that $\HH$ is a building set (resp.~a graph, resp.~an interval hypergraph) can be done in $O(m^3)$ (resp.~$O(m)$, resp.~$O(n m)$) time.
These families however do not cover all simple hypergraphic polytopes.
\end{remark}

\section{Tamari interval posets and vertices of interval hypergraphic polytopes}
\label{sec:TamariIntervalPosets}

In this section, we consider the vertex posets of the interval hypergraphic polytopes.
These posets are precisely the Tamari interval posets defined by G.~Chatel and V.~Pons~\cite{ChatelPons}.


\subsection{Tamari interval posets}
\label{subsec:TamariIntervalPosets}

\enlargethispage{.3cm}
Recall that the \defn{weak order} is the lattice on permutations of~$\fS_n$ whose cover relations are pairs of permutations related by the swap of two values at adjacent positions.
Equivalently $\sigma \le \tau$ if and only if~$\inv(\sigma) \subseteq \inv(\tau)$ where~$\inv(\sigma)$ is the \defn{inversion set} of the permutation~$\sigma$, defined by~$\inv(\sigma) \eqdef \set{(\sigma(i), \sigma(j))}{1 \le i < j \le n \text{ and } \sigma(i) > \sigma(j)}$.
The Hasse diagram of the weak order is the graph of the permutahedron oriented in the direction~$\b{\omega} \eqdef (n, \dots, 1) - (1, \dots, n) = (n-1, n-3, \dots, 3-n, 1-n)$.
The \defn{Tamari lattice}~\cite{Tamari} is the lattice on binary trees on~$n$ nodes whose cover relations are pairs of binary trees related by a right rotation.
The Hasse diagram of the Tamari lattice is the graph of the associahedron oriented in the direction~$\b{\omega}$.
The following are classical characterizations of the intervals in the weak order and in the Tamari lattice.

\begin{definition}[{\cite[Thm.~6.8]{BjornerWachs}, \cite[Prop.~2.5]{ChatelPilaudPons}}]
\label{def:weakOrderIntervalPoset}
A \defn{weak order interval poset} is a poset~$\less$ on~$[n]$ with the following equivalent properties:
\begin{enumerate}[(i)]
\item the set of linear extensions of~$\less$ forms an interval~$[\sigma, \tau]$ of the weak order,
\item for~$1 \le a < b < c \le n$, one has ${(a \less c \implies a \less b \text{ or } b \less c)}$ and~${(a \more c \implies a \more b \text{ or } b \more c)}$.
\end{enumerate}
\end{definition}

\begin{definition}[{\cite[Def.~2.7 \& Thm.~2.8]{ChatelPons}, \cite[Coro.~2.24]{ChatelPilaudPons}}]
\label{def:TamariIntervalPoset}
A \defn{Tamari interval poset} is a poset~$\less$ on~$[n]$ with the following equivalent properties:
\begin{enumerate}[(i)]
\item the set of linear extensions of~$\less$ is the union of the sets of linear extensions of the binary trees in an interval of the Tamari lattice,\label{item:TIPprop1}
\item the set of linear extensions of~$\less$ is an interval~$[\sigma, \tau]$ of the weak order such that~$\sigma$ avoids the pattern $231$ and $\tau$ avoids the pattern~$213$,\label{item:TIPprop2}
\item for~$1 \le a < b < c \le n$, one has ${(a \less c \implies a \less b)}$ and~${(a \more c \implies b \more c)}$,\label{item:TIPprop3}
\item for any~$a \in [n]$, the principal upper set~$\mathdefn{a^\less} \eqdef \set{b \in [n]}{a \less b}$ of~$a$ is an interval of~$[n]$. \label{item:i_less_m_M}\label{item:TIPprop4}
\end{enumerate}
\end{definition}

\pagebreak

\begin{remark}
\label{rem:drawingTIP}
Following~\cite{ChatelPilaudPons}, we represent a Tamari interval poset~$\less$ on~$[n]$ by its Hasse diagram, with node~$a$ at coordinates~$(a,0)$ and with increasing (resp.~decreasing) cover relations~${a \lessdot b}$ (resp.~$a \moredot b$) with~$a < b$ drawn above (resp.~below) the horizontal axis.
Note that it yields a non-crossing graph (as for any~$a < b < c < d$, if~$a \less c$ and~$b \less d$, then~$a \less b$ and~$b \less c$ so that~$a \less c$ is not a cover relation).
A typical example of Tamari interval poset is illustrated~in~\cref{fig:TamariIntervalPoset}.
\begin{figure}[h]
	\centerline{\includegraphics[scale=1]{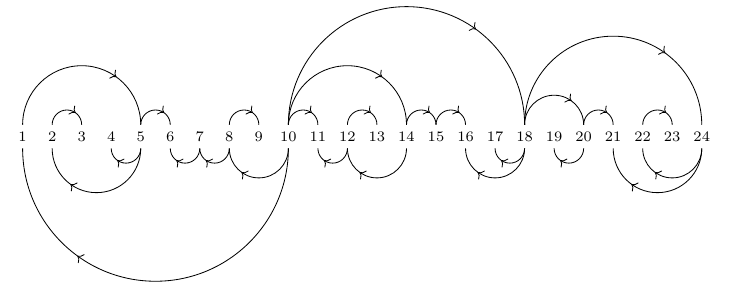}}
	\caption{The Hasse diagram of a typical Tamari interval poset.}
	\label{fig:TamariIntervalPoset}
\end{figure}
\end{remark}


\subsection{Two families of Tamari interval posets}
\label{subsec:examplesTamariIntervalPosets}

We describe two interesting families of Tamari interval posets, which will later appear as the vertex posets of certain families of hypergraphic polytopes (see also \cref{subsec:examplesWeepingWillows} for other families where all Hasse diagrams are trees).


\subsubsection{Capped unit interval posets}
\label{subsubsec:cappedUnitIntervalPosets}

We begin by considering the following posets, whose full significance will become clear in \cref{exm:cappedUnitIntervalHypergraphicPolytope} (see also \cref{subsubsec:cappedSUnitIntervalPosets,exm:cappedUnitIntervalHypergraphicPolytopePreposets}).

\begin{definition}
\label{def:cappedUnitIntervalPosets}
A \defn{capped unit interval poset} is a poset~$\less$ on~$[n]$ such that
\begin{itemize}
\item $\less$ admits a unique minimum~$m$, 
\item for any cover relation~$a \lessdot b$, either~$a = m$ or~$|a-b| = 1$.
\end{itemize}
\end{definition}

\begin{example}
\label{exm:capped_unit_n=1,2,3,4}
For instance, the capped unit interval posets for~$n = 4$ are the following

\medskip
\centerline{\includegraphics[scale=1]{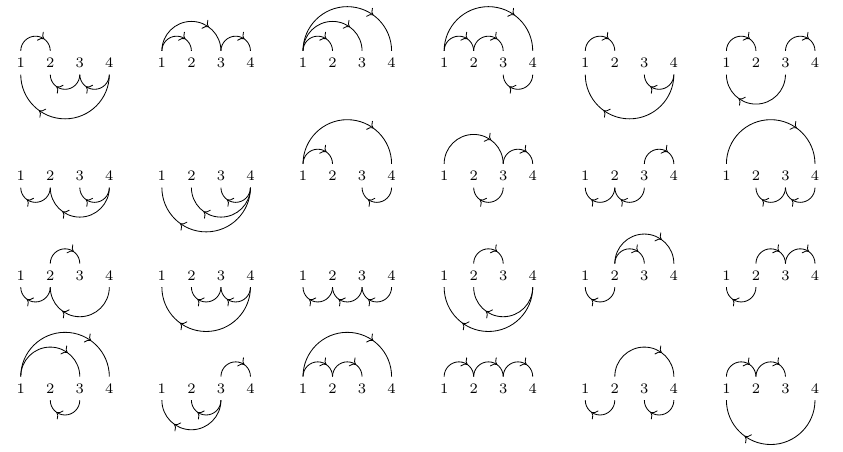}}
\end{example}

\pagebreak

\begin{proposition}
Capped unit interval posets are Tamari interval posets.
\end{proposition}

\begin{proof}
Assume that~$a \less c$ for some~$1 \le a < b < c \le n$.
Consider any saturated chain from~$a$ to~$c$ in~$\less$.
If~$b$ belongs to this chain, then~$a \less b$.
Otherwise, let~$a'$ be the last element~$< b$ and~$c'$ be the first element~$>b$ along this chain.
Hence, we have~$a \less a' \lessdot c' \less c$ and~$1 \le a' < b < c' \le n$.
As~$a' \lessdot c'$ and~$|a'-c'| > 1$, we get that~$a' = m$, hence that~$a = m$.
We conclude that~$a \less b$.
The proof is symmetric if~$a \more c$.
\end{proof}

\begin{remark}
The number of capped unit interval posets is
\[
3^{n-2} + \Bigl( \sum_{m = 2}^{n-1} 3^{m-2} 3^{n-m-1} \Bigr) + 3^{n-2} = (n+5)3^{n-2} \qquad \text{\OEIS{A038765},}
\]
because they are determined by their minimum~$m$ together with their restrictions to each consecutive pair of integers of~$[n] \ssm \{m\}$ (and for each pair $\{i, i+1\}$, there are three possibilities: either $i$ and $i+1$ are incomparable, or $i \less i+1$, or $i+1 \less i$).
The number of capped unit interval posets whose Hasse diagram is a tree is
\[
G_{n-2} + \Bigl( \sum_{m = 2}^{n-1} G_{m-2} G_{n-m-1} \Bigr) + G_{n-2} \qquad\text{\OEIS{A290917}}, 
\]
where~$G_{n} = 3 G_{n-1} - G_{n-2}$ is the bisection of the Fibonacci sequence \OEIS{A001906} (now for each pair $\{i, i+1\}$, say with $m > i$, there are three possibilities: either $i$ and $i+1$ are incomparable, or $i$ is smaller than $i+1$, or $i+1$ is smaller than $i$ which forbids $i-1$ to be smaller than $i$).
Note that this implies that asymptotically, most capped unit interval posets on $n$ nodes are not trees.
\end{remark}

We now consider subfamilies of capped unit interval posets indexed by subsets~$S$ of~$[n-1]$.
Again, the significance will become clear in \cref{exm:cappedUnitIntervalHypergraphicPolytope} (see also \cref{subsubsec:cappedSUnitIntervalPosets,exm:cappedUnitIntervalHypergraphicPolytopePreposets}).
 
\begin{definition}
\label{def:cappedSUnitIntervalPosets}
For~$S \subseteq [n-1]$, a \defn{capped $S$-unit interval poset} is a capped unit interval poset~$\less$ with minimum~$m$ such that
\begin{itemize}
\item $i$ and $i+1$ are comparable in~$\less$ for any $i \in S$,
\item if $i$ and $i+1$ form a cover relation in~$\less$, then~$i \in S$ or~$m \in \{i,i+1\}$.
\end{itemize}
\end{definition}

\begin{remark}
By the same argument as above, capped $S$-unit interval posets are counted by
\[
\sum_{m = 1}^n 2^{|S \ssm \{m-1, m\}|},
\]
and those whose Hasse diagram is a tree are counted by
\[
\sum_{m=1}^n \ell_S(1,m) \cdot r_S(m,n)
\]
where~$\ell_S(j,m)$ (resp.~$r_S(m,j)$) are defined inductively for~$1 \le j \le m$ (resp.~$m \le j \le n$) by~$\ell_S(j,m) = 1$ if $j \ge m-1$ (resp.~$r_S(m,j) = 1$ if $j \le m+1$) and
\begin{align*}
& \ell_S(j,m) =  (1 + \delta_{j \in S}) \cdot \ell_S(j+1,m) - \delta_{\{j, j+1\} \subseteq S \ssm \{m-1\}} \cdot \ell_S(j+2,m) \\
\qquad\text{(resp.}\quad & r_S(m,j) = (1 + \delta_{j-1 \in S}) \cdot r_S(m,j-1) - \delta_{\{j-1, j-2\} \subseteq S \ssm \{m\}} \cdot r_S(m,j-2) \quad \text{)}
\end{align*}
otherwise. (Here, $\delta_X$ denotes the Kronecker symbol, meaning that~$\delta_X = 1$ if~$X$ holds true, and~$\delta_X = 0$ otherwise).
\end{remark}

\begin{example}
In particular:
\begin{itemize}
\item The capped $\varnothing$-unit interval posets are the star posets with a minimum~$m$ covered by all elements of~$[n] \ssm \{m\}$.
They are counted by~$n$, and their Hasse diagrams are all trees.
\item The capped $[n-1]$-unit interval posets are the capped unit interval posets in which any consecutive integers are comparable. 
They are counted by~$(n+2)2^{n-3}$ \OEIS{A045623},
and those whose Hasse diagram is a tree are counted by~$(n-1)(n^2-2n+12)/6$ \OEIS{A177787}.
\end{itemize}
\end{example}

\subsubsection{Uniform interval posets}
\label{subsubsec:uniformIntervalPosets}

We now consider the following posets, whose significance will become clear in \cref{exm:uniformIntervalHypergraphicPolytope} (see also \cref{subsubsec:uniformIntervalPreposets,exm:uniformIntervalHypergraphicPolytopePreposets}).


\begin{definition}
\label{def:uniformIntervalPosets}
For $k \ge 1$, a \defn{$k$-uniform interval poset} is a poset~$\less$ on~$[n]$ such that
\begin{itemize}
\item every length~$k$ subinterval of~$[n]$ admits a unique minimum in~$\less$, and
\item every cover relation connects the minimum in~$\less$ of a length $k$ subinterval to another element of that subinterval.
\end{itemize}
\end{definition}

\begin{example}
The antichain is the only $1$-uniform interval poset.
The $k$-uniform interval posets on~$[4]$ for~$k = 2, 3, 4$ are the following

\medskip
\centerline{\includegraphics[scale=1]{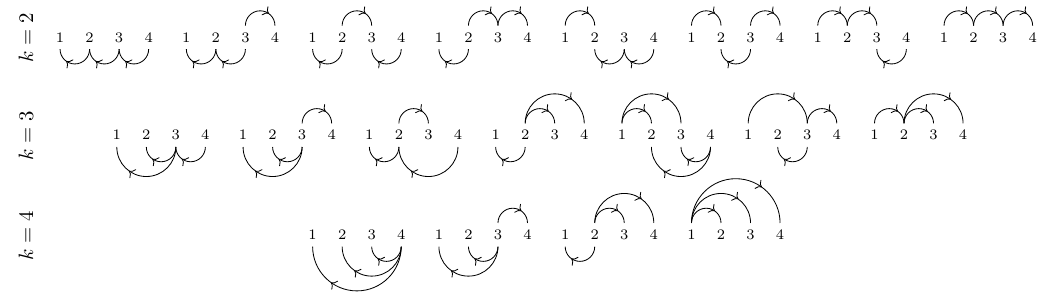}}
\end{example}

\begin{proposition}
All $k$-uniform interval posets are Tamari interval posets.
\end{proposition}

\begin{proof}
Assume that~$a \less c$ for some~$1 \le a < b < c \le n$.
Consider any saturated chain from~$a$ to~$c$ in~$\less$.
If~$b$ belongs to this chain, then~$a \less b$.
Otherwise, let~$a'$ be the last element~$< b$ and~$c'$ be the first element~$>b$ along this chain.
Hence, we have~$a \less a' \lessdot c' \less c$ and~$1 \le a' < b < c' \le n$.
Thus, $a'$ is the minimal element of a length~$k$ interval~$I$ and~$c' \in I$.
We thus obtain that~$b \in I$, so that~$a' \less b$.
We conclude that~$a \less b$.
The proof is symmetric if~$a \more c$.
\end{proof}

\begin{remark}
The first few numbers of such posets are gathered in \cref{table:uniformIntervalPosets}.
We leave as an open problem to compute explicit formulas.

\begin{table}[h]
	\[
	\begin{tabular}{c|ccccccccccccc}
	$k \backslash n$ & $1$ & $2$ & $3$ & $4$ & $5$ & $6$ & $7$ & $8$ & $9$ & $10$ & $11$ & $12$ & $\cdots$ \\
	\hline
	$2$ & $1$ & $2$ & $4$ & $8$ & $16$ & $32$ & $64$ & $128$ & $256$ & $512$ & $1\,024$  & $2\,048$ & $\cdots$ \\ 
	$3$ & $1$ & $1$ & $3$ & $7$ & $16$ & $36$ & $81$ & $182$ & $409$ & $919$ & $2\,065$ & $4\,640$ & $\cdots$ \\ 
	$4$ & $1$ & $1$ & $1$ & $4$ & $10$ & $24$ & $56$ & $128$ & $292$ & $ 664$ & $1\,508$ & $3\,424$ & $\cdots$ \\ 
	$5$ & $1$ & $1$ & $1$ & $1$ & $5$ & $13$ & $32$ & $76$ & $176$ & $400$ & $905$ & $2\,038$ & $\cdots$ \\ 
	$6$ & $1$ & $1$ & $1$ & $1$ & $1$ & $6$ & $16$ & $40$ & $96$ & $224$ & $512$ & $1\,152$ & $\cdots$ \\ 
	$7$ & $1$ & $1$ & $1$ & $1$ & $1$ & $1$ & $7$ & $19$ & $48$ & $116$ & $272$ & $624$ & $\cdots$ \\ 
	\end{tabular}
	\]
	\caption{The numbers of $k$-uniform interval posets on~$[n]$.}
	\label{table:uniformIntervalPosets}
\end{table}
\end{remark}


\subsection{Counting Tamari interval posets}
\label{subsec:countingTamariIntervalPosets}

%

A celebrated result of F.~Chapoton~\cite{Chapoton1} states that the Tamari intervals are enumerated by
\[
\frac{2}{(3n+1)(3n+2)} \binom{4n+1}{n+1}.
\]
The first few values can be found in \cref{table:relevantCombinatorialObjects} or~\OEIS{A000260}. 
These numbers also count the rooted $3$-connected planar triangulations with $2n+2$ faces.
An explicit bijection between Tamari intervals and $3$-connected triangulations was given in~\cite{BernardiBonichon}.
A more recent and powerful approach can be found in~\cite{FangFusyNadeau}.
We also refer to \cite{bostan2023refined} for intriguing refined counting for Tamari interval posets, providing explicit product formulas for the $f$-vectors of the canonical complex of the Tamari lattice and of the cellular diagonal of the associahedron.

The number of connected Tamari interval posets (\ie whose Hasse diagram is connected) is also reported in \cref{table:relevantCombinatorialObjects} or \OEIS{A294084}.
The generating functions $f$ of the Tamari interval posets and $g$ of the connected Tamari interval posets are clearly related by~$f = 1+g+g^2+\cdots = 1/(1-g)$.
\begin{table}[h]
	\centerline{
		\begin{tabular}{c|rrrrrrrrrrr|c}
			$n$ & $1$ & $2$ & $3$ & $4$ & $5$ & $6$ & $7$ & $8$ & $9$ & $10$ & $\cdots$ & OEIS \\
			\hline
			$TIP_n$ & $1$ & $3$ & $13$ & $68$ & $399$ & $2\,530$ & $16\,965$ & $118\,668$ & $857\,956$ & $6\,369\,883$ & $\cdots$ & \href{https://oeis.org/A000260}{A000260} \\
			$CTIP_n$ & $1$ & $2$ & $8$ & $41$ & $240$ & $1\,528$ & $10\,312$ & $72\,647$ & $528\,992$ & $3\,954\,488$ & $\cdots$ & \href{https://oeis.org/A294084}{A294084} \\
			$\ww_n$ & $1$ & $2$ & $8$ & $38$ & $196$ & $1\,062$ & $5\,948$ & $34\,120$ & $199\,316$ & $1\,181\,126$ & $\cdots$ & \href{https://oeis.org/A047098}{A047098} \\
			$R\ww_n$ & $1$ & $2$ & $7$ & $30$ & $143$ & $728$ & $3\,876$ & $21\,318$ & $120\,175$ & $690\,690$ & $\cdots$ & \href{https://oeis.org/A006013}{A006013} \\
			 $TIPX_n$ & $1$ & $2$ & $5$ & $17$ & $67$ & $287$ & $\cdots$ &&&&& ---  \\
			$\ww\!\!X_n$ & $1$ & $2$ & $5$ & $16$ & $67$ & $316$ & $1599$ & $8480$ & $46512$ & $261668$ & $\cdots$ & ---
		\end{tabular}
	}
	\caption{The numbers~$TIP_n$ of Tamari interval posets, $CTIP_n$ of connected Tamari interval posets, $\ww_n$ of weeping willows, $R\ww_n$ of rooted weeping willows, $TIPX_n$ of Tamari interval posets satisfying \cref{prop:inclusionPosetIntervalHypergraphsContainingTamariIntervalPoset}\,(ii), and $\ww\!\!X_n$ of weeping willows satisfying \cref{prop:inclusionPosetIntervalHypergraphsContainingWeepingWillow}\,(ii).}
	\label{table:relevantCombinatorialObjects}
\end{table}


\subsection{Tamari interval posets and interval hypergraphic polytopes}
\label{subsec:TIPIHP}

In this section, we prove \cref{prop:TamariIntervalPosets} connecting Tamari interval posets to interval hypergraphic polytopes.
In the following proposition, we allow for positive scaling of the Minkowski summands (which preserves the normal fan) in the definition of hypergraphic polytopes.
That is, we still call hypergraphic polytope the Minkowski sum~$\sum_{I \subseteq [n]} \lambda_I \simplex_I$ for any~$\lambda_I \ge 0$ for~$I \subseteq [n]$ with~$|I| > 1$.

\begin{proposition}
\label{prop:TamariIntervalPosets1}
The following are equivalent for a polytope~$\poly{P}$:
\begin{enumerate}[(i)]
\item $\poly{P}$ is an interval hypergraphic polytope,
\item $\poly{P}$ is a deformed permutahedron whose vertex posets are Tamari interval posets,
\item $\poly{P}$ is a deformed associahedron, \ie a deformation of $\Asso$.
\end{enumerate}
\end{proposition}

\begin{proof}
\para{(i) $\Rightarrow$ (ii).}
We already mentioned that all hypergraphic polytopes are deformed permutahedra, so we just need to show that all vertex posets of an interval hypergraphic polytope are Tamari interval posets.
This was already observed in~\cite[Prop.~3.10]{BergeronPilaud}, and we already repeated the argument in \cref{lem:avoidingPatterns}\,\eqref{item:ap2}.

\para{(ii) $\Rightarrow$ (iii).}
Consider a Tamari interval poset~$\less$, corresponding to an interval~$[T,T']$ between two binary trees~$T$ and~$T'$ in the Tamari lattice.
The cone defined by~$\less$ is then the union of the cones of the sylvester fan corresponding to the binary trees in~$[T,T']$.
Hence, if all vertex posets of~$\poly{P}$ are Tamari interval posets, then the normal fan of~$\poly{P}$ coarsens the sylvester fan, so that~$\poly{P}$ is a deformed associahedron.

\para{(iii) $\Rightarrow$ (i).}
It follows from~\cite{BazierMatteChapelierLaguetDouvilleMousavandThomasYildirim,PadrolPaluPilaudPlamondon,PadrolPilaudPoullot-deformedNestohedra} that the deformation cone of the associahedron~$\Asso$ is simplicial, and its rays are the faces~$\simplex_{[i,j]}$ of the standard simplex corresponding to intervals~$[i,j]$ of~$[n]$.
Hence, any deformed associahedron is a non-negative Minkowski sum of~$\simplex_{[i,j]}$, in other words, an interval hypergraphic polytope.
\end{proof}

\begin{proposition}
\label{prop:TamariIntervalPosets2}
Any Tamari interval poset is a vertex poset of an interval  hypergraphic polytope.
\end{proposition}

\begin{proof}
%
For a Tamari interval poset~$\less$, recall that we denote $a^\less \eqdef \set{b \in [n]}{a \less b}$ for $a \in [n]$.
By \cref{def:TamariIntervalPoset}\,\eqref{item:i_less_m_M}, the set $a^\less$ is an interval for all $a \in [n]$.
Consider the orientation~$A_\less$ of the interval hypergraph~$\II_\less \eqdef \set{a^\less}{a \in [n]}$ defined by~$A_\less(a^\less) = a$ for all~$a \in [n]$.
It is clear that~$A_\less$ is acyclic and that~$\less_{A_\less} = \less$.
\end{proof}

We now give explicitly the coordinates of the vertex of $\simplex_\II$ associated to a given Tamari interval poset, see \Cref{fig:coordinates} for an example.

\begin{proposition}
\label{prop:coordinates}
The vertex of an interval hypergraphic polytope~$\simplex_\II$ corresponding to a Tamari interval poset~$\less$ has $i$-th coordinate $\# \set{I \in \II}{i \in I \subseteq i^\less}$, that is the number of intervals of $\II$ having $i$ as its minimum.
\end{proposition}

\begin{proof}
Consider any vector~$\b{c}$ such that~$a \less b$ implies~$c_a < c_b$.
Recall that the vertex of a Minkowski sum~$\sum_i \poly{P}_i$ minimizing a generic direction~$\b{c}$ is the Minkowski sum of the vertices of the summands~$\poly{P}_i$ minimizing~$\b{c}$.
Moreover, the vertex minimizing~$\b{c}$ on a simplex~$\simplex_I$ is $\b{e}_j$, where~$j \in I$ is such that~$c_j = \min\set{c_k}{k \in I}$.
Hence, the vertex of~$\simplex_\II$ minimizing~$\b{c}$ has $i$-th coordinate~$\# \set{I \in \II}{i \in I \subseteq i^\less}$.
\end{proof}


\begin{figure}
	\centerline{\includegraphics[scale=1]{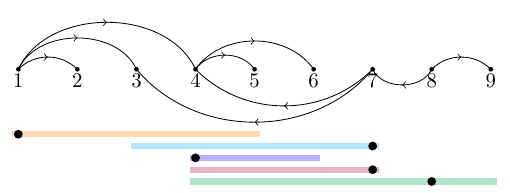}}
	\caption{An interval hypergraph $\II$ with a Tamari interval poset associated to the vertex $\b v$ of $\simplex_\II$. By \Cref{prop:coordinates}, counting the number of intervals of $\II$ with minimum element $i$ for each $i\in [9]$ yields $\b v = (1, 0, 0, 1, 0, 0, 2, 1, 0)$}
	\label{fig:coordinates}
\end{figure}


\subsection{A finer description}
\label{subsec:finerDescriptionTamariIntervalPoset}

We now refine \cref{prop:TamariIntervalPosets2} to describe which Tamari interval poset appears as a vertex poset in which interval hypergraphic polytope.

\begin{theorem}
\label{thm:TamariIntervalPosetInIntervalHypergraphic}
A poset~$\less$ is a vertex poset of an interval hypergraphic polytope~$\simplex_\II$ if and only if
\begin{enumerate}[(i)]
\item for any~$I \in \II$, there is~$a \in [n]$ such that~$a \in I \subseteq a^\less$, \label{cond:belong1}
\item for any cover relation~$a \lessdot b$, there is~$I \in \II$ such that~$\{a,b\} \subseteq I \subseteq a^\less$. \label{cond:belong2}
\end{enumerate}
\end{theorem}

\begin{proof}
Assume first that~${\less} = {\less_A}$ is the vertex poset of the hypergraphic polytope~$\simplex_\II$ corresponding to the acyclic orientation~$A$ of~$\II$.
Then, by \Cref{prop:acyclicOrientations}:
\begin{enumerate}[(i)]
\item For any $I \in \II$, we have $A(I) \in I \subseteq A(I)^{\less}$.
\item For any cover relation~$a \lessdot b$, there is $I \in \II$ such that $a = A(I)$ and $b \in I$. In particular, $\{a,b\} \subseteq I \subseteq A(I)^\less = a^\less$.

\end{enumerate}
Hence, $\less$ and~$\II$ satisfy~\eqref{cond:belong1} and~\eqref{cond:belong2}.

Conversely, consider a poset~$\less$ and an interval hypergraph~$\II$ that satisfy~\eqref{cond:belong1} and~\eqref{cond:belong2}.
We now construct an acyclic orientation~$A$ of~$\II$ such that~${\less} = {\less_A}$.
For all~$I \in \II$, we define~$A(I)$ to be the unique~$a \in [n]$ such that~$a \in I \subseteq a^\less$.
Note that it exists by~\eqref{cond:belong1}, and it is unique by antisymmetry of~$\less$. 
We claim that
\begin{itemize}
\item $A$ is acyclic. Otherwise, there are~$I,J \in \II$ with~$A(I) \in J$ and~$A(J) \in I$ by \cref{prop:length2cycles}. By definition of~$A$, we then get~$A(I) \in J \subseteq A(J)^\less$ and~$A(J) \in I \subseteq A(I)^\less$. This implies that~$A(I) = A(J)$ by antisymmetry of~$\less$.
\item ${\less} = {\less_A}$. Indeed, any cover relation of~$\less_A$ is clearly a relation in~$\less$. Conversely, for any cover relation~$a \lessdot b$, there exists~$I \in \II$ such that~$\{a,b\} \subseteq I \subseteq a^\less$ by~\eqref{cond:belong2}. Hence, $a = A(I)$ and~$b \in I$ so that we get~$a \less_A b$.
\end{itemize}
We conclude that~${\less} = {\less_A}$ is indeed the vertex poset of the hypergraphic polytope~$\simplex_\II$ corresponding to the acyclic orientation~$A$ of~$\II$.
\end{proof}

\begin{remark}
Two comments on \cref{thm:TamariIntervalPosetInIntervalHypergraphic}.
First, in \eqref{cond:belong1}, for any~$I \in \II$, there exists a unique~$a\in[n]$ such that~$a \in I \subseteq a^\less$.
Second, \eqref{cond:belong2} implies that for any~$a \in [n]$, there is~$I \in \II$ such that~$a \in I \subseteq a^\less$. However, the latter is not equivalent to \eqref{cond:belong1} (nor to \eqref{cond:belong2}).
\end{remark}

\begin{example}
\label{exm:cappedUnitIntervalHypergraphicPolytope}
For~$S \subseteq [n-1]$, the vertex posets of the hypergraphic polytope~$\simplex_{\II_S}$ of \linebreak ${\II_S = \set{[i,i+1]}{i \in S} \cup \{[n]\}}$ are precisely the capped $S$-unit interval posets of~\cref{def:cappedSUnitIntervalPosets}.
From \cref{prop:coordinates}, we obtain that the $i$-th coordinate of the vertex~$\b{v}_\less$ of~$\simplex_{\II_S}$ corresponding to a capped $S$-unit interval poset~$\less$ with minimum~$m$ is $\delta_{i-1 \more i} + \delta_{i \less i+1} + \delta_{i = m}$ (here, $n > 2$ so that the intervals~$[1,2]$ and~$[n]$ do not coincide).
For instance, we illustrate below the correspondence between the capped $[3]$-unit interval posets on~$[4]$ and the acyclic orientations of the interval hypergraph~$\{[1,2], [2,3], [3,4], [1,4]\}$.
We mark the minimum of each interval with a bullet {\scriptsize $\bullet$}, and the $i$-th coordinate of the associated vertex is the number of intervals whose bullet {\scriptsize $\bullet$} lies at~$i$.

\medskip
\centerline{\includegraphics[scale=1]{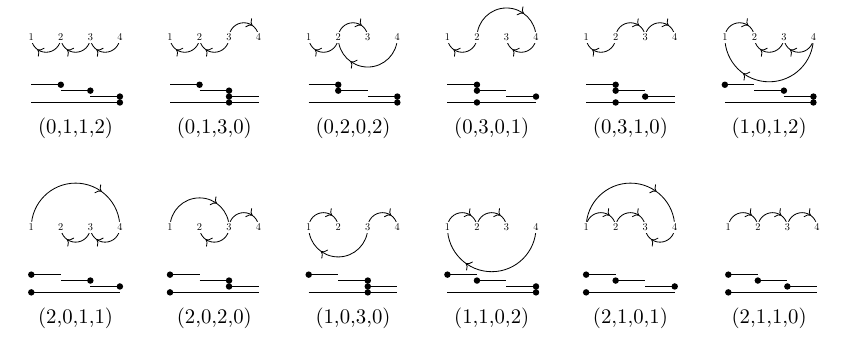}}
\end{example}


\begin{example}
\label{exm:uniformIntervalHypergraphicPolytope}
For the (complete) $k$-uniform interval hypergraph~$\II \eqdef \set{[i,i+k-1]}{i \in [n-k]}$, the hypergraphic polytope~$\simplex_\II$ is a non-simple polytope whose vertex posets are precisely the $k$-uniform interval posets of~\cref{def:uniformIntervalPosets}.
We illustrate below the correspondence for~$n = 4$ and~$k = 2, 3, 4$.

\medskip
\centerline{\includegraphics[scale=1]{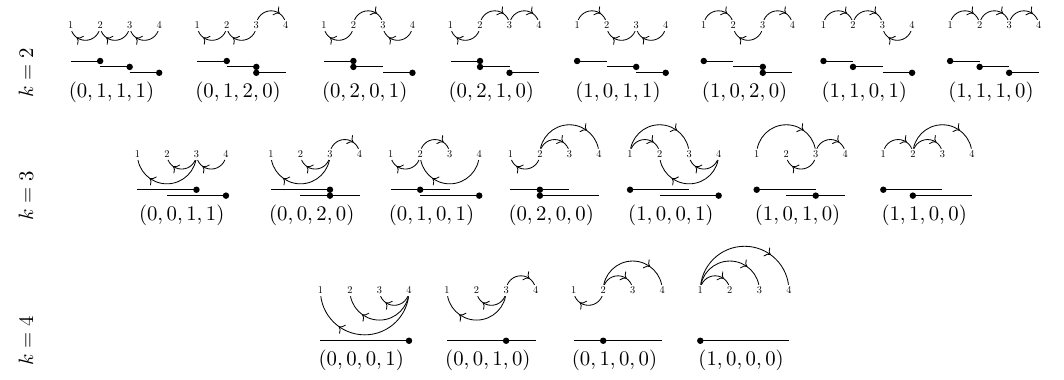}}
\end{example}

We now exploit \cref{thm:TamariIntervalPosetInIntervalHypergraphic} to describe the set of interval hypergraphs for which~$\less$ is a vertex poset.
Here, we restrict attention to interval hypergraphs containing all singletons, as dropping this assumption would simply yield a Cartesian product with a boolean lattice (see \cref{rem:singletons}).
We denote by~\defn{$U(\less)$} the maximal cover relations of~$\less$ for the order given by~$(a \lessdot b) \prec (a \lessdot c)$ for~$1 \le a < b < c \le n$ or~$1 \le c < b < a \le n$.
In other words, $U(\less)$ consists of the covers that are maximal among the covers with the same source.

\begin{proposition}
\label{prop:inclusionPosetIntervalHypergraphsContainingTamariIntervalPoset}
We denote by~$\mathdefn{\c{I}}$ the inclusion poset of interval hypergraphs~$\II$ on~$[n]$ containing all singletons (which is isomorphic to a boolean lattice on $\binom{n}{2}$ elements).
For a given Tamari interval poset~$\less$, denote by~\defn{$\c{I}_\less$} the subposet of~$\c{I}$ induced by interval hypergraphs~$\II$ on~$[n]$ such that~$\less$ is a vertex poset of~$\simplex_\II$.
Then
\begin{enumerate}[(i)]
\item $\c{I}_\less$ is order convex in~$\c{I}$.
\item The interval hypergraph~$\mathdefn{\II_\less} \eqdef \set{a^\less}{a \in [n]}$ is a minimal element of~$\c{I}_\less$. Moreover, it is the unique minimal element of~$\c{I}_\less$ if and only if there are no~$1 \le a < b < c \le n$ such that~$a \lessdot b \lessdot c$, or $a \moredot b \moredot c$, or $a \moredot b \lessdot c$ in~$U(\less)$. \label{item:existsMinimalIntervalHypergraph}
\item The interval hypergraph~$\mathdefn{\JJ_\less} \eqdef \bigcup_{a \in [n]} \set{I}{a \in I \subseteq a^\less}$ is the unique maximal element of~$\c{I}_\less$.
\end{enumerate}
\end{proposition}

\begin{example}
The interval hypergraphs~$\II_\less$ and~$\JJ_\less$ are illustrated in \cref{fig:poset_of_intervals} for the Tamari interval poset drawn in \Cref{fig:coordinates}.
Note that~$\II_\less$ is not the unique minimal element of~$\c{I}_\less$ since we have~$1 \lessdot 4 \lessdot 6$ in~$U(\less)$.
Another minimal element of~$\c{I}_\less$ is obtained by replacing~$[1,6]$ by~$[1,4]$.
\begin{figure}
	\centerline{\includegraphics[scale=1]{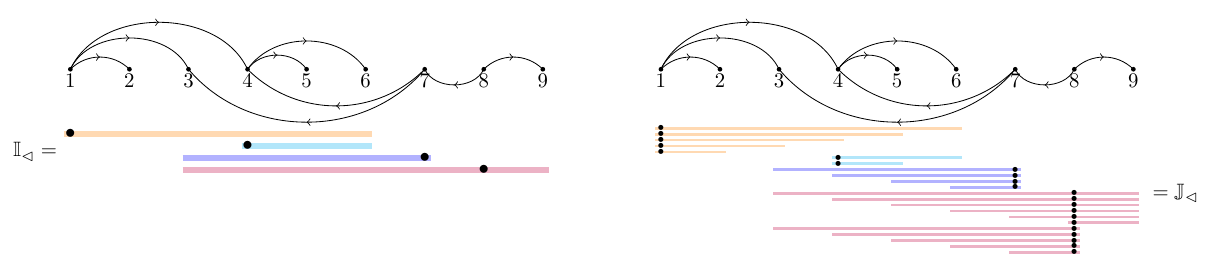}}
	\caption{The interval hypergraphs $\II_\less$ (left) and $\JJ_\less$ (right)  for the Tamari interval poset~$\less$ drawn on top. We did not draw the singletons they contain. The vertex corresponding to $\less$ is $\b v_{\less} = (2, 1, 1, 2, 1, 1, 2, 2, 1)$ in the polytope $\simplex_{\II_\less}$ and $\b v_\less = (6, 1, 1, 3, 1, 1, 5, 12, 1)$ in the polytope $\simplex_{\JJ_\less}$. (Count the number of {\scriptsize $\bullet$} above each coordinate and add~$1$ to account for the singleton which is not drawn.) Note that~$U(\less) = \{1 \lessdot 4, \, 4 \lessdot 6, \, 7 \lessdot 3, \, 8 \lessdot 7, \, 8 \lessdot 9\}$.}
	\label{fig:poset_of_intervals}
 \end{figure}   
\end{example}

\begin{example}
\cref{prop:inclusionPosetIntervalHypergraphsContainingTamariIntervalPoset} is illustrated in \cref{fig:LatticeOfSlessInIless} with the Tamari interval poset~$\less$ given by the transitive closure of the tree~$\ww$.
The set~$\c{I}_\less$ contains a unique maximal element~\mbox{$\JJ_\less =$ \!\drawSubset[green]{2,3,4,5,6}\!\!} but three minimal elements $\II_\less =$ \!\drawSubset[green]{3}\!\!, and also \!\drawSubset[red]{2,5}\!\!, and \!\drawSubset[blue]{2,6}\!\!.
\begin{figure}
    \centering
    \includegraphics{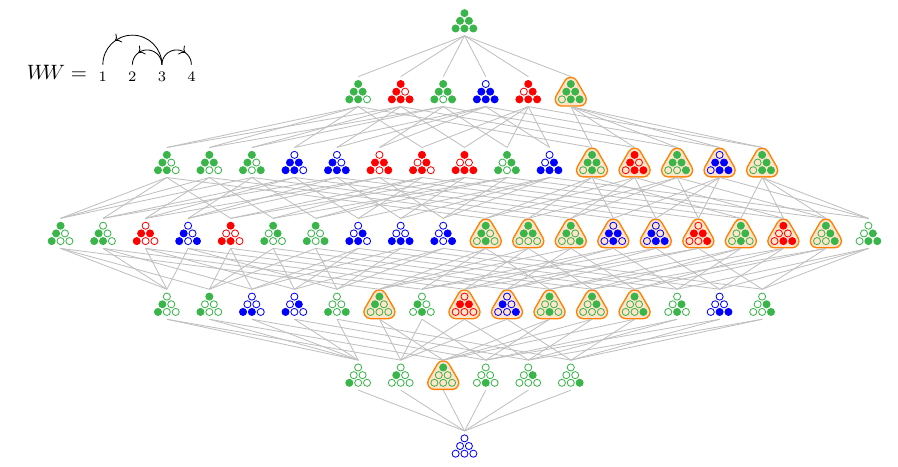}
    \caption{The boolean lattice $\c I$ of interval hypergraphs on $4$ nodes (see \cref{fig:LatticeOfSinI} for the encoding and color conventions: \textcolor{red}{red} yield non-simple interval hypergraphic polytopes; \textcolor{blue}{blue} and \textcolor{green}{green} yield simple ones; \textcolor{green}{green} yield nestohedra), we have encircled (in \textcolor{orange}{orange}) the interval hypergraphs containing a given tree~$\ww$ as a vertex poset.}
    \label{fig:LatticeOfSlessInIless}
   \end{figure}
\end{example}

\begin{proof}[Proof of \Cref{prop:inclusionPosetIntervalHypergraphsContainingTamariIntervalPoset}]
Recall that $\II\in\c I_\less$ if and only if ${\poly{C}_\less \eqdef \set{\b{x} \in \R}{x_i \ge x_j \text{ if } i \less j}}$ is a maximal cone in the normal fan $\simplex_{\II}$.
Observe that for any polytopes~$\poly{P}, \poly{Q}, \poly{R}$, if~$\poly{C}$ is a cone in the normal fan of~$\poly{P}$ and a cone in the normal fan of~$\poly{P} + \poly{Q} + \poly{R}$, then $\poly{C}$ is a cone in the normal fan of~$\poly{P} + \poly{Q}$.
Hence for interval hypergraphs~$\II \subseteq \JJ \subseteq \KK$, if $\II, \KK \in \c I_\less$, then $\poly{C}_\less$ is a cone in the normal fan of $\simplex_{\II}$ and of $\simplex_{\KK} = \simplex_{\II} + \sum_{J \in \JJ\ssm\II} \simplex_J + \sum_{K \in \KK\ssm\JJ} \simplex_K$, thus also in the normal fan of $\simplex_{\JJ} = \simplex_{\II} + \sum_{J \in \JJ\ssm\II} \simplex_J$, implying that $\JJ\in\c I_\less$.
This shows that~$\c{I}_\less$ is order convex inside~$\c{I}$.

Note that both~$\II_\less$ and~$\JJ_\less$ belong to~$\c{I}_\less$ as they clearly satisfy the conditions~\eqref{cond:belong1} and~\eqref{cond:belong2} of \cref{thm:TamariIntervalPosetInIntervalHypergraphic}.
Moreover, $\II_\less$ is an inclusion minimal interval hypergraph satisfying~\eqref{cond:belong2} (because for any~$a \in [n]$, the interval~$a^\less$ is the only interval of~$\II_\less$ containing~$a$ and contained in~$a^\less$), while $\JJ_\less$ is the unique inclusion maximal interval hypergraph satisfying~\eqref{cond:belong1} (since~$J_\less$ contains all intervals~$I$ for which there exists~$a \in [n]$ with~$a \in I \subseteq a^\less$).
Hence, we just have to discuss when~$\II_\less$ is the unique inclusion minimal element of~$\c{I}_\less$.

Assume first that there are no~$1 \le a < b < c \le n$ such that~$a \lessdot b \lessdot c$, or $a \moredot b \moredot c$, or $a \moredot b \lessdot c$ in~$U(\less)$.
As~$\less$ is a Tamari interval poset, any path of cover relations from~$a$ to~$\min(a^\less)$ (resp.~$\max(a^\less)$) is actually a path in~$U(\less)$.
The forbidden patterns thus force~$a^{\less}$ to be either trivial or to be an interval defined by a unique arc in $U(\less)$.
That is, for any~$a \in [n]$, either~$a^\less = \{a\}$, or~$a^\less = [a,b]$ with~$a \lessdot b$, or~$a^\less = [b,a]$ with~$a \lessdot b$.
We thus obtain that any interval hypergraph satisfying~\eqref{cond:belong2} must contain the interval~$a^\less$.
We conclude that~$\II_\less$ is the unique inclusion minimal element of~$\c{I}_\less$.

Conversely, assume that there are~$1 \le a < b < c \le n$ such that~$a \lessdot b \lessdot c$ in~$U(\less)$.
Consider~$\II \eqdef \II_\less \ssm \{a^\less\} \cup \{[\min(a^\less),b]\}$.
It is immediate that~$\II$ still satisfies the conditions~\eqref{cond:belong1} and~\eqref{cond:belong2} of \cref{thm:TamariIntervalPosetInIntervalHypergraphic}.
As~$\II_\less \not\subseteq \II$, we get that~$\c{I}_\less$ does not admit a unique inclusion minimal element.
The proof is identical if $a \moredot b \moredot c$ (resp.~$a \moredot b \lessdot c$) are both in~$U(\less)$, considering the interval hypergraph~$\II_\less \ssm \{c^\less\} \cup \{[b, \max(c^\less)]\}$ (resp.~$\II_\less \ssm \{b^\less\} \cup \{[\min(b^\less),b], [b, \max(b^\less)]\}$).
%
\end{proof}

\begin{remark}\label{rmk:coordinatesIlessJless}
It immediately follows from \cref{prop:coordinates} that
\begin{itemize}
\item the vertex of~$\simplex_{\II_\less}$ corresponding to~$\less$ is the characteristic vector of the nodes of~$\less$ with at least one outgoing neighbor (\ie of the non-maximal elements of $\less$), translated by~$(1,\dots,1)$ to account for singletons, and
\item the coordinates of the vertex of~$\simplex_{\JJ_\less}$ corresponding to~$\less$ are given by Loday's product formula
\[
\Bigl( \bigl( a - \min(a^\less) + 1 \bigr) \bigl( \max(a^\less) - a + 1 \bigr) \Bigr)_{a \in [n]}.
\]
\end{itemize}
See \Cref{fig:poset_of_intervals}.
Note that these formulas hold since~$\II_\less$ and~$\JJ_\less$ contain all singletons (in contrast to the vertex coordinates presented in \cref{fig:coordinates,exm:cappedUnitIntervalHypergraphicPolytope,exm:uniformIntervalHypergraphicPolytope}).
\end{remark}

\begin{remark}
The number of Tamari interval posets~$\less$ on~$[n]$ for which there are no~$1 \le a < b < c \le n$ such that~$a \lessdot b \lessdot c$, or $a \moredot b \moredot c$, or $a \moredot b \lessdot c$ in~$U(\less)$ as described in \cref{prop:inclusionPosetIntervalHypergraphsContainingTamariIntervalPoset}\,\eqref{item:existsMinimalIntervalHypergraph} are reported in \cref{table:relevantCombinatorialObjects}.
\end{remark}


\section{Weeping willows and vertices of simple interval hypergraphic polytopes}
\label{sec:weepingWillows}

In this section, we consider the vertex trees of the simple interval hypergraphic polytopes.
These trees are certain relevant non-crossing trees, closely connected to Tamari interval posets.


\subsection{Weeping willows}
\label{subsec:weepingWillows}

We now focus on the following trees, whose name will be explained by \cref{rem:weeping}.
A typical example is illustrated in \cref{fig:weepingWillow}, and four interesting families will be presented in \cref{subsec:examplesWeepingWillows}.

\begin{figure}[h]
	\centerline{\includegraphics[scale=1]{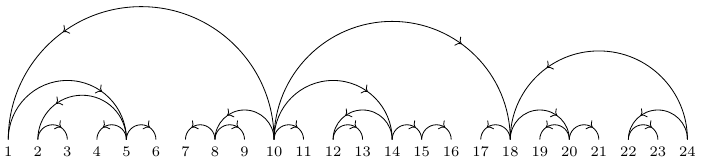}}
	\caption{A typical weeping willow. The crown arcs are $(10,1),(10,18)$ and $(24,18)$. In contrast, $U(\less)$ contains all cover relations except $5 \lessdot 4$, $10 \lessdot 8$, $10 \lessdot 11$, $10 \lessdot 14$ and $24 \lessdot 22$.}
	\label{fig:weepingWillow}
\end{figure}

\begin{definition}
\label{def:weepingWillow}
A \defn{weeping willow} is a directed tree~$\ww$ on~$[n]$ such that for all~$1 \le a < b < c \le n$, if~$\ww$ contains the arc~$(a,c)$ (resp.~$(c,a)$), then it contains a directed path from~$a$ to~$b$ (resp.~from~$c$ to~$b$).
Moreover, $\ww$ is \defn{rooted} if it has a unique source~$a \in [n]$ (\ie if there is a directed path from~$a$ to any node in~$[n]$).
\end{definition}
 
\begin{definition}
We call \defn{crown arcs} of a weeping willow~$\ww$ the arcs of~$\ww$ which are maximal for the order given by~$(i, j) \preceq (k, \ell)$ for~$\min(k,\ell) \le \min(i,j)$ and~$\max(i,j) \le \max(k,\ell)$.
\end{definition}

In our figures, the crown arcs are the topmost arcs.
In particular, every crown arc is a cover relation in~$U(\less)$ (where~$\less$ is the transitive closure of the weeping willow), but the reverse is false, see \cref{fig:weepingWillow}.
We now explore a few basic properties of weeping willows.

\begin{lemma}
\label{lem:weeping}
For any arc~$(i,j)$ of a weeping willow~$\ww$, the subtree of~$\ww$ induced by\linebreak $[\min(i,j), \max(i,j)]$ is rooted at~$i$.
\end{lemma}

\begin{proof}
Assume for instance that~$i < j$ and let~$k \in [i+1, j]$.
As~$(i, j)$ is an arc of the weeping willow~$\ww$, we obtain that there is a directed path from~$i$ to~$k$.
This implies that~$k$ is not a source, so that $i$ is the only source of~$\ww$.
\end{proof}

\begin{remark}
\label{rem:weeping}
\cref{lem:weeping} justifies our choice of name for ``weeping willows''.
Namely, except their crown arcs which can be directed in both directions, all their other arcs are weeping away from the crown arcs and towards the leaves.
\end{remark}

\begin{proposition}
\label{prop:WWTIP}
The weeping willows are precisely the Hasse diagrams of the Tamari interval posets which are trees.
\end{proposition}

\begin{proof}
Consider a weeping willow~$\ww$, and let~$\less$ be its transitive closure.
As $\ww$ is a tree, it is the Hasse diagram of~$\less$, so we just need to show that~$\less$ is a Tamari interval poset.
Consider any~$1 \le a < b < c \le n$ such that~$a \less c$.
As~$\less$ is the transitive closure of~$\ww$, there are~$a = x_0, \dots, x_k = c$ such that~$(x_{i-1}, x_i) \in \ww$ for each~$i \in [k]$.
Since~$x_0 = a < b < c = x_k$, there exists~$i \in [k]$ such that~$x_{i-1} \le b \le x_i$.
As~$\ww$ is a weeping willow, it contains a directed path from~$x_{i-1}$ to~$b$, hence also in a directed path from~$a$ to~$b$, implying that~$a \less b$.
Similarly, if~$1 \le a < b < c \le n$ such that~$a \more c$, then~$b \more c$.
Thus, $\less$ is a Tamari interval poset, whose Hasse diagram is~$\ww$.

Conversely, consider a Tamari interval poset~$\less$ whose Hasse diagram is a directed tree~$\ww$.
Consider~$1 \le a < b < c < d \le n$. 
For any~$1 \le a < b < c \le n$, if~$(a,c) \in \ww$ then $a\lessdot c$ implies $a \less b$, hence there is a path from~$a$ to~$b$.
Similarly, if~$(c,a) \in \ww$, there is a path from~$c$ to~$b$.
Hence, $\ww$ is weeping.
\end{proof}

\begin{proposition}
\label{prop:noncrossing}
Any weeping willow~$\ww$ is non-crossing, \ie for all~$1 \le a < b < c < d \le n$, the pairs $\{a,c\}$ and~$\{b,d\}$ cannot both be undirected arcs of~$\ww$.
\end{proposition}

\begin{proof}
Assume that there are~$1 \le a < b < c < d \le n$ such that both $\{a,c\}$ and~$\{b,d\}$ are undirected arcs of~$\ww$.
As~$\ww$ is a tree, it is the Hasse diagram of its transitive closure~$\less$.
We distinguish three cases, depending on the orientation of~$\{a,c\}$ and~$\{b,d\}$:
\begin{enumerate}[(i)]
\item If~$(a,c), (b,d) \in \ww$, then $a \less b\less c$ since $\ww$ is weeping, contradicting that~$(a,c) \in \ww$. The situation where~$(c,a), (d,b) \in \ww$ is symmetric.
\item If~$(c,a), (b,d) \in \ww$, then~$c \less b$ and~$b \less c$, contradicting the antisymmetry of~$\less$.
\item If~$(a,c), (d,b) \in \ww$, then $a \less b$ and~$d \less c$. If~$b \less c$ (resp.~$a \less d$), then~$a \less b \less c$ (resp.~$a \less d \less c$) contradicts that~$(a,c) \in \ww$. By symmetry, we obtain that both pairs~$\{a,d\}$ and~$\{b,c\}$ are incomparable in~$\less$. As~$a \less b$, $a \less c$, $d \less b$ and~$d \less c$, and~$\{a,d\}$ and~$\{b,c\}$ are incomparable in~$\less$, we conclude that~$\ww$ is not a tree by \cref{lem:posetTree}\,\eqref{item:bowtie}.
\qedhere
\end{enumerate}
\end{proof}

%
\begin{remark}
\cref{prop:WWTIP} implies that we can represent weeping willows as non-crossing trees using our drawing convention of \cref{rem:drawingTIP} for Tamari interval posets.
\cref{prop:noncrossing} implies that we can actually also represent weeping willows as non-crossing trees with all increasing and decreasing arcs drawn above the horizontal line.
For the sake of compactness, we will prefer the later drawing convention of weeping willows.
\end{remark}


\subsection{Four families of weeping willows}
\label{subsec:examplesWeepingWillows}

We now describe four interesting families of weeping willows, which we illustrate when~$n = 4$.

\begin{example}
\label{exm:orientationPath}
Any orientation of the path graph~$\set{\{i,i+1\}}{i \in [n-1]}$ is a weeping willow. Only~$n$ of them are rooted.
See also \cref{exm:cube}.

\medskip
\centerline{\includegraphics[scale=1]{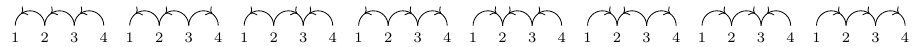}}
\end{example}

\begin{example}
\label{exm:binarySearchTrees}
Any rooted binary tree, with arcs oriented away from its root, and with nodes labeled in inorder (meaning that for each node, we first label its left subtree, then the node, then its right subtree) is a rooted weeping willow.
See also \cref{exm:associahedron}.

\medskip
\centerline{\includegraphics[scale=1]{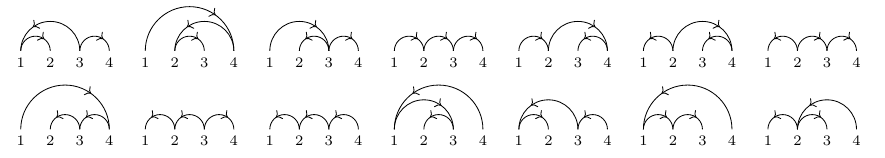}}
\end{example}

\begin{example}
\label{exm:PitmanStanleyTrees}
We inductively define the \defn{Pitman--Stanley trees} of size~$n$ as all trees obtained by adding an arbitrarily oriented arc between~$n$ and the root of a Pitman--Stanley tree of size~$n-1$.
By construction, Pitman--Stanley trees are rooted weeping willows.
Moreover, there are $2^{n-1}$ Pitman--Stanley trees of size~$n$, and there is a natural bijection sending a Pitman--Stanley tree to the subset of integers~$i \in [n-1]$ such that~$i+1$ has an incoming neighbor in~$[i]$.
See also \cref{exm:PitmanStanley}.

\medskip
\centerline{\includegraphics[scale=1]{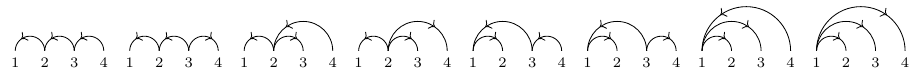}}
\end{example}

\begin{example}
\label{exm:freehedronTrees}
Let us denote $\mathdefn{{\rm PS}^-_n}$ the Pitman--Stanley trees of size $n$ for which the arc incident to $n$ is oriented away from $n$. 
Clearly $|{\rm PS}^-_n|=2^{n-2}$. 
The \defn{freehedron trees} are all  trees we get by taking the concatenation of  $W_i, \overline{W_j}$
where  $W_i \in {\rm PS}^-_i, W_j \in {\rm PS}^-_j$, $i+j=n+1$ and $\overline{W_j}$ is $W_j $ reflected. 
Their total number is $2^{n-2} + \sum_{i = 2}^{n-1} 2^{i-2} 2^{n-i-1} + 2^{n-2} = (n+2)2^{n-3}$ \OEIS{A045623}.
They are rooted weeping willows. 
See also \cref{exm:freehedron}.

\medskip
\centerline{\includegraphics[scale=1]{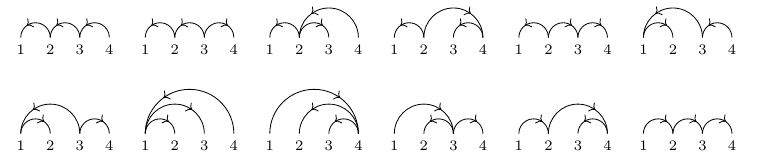}}
\end{example}

\subsection{Counting weeping willows}
\label{subsec:countingWeepingWillows}

We now exploit \cref{prop:noncrossing} to count (rooted) weeping willows.
See also \cref{table:relevantCombinatorialObjects}.

%
%

\begin{proposition}
\label{prop:generatingFunctionsWeepingWillows}
Let
\begin{align*}
f & = 1 + x + 3x^2 + 12x^3 + 55x^4 + 273x^5 + 1428x^6 + 7752x^7 + 43263x^8 + 246675x^9 + \cdots \\
g & = 1 + 2x + 8x^2 + 38x^3 + 196x^4 + 1062x^5 + 5948x^6 + 34120x^7 + 199316x^8 + 1181126x^9 + \cdots \\
h & = 1 + 2x + 7x^2 + 30x^3 + 143x^4 + 728x^5 + 3876x^6 + 21318x^7 + 120175x^8 + 690690x^9 + \cdots
\end{align*}
denote the ordinary generating functions of unoriented non-crossing trees \OEIS{A001764}, of weeping willows \OEIS{A047098}, and of rooted weeping willows \OEIS{A006013} respectively, where~$x$ counts the number of arcs.
Then
\[
f = 1 + x f^3,
\qquad
g = 1 + 2 x g f^2,
\qquad\text{and}\qquad
h = 1 + x f^2 (h+f).
\]
Hence, $f$, $g$, and $h$ are solutions of
\[
x f^3 - f + 1 = 0,
\qquad
(8x-1)g^3-g^2+g+1 = 0,
\qquad\text{and}\qquad
x^2 h^3-2x h^2+h-1 = 0,
\]
and, on $n$ nodes, there are $\frac{1}{2n+1}\binom{3n}{n}$ non-crossing trees, $2\binom{3n}{n}-\sum_{k=0}^n \binom{3n}{k}$ weeping willows, and $\frac{1}{n+1}\binom{3n+1}{n}$ rooted weeping willows.
\end{proposition}

\begin{figure}[H]
	\centering
	\includegraphics[width=0.999\linewidth]{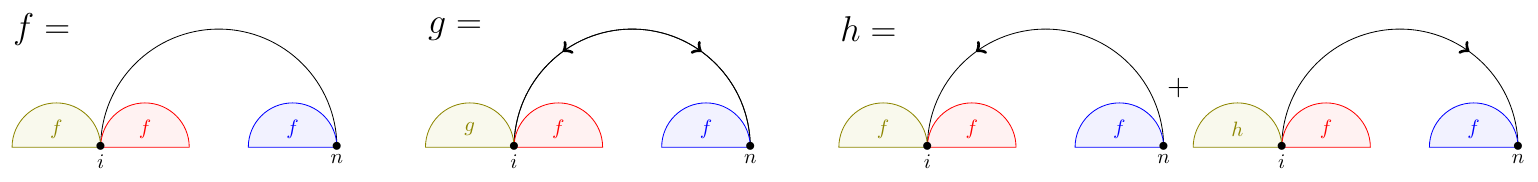}
	\caption{Illustration for the proof of \Cref{prop:generatingFunctionsWeepingWillows}.} 
	\label{fig:GeneratingFunctionsFGH}
\end{figure}

\begin{proof}
We refer the reader to \Cref{fig:GeneratingFunctionsFGH} for an illustration of the following decompositions.
The longest arc~$\{i,n\}$ incident to node~$n$ in a non-crossing tree decomposes it into three non-crossing trees (one on the left of~$i$, one on the right of~$i$, and one on the left of~$n$), which yields the formula~$f = 1 + x f^3$.
The longest arc~$\{i,n\}$ incident to node~$n$ in a weeping willow is oriented arbitrarily and decomposes it into a weeping willow (left of~$i$) and two non-crossing trees by \Cref{lem:weeping} (one on the right of~$i$, and one on the left of~$n$), which yields~$g = 1 + 2 x g f^2$.
The longest arc~$\{i,n\}$ incident to node~$n$ in a rooted weeping willow with root $n$ (resp.~distinct from~$n$) decomposes it into three non-crossing trees by \Cref{lem:weeping} (resp. into a rooted weeping willow and two non-crossing trees), which yields~${h = 1 + x f^2 (h+f)}$.
Observing that~$xf^2 = 1 - 1/f$, we obtain that ${g = 1 / (1-2 x f^2) = f / ( 2 - f )}$ and~$h = f / (1 - x f^2) = f^2$.
The functional equations can be verified by straightforward computations, and the explicit numbers can be recovered by Lagrange inversion (or extracted from their respective OEIS pages).
\end{proof}

\begin{remark}
It follows from \cref{subsec:countingTamariIntervalPosets,prop:generatingFunctionsWeepingWillows} that asymptotically, most Tamari interval posets are not weeping willows.
\end{remark}

%
%

It is also interesting to refine the enumeration of weeping willows with respect to increasing and decreasing arcs.
The proof is similar and left to the reader.

\begin{proposition}
\label{prop:bivariateGeneratingFunctionsWeepingWillows}
Let
\begin{align*}
f & = 1 + x + 2x^2 + xy + 5x^3 + 5x^2y + 2xy^2 + 14x^4 + 21x^3y + 15x^2y^2 + 5xy^3 + \cdots \\ 
g & = 1 + x + y + 2x^2 + 4xy + 2y^2 + 5x^3 + 14x^2y + 14xy^2 + 5y^3 + 14x^4 + 49x^3y + 70x^2y^2 + 49xy^3 + 14y^4 + \cdots \\
h & = 1 + x + y + 2x^2 + 4xy + y^2 + 5x^3 + 14x^2y + 9xy^2 + 2y^3 + 14x^4 + 49x^3y + 50x^2y^2 + 25xy^3 + 5y^4 + \cdots
\end{align*}
denote the ordinary generating functions of non-crossing trees oriented towards~$n$, of weeping willows, and of rooted weeping willows respectively, where~$x$ and~$y$ count the number of increasing and decreasing arcs respectively.
Then
\[
f = 1 + x f^2 \bar f,
\qquad
g = 1 + (x+y) g f \bar f,
\qquad\text{and}\qquad
h = 1 + f \bar f (x h + y f),
\]
where~$\bar f(x,y) = f(y,x)$.
\end{proposition}

\begin{proof}
Similar to the proof of \Cref{prop:generatingFunctionsWeepingWillows}.
\end{proof}


\subsection{Weeping willows and simple interval hypergraphic polytopes} 
\label{subsec:WWSIHP}

We now specialize \cref{subsec:TIPIHP} to prove \cref{prop:weepingWillows}.

\begin{proposition}
\label{prop:weepingWillows1}
The following are equivalent for a polytope~$\poly{P}$:
\begin{enumerate}[(i)]
\item $\poly{P}$ is a simple interval hypergraphic polytope,
\item $\poly{P}$ is a deformed permutahedron whose vertex digraphs are forests of weeping willows.
\item $\poly{P}$ is a simple deformed associahedron.
\end{enumerate}
\end{proposition}

\begin{proof}
Immediate from \cref{prop:TamariIntervalPosets1,prop:WWTIP}.
\end{proof}

\begin{proposition}
\label{prop:weepingWillows2}
Any weeping willow is a vertex tree of some simple interval hypergraphic polytope.
\end{proposition}

\begin{proof}
It immediately follows from \cref{prop:TamariIntervalPosets2,prop:WWTIP} that any weeping willow is the vertex tree of a simple vertex in some interval hypergraphic polytope.
However, we have to work slightly more to prove that any weeping willow appears as a vertex tree in some \emph{simple} interval hypergraphic polytope.

Consider a weeping willow~$\ww$, its transitive closure~$\less$, and the interval hypergraph~$\II_\less \eqdef \set{a^\less}{a \in [n]}$ defined in \cref{prop:inclusionPosetIntervalHypergraphsContainingTamariIntervalPoset}\, \eqref{item:existsMinimalIntervalHypergraph}. 
We just need to prove that~$\simplex_{\II_\less}$ is simple, or equivalently, that~$\II_\less$ satisfies the conditions of \cref{thm:characterizationSimpleIntervalHypergraphicPolytopes}:
\begin{enumerate}[(i)]
\item If there are~$i, j, k, \ell \in [n]$ such that~$\ell \in i^\less \cap j^\less$ and~$i^\less \cup j^\less \subseteq k^\less$, then we have~$k \less i,j \less \ell$. 
As~$\ww$ is a tree, we obtain by \cref{lem:posetTree}\,\eqref{item:diamond} that~$i$ and~$j$ are comparable in~$\less$.
Hence, we get that~$i^\less$ and~$j^\less$ are nested, so that~$i^\less \cup j^\less \in \{i^\less, j^\less\} \subseteq \c{I}_\less$.
\item If there are~$i, j, k, \ell \in [n]$ such that~$k, \ell \in i^\less \cap j^\less$, then we have~$i, j \less k, \ell$.
As~$\ww$ is a tree, we obtain by \cref{lem:posetTree}\,\eqref{item:bowtie} that either~$i$ and~$j$, or $k$ and~$\ell$ are comparable in~$\less$.
If~$i$ and~$j$ are comparable, then~$i^\less$ and~$j^\less$ are nested, so that~$i^\less \cup j^\less \in \{i^\less, j^\less\} \subseteq \c{I}_\less$.
Otherwise, we obtain that~$i^\less \cap j^\less$ is a chain in~$\less$, and we can assume that~$k$ is its minimal element.
Note that~$i \notin k^\less$ and~$j \notin k^\less$ by antisymmetry of~$\less$.
Hence, we obtained that~$i^\less \cap j^\less \subseteq k^\less$ and~$i^\less \not\subseteq k^\less$ and~$j^\less \not\subseteq k^\less$.
\qedhere
\end{enumerate}
\end{proof}


We can now prove a weaker version of \Cref{conj:nestohedralCharacterization}, in the case of \emph{interval} hypergraphic polytopes.

\begin{proposition}
\label{prop:rootedWeepingWillows1}
A deformed permutahedron has the normal fan of an interval nestohedron if and only if all its vertex digraphs are forests of rooted weeping willows.
\end{proposition}

\begin{proof}
By \Cref{prop:NestoVertexDiGraphsAreBBforests}, vertex digraphs of nestohedra are $\BB$-forests, notably rooted forests.
Thus, the vertex digraphs of interval nestohedra are forests of rooted weeping willows.

We now prove the converse.
Fix a deformed permutahedron $\pol$ whose vertex digraphs are forests of rooted weeping willows.
Thus, by \Cref{prop:TamariIntervalPosets1,prop:WWTIP}, there exists an interval hypergraph $\II$ such that $\pol$ is normally equivalent to $\simplex_{\II}$.
Moreover, since weeping willows are trees, the polytope $\simplex_{\II}$ is simple, so $\II$ satisfies the conditions of \Cref{thm:characterizationSimpleIntervalHypergraphicPolytopes}.
Consider $I, J\in \II$ with $I\cap J\ne \varnothing$. 
Without loss of generality, we assume that $i\eqdef \min I < \max J \defeq j$.
Construct the permutation $\sigma \eqdef ij X$ where $X$ is the word formed by $[n]\ssm\{i, j\}$ written in an arbitrary order, and let $\ww$ be the associated forest of weeping willows (\ie the vertex digraph of the vertex of $\simplex_{\II}$ whose normal cone contains~$\poly{C}_{\sigma}$).
Then,~$i$ is a source of $\ww$ (because $i$ is the first letter of $\sigma$).
Since $\ww$ is a vertex digraph of $\simplex_\II$, it is rooted, and node $i$ is necessarily its root.
Moreover, as $I\cap J\ne\emptyset$, the node $j$ is in the same rooted tree of the forest $\ww$ as $i$ is.
Hence, $j$ is not a source of $\ww$, and $j$ precedes every $[n] \setminus \{i,j\}$ in $\sigma$.
Thus, the couple $(i,j)$ must be an arc of $\ww$.
It follows that there exists $K \in\II$ with $\{i, j\}\subseteq K$, implying $[i, j] = (I \cup J) \subseteq K$.
By \Cref{thm:characterizationSimpleIntervalHypergraphicPolytopes}~\eqref{cond:simple1}, we get $I\cup J\in \II$.
Therefore, $\II$ is a building set, and $\pol$ is normally equivalent to the interval nestohedron~$\simplex_{\II}$.
\end{proof}

\begin{proposition}
\label{prop:rootedWeepingWillows2}
Any rooted weeping willow is a vertex tree of some interval nestohedron.
\end{proposition}

\begin{proof}
Consider a rooted weeping willow~$\ww$, its transitive closure~$\less$, and the interval hypergraph~$\II_\less \eqdef \set{a^\less}{a \in [n]}$ defined in the proof of \cref{prop:inclusionPosetIntervalHypergraphsContainingTamariIntervalPoset}\, \eqref{item:existsMinimalIntervalHypergraph}.
As~$\ww$ is rooted, for any~$a, b \in [n]$, the sets~$a^\less$ and~$b^\less$ are nested (resp.~disjoint) if the unique path between~$a$ and~$b$ in~$\ww$ is directed (resp.~not directed).
Hence,~$\II_\less$ is an interval building set and~$\ww$ is indeed a vertex tree of the interval nestohedron~$\simplex_{\II_\less}$.
\end{proof}


\subsection{A finer description}
\label{subsec:finerDescriptionWeeping}

Specializing \cref{thm:TamariIntervalPosetInIntervalHypergraphic}, we know which weeping willow appear in which simple interval hypergraphic polytope.
This yields the following examples (in each example, by convention, our interval hypergraphs contain no singleton).

\begin{example}
\label{exm:cube}
For~$\II = \set{[i,i+1]}{1 \le i < n}$, the hypergraphic polytope~$\simplex_\II$ is the cube (graphical zonotope of the path-graph), and its weeping willows (\ie vertex trees) are all orientations of the path graph considered in \cref{exm:orientationPath}.
The coordinate of the vertex associated to a given weeping willow is its out-degree vector.

\medskip
\centerline{\includegraphics[scale=1]{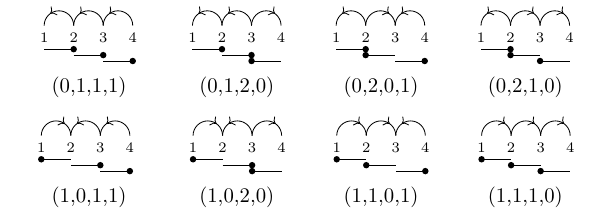}}
\end{example}

\begin{example}
\label{exm:associahedron}
For~$\II = \set{[i,j]}{1 \le i < j \le n}$, the hypergraphic polytope~$\simplex_\II$ is the associahedron~\cite{ShniderSternberg,Loday}, and its vertex trees are the rooted binary trees oriented and labeled as in \cref{exm:binarySearchTrees}.
The coordinates of the vertex associated to each rooted binary tree is the product of the number of left leaves times the number of right leaves, see \cite[Section~1]{Loday} or \Cref{rmk:coordinatesIlessJless}.

\medskip
\centerline{\includegraphics[scale=1]{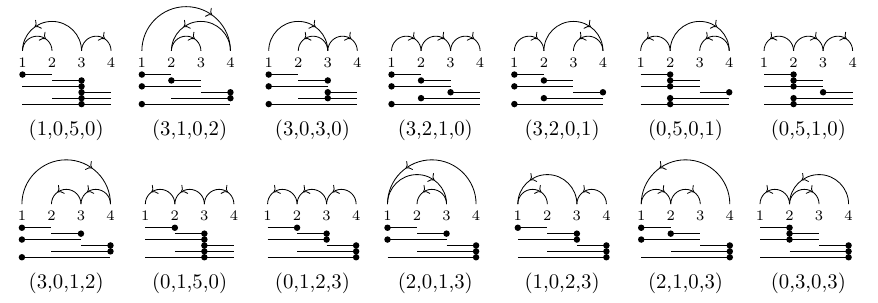}}
\end{example}

\begin{example}
\label{exm:PitmanStanley}
For~$\II = \set{[1,i]}{1 < i \le n}$, the hypergraphic polytope~$\simplex_\II$ is the Pitman--Stanley polytope~\cite{PitmanStanley}, and its vertex trees are the Pitman--Stanley trees of \cref{exm:PitmanStanleyTrees}.
The vertex~$\b v_{\ww}$ of $\simplex_\II$ corresponding to a Pitman-Stanley tree~$\ww$ is the out-degree sequence of~$\ww$. 
Through the bijection discussed in \cref{exm:PitmanStanleyTrees}, the vertex~$\b v_{\b{a}} = (v_1, \dots, v_n)$ of~$\simplex_\II$ corresponding to a binary sequence~$\b{a} \eqdef (a_1, a_2, \dots, a_{n-1}) \in \{0, 1\}^n$ is given by~$v_i = k + \ell$ where~$k$ (resp.~$\ell$) is maximal such that~$a_{i-k} = \dots = a_{i-1} = 0$ (resp.~$a_i = \dots = a_{i+\ell-1} = 1$).

\medskip
\centerline{\includegraphics[scale=1]{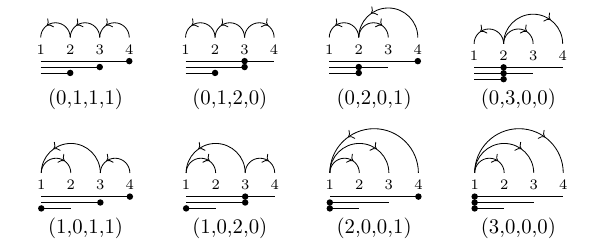}}
\end{example}

\pagebreak

\begin{example}
\label{exm:freehedron}
For~$\II = \set{[1,i],[n]\setminus[i]}{1 \le i \le n}$, the hypergraphic polytope~$\simplex_\II$ is the freehedron~\cite{Saneblidze-freehedron}, and its vertex trees are the freehedron trees of \cref{exm:freehedronTrees}.
The vertex of~$\simplex_\II$, namely~$\b v_{\ww} = (v_1, \ldots, v_n)$, corresponding to a freehedron tree~$\ww$ is given by
\[
v_i = \begin{cases}
        \text{out-degree}(i)&  \text{ if } i \text{ is not the root}, \\
        n +1 &  \text{ if } i \text{ is  the root  and } 1< i  < n,  \\
        n  &  \text{ if } i \text{ is  the root  and } i \in \{1,n\}.
\end{cases}
\]

\medskip
\centerline{\includegraphics[scale=1]{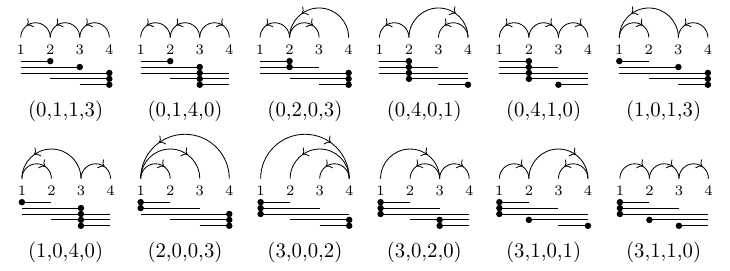}}
\end{example}

Moreover, we exploit \cref{thm:TamariIntervalPosetInIntervalHypergraphic} to give an analogue of \cref{prop:inclusionPosetIntervalHypergraphsContainingTamariIntervalPoset} for simple interval hypergraphic polytopes.
See \cref{fig:LatticeOfSlessInIless}.
Recall from \cref{thm:TamariIntervalPosetInIntervalHypergraphic} that, for a Tamari interval poset~$\less$, we defined the interval hypergraphs~$\mathdefn{\II_\less} \eqdef \set{a^\less}{a \in [n]}$ and~$\mathdefn{\JJ_\less} \eqdef \bigcup_{a \in [n]} \set{I}{a \in I \subseteq a^\less}$.
See \Cref{fig:poset_of_intervals} for illustrations of~$\II_\less$ and~$\JJ_\less$.
Recall that we call \defn{crown arcs} of~$\ww$ the arcs of~$\ww$ which are maximal for the order given by~$(a, b) \preceq (c, d)$ 
for~$\min(c,d) \le \min(a,b)$ and~$\max(a,b) \le \max(c,d)$.

\begin{proposition}
\label{prop:inclusionPosetIntervalHypergraphsContainingWeepingWillow}
We denote by~$\mathdefn{\c{S}}$ the inclusion poset of interval hypergraphs~$\II$ on~$[n]$ such that~$\simplex_\II$ is simple.
For a given weeping willow~$\ww$ with transitive closure~$\less$, denote by~$\mathdefn{\c{S}_\less}\eqdef \c I_\less \cap \c S$ the subposet of~$\c{S}$ induced by interval hypergraphs~$\II$ on~$[n]$ such that~$\ww$ is a vertex tree of~$\simplex_\II$.
Then
\begin{enumerate}[(i)]
\item $\c{S}_\less$ is order convex in~$\c{S}$.
\item The interval hypergraph~$\II_\less$ is a minimal element of~$\c{S}_\less$. Moreover, it is the unique minimal element of~$\c{S}_\less$ if and only if there are no~$1 \le a < b < c \le n$ such that~$a \lessdot b \lessdot c$, or~$a \moredot b \moredot c$, or $a \moredot b \lessdot c$, where at least one of the two cover relations is a crown arc of~$\ww$.\label{item:existsMinimalSimpleIntervalHypergraph}
\item The interval hypergraph~$\JJ_\less$ is the unique maximal element of~$\c{S}_\less$.
\end{enumerate}
\end{proposition}

\begin{example}
\cref{prop:inclusionPosetIntervalHypergraphsContainingWeepingWillow} is illustrated in \cref{fig:LatticeOfSlessInIless}.
Recall that the set~$\c{S}$ appears in \textcolor{blue}{blue} and \textcolor{green}{green} and is not order convex inside the boolean lattice~$\c{I}$.
The set~$\c{S}_\less$ contains a unique maximal element~$\JJ_\less =$ \!\drawSubset[green]{2,3,4,5,6}\!\! but three minimal elements $\II_\less =$ \!\drawSubset[green]{3}\!\!, and also \!\drawSubset[blue]{2,6}\!\! and \!\drawSubset[blue]{2,4,5}\!\!. The latter is not a minimal element of $\c I_\less$.
\end{example}

%
%

\begin{proof}[Proof of \Cref{prop:inclusionPosetIntervalHypergraphsContainingWeepingWillow}]
The order convexity of~$\c{S}_\less$ is argued as in the proof of \cref{prop:inclusionPosetIntervalHypergraphsContainingTamariIntervalPoset}.

We know that~$\II_\less \in \c{S}_\less$ from the proof of \cref{prop:weepingWillows2}.
With a very similar approach, we now prove that~$\JJ_\less$ satisfies the conditions of \cref{thm:characterizationSimpleIntervalHypergraphicPolytopes}.
Consider~$I,J \in \JJ_\less$ such that~$I \cap J \ne \varnothing$, and let~$i, j \in [n]$ be such that~$i \in I \subseteq i^\less$ and~$j \in J \subseteq j^\less$.
Then
\begin{enumerate}[(i)]
\item Assume that~$I \cap J \ne \varnothing$ and there is~$K \in \JJ_\less$ with~$I \cup J \subseteq K$. Let~$\ell \in I \cap J$ and~$k \in [n]$ be such that~$k \in K \subseteq k^\less$. 
Then we have~$k \less i, j \less \ell$. 
As~$\ww$ is a tree, we obtain by \cref{lem:posetTree}\,\eqref{item:diamond} that~$i$ and~$j$ are comparable in~$\less$, say for instance that~$i \less j$.
Then~$j^\less \subseteq i^\less$, so that~$i \in I \cup J \subseteq i^\less$, which implies that~$I \cup J \in \JJ_\less$ by construction of $\JJ_\less$.
\item Assume that~$|I \cap J| \ge 2$.
If~$i$ and~$j$ are comparable, then~$i^\less$ and~$j^\less$ are nested, and we conclude as above that~$I \cup J \in \JJ_\less$.
Otherwise, consider $k,\ell \in I \cap J$.
As $I \cap J \subseteq i^\less \cap j^\less$, there are paths from both~$i$ and~$j$ to both~$k$ and~$\ell$.
If~$k$ and~$\ell$ are incomparable, then these four paths contradict that~$\ww$ is non-crossing.
Hence, $I \cap J$ is a chain in~$\less$, with minimal element~$k_\circ$.
Note that~$i \notin k_\circ^\less$ and~$j \notin k_\circ^\less$ by antisymmetry of~$\less$.
Hence, we obtained that~$K = k_\circ^\less \in \JJ_\less$ satisfies~$I \cap J \subseteq K$ and~$I \not\subseteq K$ and~$J \not\subseteq K$.
\end{enumerate}

As~$\c{S}_\less = \c{I}_\less\cap \c S$, note that~$\II_\less$ (resp.~$\JJ_\less$) is still an inclusion minimal (resp. the \emph{unique} inclusion maximal) element of~$\c{S}_\less$ as it was already an inclusion minimal (resp.~the unique inclusion maximal) element of~$\c{I}_\less$.
We thus just need to discuss when is~$\II_\less$ inclusion-minimally unique.

Assume first that there are no~$1 \le a < b < c \le n$ such that~$a \lessdot b \lessdot c$, or $a \moredot b \moredot c$, or $a \moredot b \lessdot c$, where at least one of the two cover relations is a crown arc of~$\ww$, see \Cref{fig:GeneratingFunctionOrientedK}.
For any~$a \in [n]$, consider a path~$\pi_{\min}$ (resp.~$\pi_{\max}$) in~$\ww$ from~$a$ to~$\min(a^\less)$ (resp.~to~$\max(a^\less)$).
If~$\pi_{\min} \cup \pi_{\max}$ contains at most one arc, then any interval hypergraph containing all singletons and satisfying \cref{thm:TamariIntervalPosetInIntervalHypergraphic}\,\eqref{cond:belong2} must contain~$a^\less$.
Otherwise, either~$\pi_{\min}$ contains at least two consecutive arcs, or~$\pi_{\max}$ does, or there are two arcs of~$\ww$ incoming at~$a$.
In all three cases, our assumption ensures that no arc of~$\pi_{\min} \cup \pi_{\max}$ is a crown arc of~$\ww$.
Hence, there is a crown arc~$(k, k')$ passing above~$\pi_{\min} \cup \pi_{\max}$.
In particular, there is an interval~$K\in \II$ such that $K\supseteq[k, k']\supseteq a^\less$.
Besides, any interval hypergraph~$\II$ such that~$\simplex_\II$ is simple and admits~$\ww$ as a vertex tree must contain an interval containing the endpoints of each arc of~$\pi_{\min} \cup \pi_{\max}$ by \cref{thm:TamariIntervalPosetInIntervalHypergraphic}\,\eqref{cond:belong2}.
Thus, by \cref{thm:characterizationSimpleIntervalHypergraphicPolytopes}\,\eqref{cond:simple1}, the existence of $K$ implies that the union $X$ of these intervals, belongs to $\II$.
On the one hand, $X\supseteq[\min(a^\less),\, \max(a^\less)] = a^\less$, by construction.
On the other hand, all elements of $X$ belong to intervals whose minimum according to $\less$ is in the path $\pi_{\min}\cup\pi_{\max}$, hence all elements in $X$ are bigger than $a$, \ie $X\subseteq a^\less$.
Thus, $X = a^\less$, yielding $a^\less\in \II$.
We conclude that any such interval hypergraph~$\II$ must contain~$\II_\less$: it is the \emph{unique} inclusion minimal element of~$\c S_\less$.

Conversely, assume that there are~$1 \le a < b < c \le n$ such that~$a \lessdot b \lessdot c$ where~$a \lessdot b$ is a crown arc of~$\ww$.
Without loss of generality, we can consider that $a$ is a root of $\ww$.
Consider the interval hypergraph $\II \eqdef \II_\less \ssm \{a^\less\} \cup \{[\min(a^\less), b]\}$.
We have seen at the end of the proof of \Cref{prop:inclusionPosetIntervalHypergraphsContainingTamariIntervalPoset} that $\ww$ is a vertex tree of $\simplex_\II$, it remains to prove that $\II$ satisfies the conditions of \Cref{thm:characterizationSimpleIntervalHypergraphicPolytopes}.
First, note that $a^\less$ is inclusion maximal in $\II_\less$, and $[\min(a^\less),\, b]$ is inclusion maximal in $\II$.
As $[\min(a^\less),\, b]\subseteq a^\less$, this ensures that \Cref{thm:characterizationSimpleIntervalHypergraphicPolytopes}\,\eqref{cond:simple1} is satisfied for $\II$.
Besides, we have that for all $H\in \II$ and $j\in [n]$, there exists $k\in [n]$ such that~$H \cap j^\less = k^\less$.
Hence, $\II$ still satisfies \Cref{thm:characterizationSimpleIntervalHypergraphicPolytopes}\,\eqref{cond:simple2}.
The proof is similar if~$a \lessdot b \lessdot c$ and~$b \lessdot c$ is a crown arc (resp.~if~$a \moredot b \moredot c$ and~$b \moredot c$ is a crown arc, resp.~if~$a \moredot b \moredot c$ and~$a \moredot b$ is a crown arc, resp.~$a \moredot b \lessdot c$ and one of them is a crown arc), considering the interval hypergraph~$\II \eqdef \II_\less \ssm \{a^\less\} \cup \{[b, \max(a^\less)]\}$ (resp.~$\II_\less \ssm \{c^\less\} \cup \{[b, \max(c^\less)]\}$, resp.~$\II_\less \ssm \{c^\less\} \cup \{[\min(c^\less), b]\}$, resp.~$\II_\less \ssm \{b^\less\} \cup \{[\min(b^\less),b], [b, \max(b^\less)]\}$).
\end{proof}

\begin{remark}\label{rmk:GeneratingFunctionSpecialWW}
Let
\begin{align*}
f & = 1 + x + 3x^2 + 12x^3 + 55x^4 + 273x^5 + 1428x^6 + 7752x^7 + 43263x^8 + 246675x^9 + \cdots \\
k & = 1 + 2x + 5x^2 + 16x^3 + 67x^4 + 316x^5 + 1599x^6 + 8480x^7 + 46512x^8 + 261668x^9 + \cdots
\end{align*}
denote the ordinary generating functions of unoriented non-crossing trees and of the weeping willows described in \cref{prop:inclusionPosetIntervalHypergraphsContainingWeepingWillow}\,\eqref{item:existsMinimalSimpleIntervalHypergraph}.
Then
\[
k = 1 + (2x+x^2)f^2,
\]
and there are $\frac{2}{2n-1}\binom{3n-2}{n} + \frac{1}{2n-3}\binom{3n-5}{n-1}$ such weeping willows on $n\geq 1$ nodes.
See \cref{table:relevantCombinatorialObjects} for the first values.
The proof is left to the reader, after contemplation of \Cref{fig:GeneratingFunctionOrientedK}.
\begin{figure}[H]
	\centering
	\includegraphics[width=0.8\linewidth]{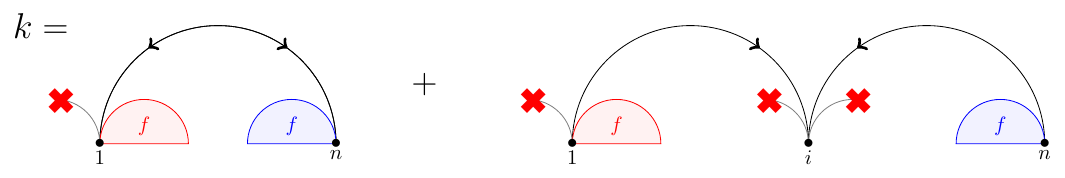}
	\caption{The two possible decomposition of the weeping willows described in \cref{prop:inclusionPosetIntervalHypergraphsContainingWeepingWillow}\,\eqref{item:existsMinimalSimpleIntervalHypergraph}.
	Red crosses \textcolor{red}{\ding{54}} indicate forbidden arcs.}
	\label{fig:GeneratingFunctionOrientedK}
\end{figure}
\end{remark}


\begin{remark}
For the transitive closure~$\less$ of a weeping willow, the condition of \cref{prop:inclusionPosetIntervalHypergraphsContainingWeepingWillow}\,\eqref{item:existsMinimalSimpleIntervalHypergraph} clearly implies the condition of \cref{prop:inclusionPosetIntervalHypergraphsContainingTamariIntervalPoset}\,\eqref{item:existsMinimalIntervalHypergraph}.
The converse however is false.
For instance, for the total order~$\less$ given by~$4 \lessdot 1 \lessdot 2 \lessdot 3$, the interval hypergraph~$\II_\less$ is not the unique minimal element of~$\c{I}_\less$, while it is the unique minimal element of~$\c{S}_\less$.
\end{remark}

\begin{proposition}
\label{prop:inclusionPosetIntervalBuildingSetsContainingRootedWeepingWillow}
We denote by~$\mathdefn{\c{N}}$ the inclusion poset of interval building sets on~$[n]$ (\ie interval hypergraphs~$\II$ containing all singletons, and such that~$I, J \in \II$ and~$I \cap J \ne \varnothing$ implies~$I \cup J \in \II$).
For a given rooted weeping willow~$\ww$ with transitive closure~$\less$, denote by~$\mathdefn{\c{N}_\less} \eqdef \c I_\less \cap \c N$ the subposet of~$\c{N}$ induced by interval building sets~$\II$ on~$[n]$ such that~$\ww$ is a vertex tree of~$\simplex_\II$.
Then~$\c{N}_\less$ is the interval between~$\II_\less$ and~$\JJ_\less$.
\end{proposition}

\begin{example}
\cref{prop:inclusionPosetIntervalBuildingSetsContainingRootedWeepingWillow} is illustrated in \cref{fig:LatticeOfSlessInIless}.
Recall that the set~$\c{N}$ appears in \textcolor{green}{green} and is not order convex inside the boolean lattice~$\c{I}$.
The set~$\c{N}_\less$ is the interval of~$\c{N}$ between~\mbox{$\II_\less =$ \!\drawSubset[green]{3}\!\!} and $\JJ_\less = $ \!\drawSubset[green]{2,3,4,5,6}\!\!.
\end{example}

%

\begin{proof}[Proof of \cref{prop:inclusionPosetIntervalBuildingSetsContainingRootedWeepingWillow}]
The order convexity of~$\c{N}_\less$ is argued as in the proof of \cref{prop:inclusionPosetIntervalHypergraphsContainingTamariIntervalPoset}.

We know that~$\II_\less$ is an interval building set from the proof of \cref{prop:rootedWeepingWillows2}.
Moreover, if~$\II$ is an interval building set satisfying \cref{thm:TamariIntervalPosetInIntervalHypergraphic}\,\eqref{cond:belong2}, then it contains an interval containing~$[\min(a,b), \max(a,b)]$ for any arc~$(a,b)$ of~$\ww$ by \cref{thm:TamariIntervalPosetInIntervalHypergraphic}\,\eqref{cond:belong2}, hence an interval containing~$[\min(a,b), \max(a,b)]$ for any directed path from~$a$ to~$b$ by the building set property.
It immediately follows that it contains~$\II_\less$.

We now prove that~$\JJ_\less$ is also an interval building set.
Consider~$I,J \in \JJ_\less$ such that~$I \cap J \ne \varnothing$, let~$i, j \in [n]$ be such that~$i \in I \subseteq i^\less$ and $j \in J \subseteq j^\less$.
As~$i^\less$ and~$j^\less$ intersect, they are nested (since $\ww$ is rooted), say~$i^\less \subseteq j^\less$.
We conclude that~$j \in I \cup J \subseteq i^\less \cup j^\less = j^\less$, so that~$I \cup J \in \JJ_\less$.
(An alternative proof is given by \cref{coro:fullHyperedgePresent}.)
As~$\c{N}_\less \subseteq \c{I}_\less$ and~$\JJ_\less$ is the unique maximal element of~$\c{I}_\less$, we conclude that~$\JJ_\less$ is also the unique inclusion maximal element of~$\c{N}_\less$.
\end{proof}


\section{Tamari interval preposets and faces of interval hypergraphic polytopes}
\label{sec:TamariIntervalPreposets}

In this section, we consider the face preposets of the interval hypergraphic polytopes.
They naturally generalize the Tamari interval posets of \cref{sec:TamariIntervalPosets}.


\subsection{Tamari interval preposets}
\label{subsec:TamariIntervalPreposets}

We first generalize the Tamari interval posets to preposets.

\begin{definition}
\label{def:TamariIntervalPreposet}
A \defn{Tamari interval preposet} is a preposet~$\bless$ on~$[n]$ with the following equivalent properties:
\begin{enumerate}[(i)]
\item for~$1 \le a < b < c \le n$, one has ${a \bless c \implies a \bless b}$ and~$a \bmore c \implies b \bmore c$, \label{item:TamIntPreposDef1}
\item for any~$a \in [n]$, the principal upper set~$\mathdefn{a^\bless} \eqdef \set{b \in [n]}{a \bless b}$ of~$a$ is an interval of~$[n]$. \label{item:TamIntPreposDef2}
\end{enumerate}
\end{definition}

\begin{remark}
The Tamari interval posets of \cref{def:TamariIntervalPoset} are precisely the Tamari interval preposets \cref{def:TamariIntervalPreposet} which are antisymmetric.
\end{remark}

\begin{remark}
Note that the Tamari interval preposets could also be defined by properties similar to~\eqref{item:TIPprop1} and~\eqref{item:TIPprop2} of \cref{def:TamariIntervalPoset}, but using the facial Tamari lattice and the facial weak order~\cite{DermenjianHohlwegPilaud}.
We skip these definitions here as they are slightly technical and not needed for our purposes.
\end{remark}

We next introduce a way to draw a Tamari interval preposet as a non-crossing diagram.

\begin{definition}
\label{def:canonicalRelation}
For~$A = \{a_1 < \dots < a_k\} \subseteq [n]$, define~$\mathdefn{\tour(A)} \eqdef \{(a_1,a_2), \dots, (a_{k-1},a_k), (a_k,a_1)\}$.
For two non-crossing subsets~$A, B\subseteq[n]$, define the \defn{canonical relation $\canoArc$} by
\begin{itemize}
\item $\canoArc \eqdef (\max A, \, \min B)$ if $A$ lies entirely to the left of $B$, \ie if~$\max A < \min B$,
\item $\canoArc \eqdef (a, \, \min B)$ if $B$ lies inside a gap of $A$, \ie if there are~$a < a'$ in~$A$ such that~${B \subseteq {]a,a'[} \subseteq [n] \ssm A}$,
\item otherwise, $\canoArc \eqdef (a,b)$ where $(b,a) \eqdef \canoArc[B][A]$.
\end{itemize}
The \defn{canonical arc} of~$(A,B)$ is the oriented arc~$a \to b$ where~$(a,b) \eqdef \canoArc$.
\end{definition}

\begin{remark}
\label{rem:drawingTIPP}
We then represent a Tamari interval preposet~$\bless$ by the union of the tours of its nodes and the canonical arcs of its cover relations, drawn above or below the horizontal axis depending on whether they are increasing or decreasing relations (as in \cref{rem:drawingTIP}).
To alleviate the figures, we draw the arcs of the tour above the horizontal axis, but the reader can imagine them drawn both above and below: drawing both still yields a non-crossing diagram.
A typical example of Tamari interval poset is illustrated in \cref{fig:TamariIntervalPreposet}.
\begin{figure}[h]
	\centerline{\includegraphics[scale=1]{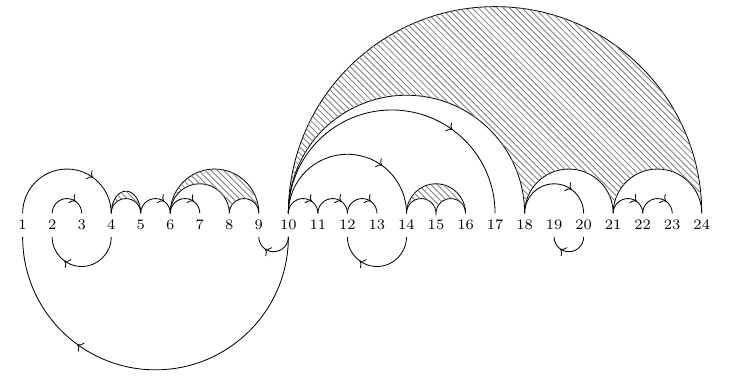}}
	\caption{A typical Tamari interval preposet.}
	\label{fig:TamariIntervalPreposet}
\end{figure}
\end{remark}


\subsection{Inclusion versus refinement}
\label{subsec:inclusionRefinementTamariIntervalPreposets}

Denote by~$\mathdefn{\poly{C}_\bless} \eqdef \set{\b{x} \in \R^n}{x_i \ge x_j \text{ if } i \bless j}$ the cone of a preposet~$\bless$ on~$[n]$.
One can compare two preposets~$\bless$ and~$\bless'$ on~$[n]$ in two natural ways:
\begin{itemize}
\item either by \defn{inclusion}: $\poly{C}_\bless \subseteq \poly{C}_{\bless'}$, that is, $i \bless' j$ implies $i \bless j$,
\item or by \defn{refinement}: $\poly{C}_\bless$ is a face of~$\poly{C}_{\bless'}$.
\end{itemize}
Note that refinement implies inclusion.
We are mostly interested in the refinement order on Tamari interval preposets, but it is unfortunately not a lattice, as illustrated in our next example.

\begin{example}
\label{exm:inclusionLatticeTamariIntervalPreposets}
Consider the four Tamari interval preposets
\[
{\bless_1} \eqdef \;\raisebox{.25cm}{\TamariIntervalPreposet{4}{4/1,4/3,1/2,3/2}{}{1}} \qquad
{\bless_2} \eqdef \;\raisebox{.25cm}{\TamariIntervalPreposet{4}{4/3,1/2}{{1,3}}{1}} \qquad
{\bless_3} \eqdef \;\raisebox{.25cm}{\TamariIntervalPreposet{4}{4/3}{{1,2,3}}{1}} \qquad
{\bless_4} \eqdef \;\raisebox{.25cm}{\TamariIntervalPreposet{4}{1/2}{{1,3,4}}{1}}
\]
Then~$\bless_1$ and~$\bless_2$ do not refine each other, and $\bless_3$ and~$\bless_4$ do not refine each other.
Moreover, $\bless_1$ and~$\bless_2$ refine $\bless_3$ and~$\bless_4$, and the relation between $\bless_2$ and $\bless_3$ is a cover relation.
Hence, $\bless_1$ and~$\bless_2$ have no join, while $\bless_3$ and~$\bless_4$ have no meet.
\end{example}

In contrast, although it is less meaningful geometrically,  the inclusion poset is a lattice

\begin{proposition}
\label{prop:inclusionLatticeTamariIntervalPreposets}
The inclusion poset of Tamari interval preposets is a lattice.
\end{proposition}

\begin{proof}
The intersection of two Tamari interval preposets is a Tamari interval preposet.
The inclusion poset of Tamari interval preposets thus admits a meet, thus is a lattice as it is bounded by the empty preposet and the full preposet.
\end{proof}


\subsection{Two families of Tamari interval preposets}
\label{subsec:examplesTamariIntervalPreposets}

Similarly as in \cref{subsec:examplesTamariIntervalPosets}, we now describe two interesting families of Tamari interval preposets, which will later appear as the face preposets of certain families of hypergraphic polytopes (see also \cref{subsec:examplesSchroderWeepingWillows} for other families where all Hasse diagrams are trees).


\subsubsection{Capped unit interval preposets}
\label{subsubsec:cappedSUnitIntervalPosets}

Continuing \cref{subsubsec:cappedUnitIntervalPosets}, we first consider the following preposets, which will take their significance in \cref{exm:cappedUnitIntervalHypergraphicPolytopePreposets}.
Recall that for a preposet~$\bless$, we denote by~$\less$ the poset on the equivalence classes of~$\bless$, and by~$\lessdot$ its cover relations.

\begin{definition}
\label{def:cappedUnitIntervalPreposets}
We call \defn{capped unit interval preposets} the preposets~$\bless$ on~$[n]$ such that
\begin{itemize}
\item $\less$ admits a unique minimum~$M$,
\item for any cover relation~$A \lessdot B$, either~$A = M$ or there is~$a \in A$ and~$b \in B$ with~$|a-b| = 1$,
\pagebreak
\item for any relation~$A \less B$, 
	\begin{itemize}
	\item $A \cap [\min(B), \max(B)] = \varnothing$,
	\item either~$A = M$ or~$[\min(A), \max(A)] \cap B = \varnothing$.
	\end{itemize}
\item for any~$A, B$ incomparable in~$\less$, we have~$[\min(A), \max(A)] \cap [\min(B), \max(B)] = \varnothing$.
\end{itemize}
\end{definition}

\begin{example}
For instance, the capped unit interval preposets for $n = 3$ are the following

\medskip
\centerline{\includegraphics[scale=1]{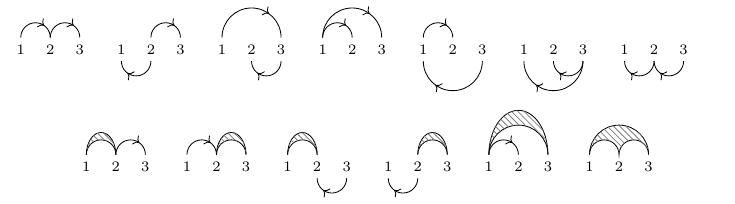}}
\end{example}

\begin{remark}
The numbers of capped unit interval preposets on~$[n]$ with~$k$ equivalence classes are given in \cref{table:cappedIntervalPreposets}.
\begin{table}[h]
	\[
	\begin{tabular}{c|ccccccc|c}
	$n \backslash k$ & $1$ & $2$ & $3$ & $4$ & $5$ & $6$ & $7$ & $\Sigma$ \\
	\hline
	$1$ & $1$ &&&&&&& $2$ \\
	$2$ & $1$ & $2$ &&&&&& $3$ \\
	$3$ & $1$ & $5$ & $7$ &&&&& $13$ \\
	$4$ & $1$ & $9$ & $26$ & $24$ &&&& $60$ \\
	$5$ & $1$ & $14$ & $63$ & $118$ & $81$ &&& $277$ \\
	$6$ & $1$ & $20$ & $125$ & $358$ & $495$ & $270$ && $1\,269$ \\
	$7$ & $1$ & $27$ & $220$ & $859$ & $1\,801$ & $1\,971$ & $891$ & $5\,770$
	\end{tabular}
	\]
	\caption{The numbers of capped unit interval preposets on $n$ with $k$ equivalence classes.}
	\label{table:cappedIntervalPreposets}
\end{table}
\end{remark}

Following \cref{def:cappedSUnitIntervalPosets}, we now consider subfamilies of capped unit interval preposets indexed by subsets~$S$ of~$[n-1]$.
Again, the significance will become clear in \cref{exm:cappedUnitIntervalHypergraphicPolytopePreposets}.

\begin{definition}
\label{def:cappedSUnitIntervalPreposets}
For~$S \subseteq [n-1]$, a \defn{capped $S$-unit interval preposet} is a capped unit interval preposet~$\bless$ such that
\begin{itemize}
\item $i$ and $i+1$ are comparable in~$\bless$ for any~$i \in S$,
\item if $i \in X$ and $i+1 \in Y$, and $X$ and~$Y$ form a cover relation in~$\less$, then $i \in S$ or~$M = X$ or~$M = Y$,
\item if $i$ and $i+1$ are equivalent in~$\bless$, then $i \in S$ or $\{i,i+1\} \subseteq M$.
\end{itemize}
\end{definition}


\subsubsection{Uniform interval preposets}
\label{subsubsec:uniformIntervalPreposets}

Continuing \cref{subsubsec:uniformIntervalPosets}, we now consider the following preposets, which will take their significance in \cref{exm:uniformIntervalHypergraphicPolytopePreposets}.

\begin{definition}
\label{def:uniformIntervalPreposets}
For $k \ge 1$, a \defn{$k$-uniform interval preposet} on~$[n]$ is a preposets~$\bless$ such that
\begin{enumerate}[(i)]
\item for each~$i \in [n-k+1]$, there is~$i \le a \le i+k-1$ such that~$a \bless j$ for all~$i \le j \le i+k-1$, and
\item for any equivalence class of~$\bless$ given as $\{a_1 < \dots < a_q\}$, and any~$1 \le p < q$, there exists $i \in [n-k+1]$ such that~$i \le a_p, a_{p+1} \le i+k-1$ and~$a_p \bless j$ for all~$i \le j \le i+k-1$,
\item for any cover relation~$X \lessdot Y$, there is~$a \in X$, $b \in Y$, and~$i \in [n-k+1]$ with~$i \le a,b \le i+k-1$ and~$a \bless j$ for all~$i \le j \le i+k-1$.
\end{enumerate}
\end{definition}

\pagebreak

\begin{example}
\label{exm:uniformIntervalPreposets}
The $k$-uniform interval preposets on~$[3]$ for~$k = 2, 3$ are the following

\medskip
\centerline{\includegraphics[scale=1]{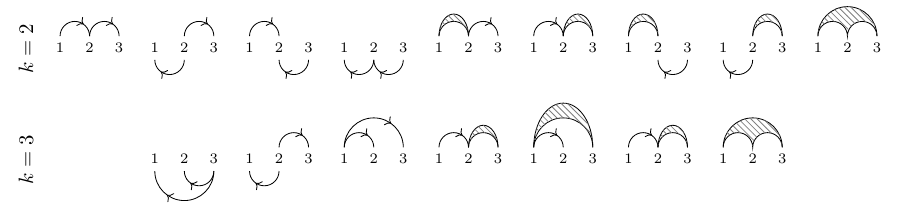}}
\end{example}


\subsection{Counting Tamari interval preposets}
\label{subsec:countingTamariIntervalPreposets}

The numbers of Tamari interval preposets and connected Tamari interval preposets on~$[n]$ with $k$ nodes are gathered in \cref{table:TamariIntervalPreposets}.
The former were enumerated in and already appeared as \OEIS{A234268}.
In contrast to~\cref{subsec:countingTamariIntervalPosets}, no simple formulas appear to exist for these numbers.

\begin{table}[h]
	\centerline{
	\begin{tabular}{c|cccccc|c}
	$n \backslash k$ & $1$ & $2$ & $3$ & $4$ & $5$ & $6$ & $\Sigma$ \\
	\hline
	$1$ & $1$ & & & & & & $1$ \\
	$2$ & $1$ & $3$ & & & & & $4$ \\
	$3$ & $1$ & $7$ & $13$  & & & & $21$ \\
	$4$ & $1$ & $12$ & $48$ & $68$ & & & $129$ \\
	$5$ & $1$ & $18$ & $116$ & $342$ & $399$ & & $876$ \\
	$6$ & $1$ & $25$ & $230$ & $1\,060$ & $2\,530$ & $2\,530$ & $6\,376$
	\end{tabular}
	\qquad
	\begin{tabular}{c|cccccc|c}
	$n \backslash k$ & $1$ & $2$ & $3$ & $4$ & $5$ & $6$ & $\Sigma$ \\
	\hline
	$1$ & $1$ & & & & & & $1$ \\
	$2$ & $1$ & $2$ & & & & & $3$ \\
	$3$ & $1$ & $5$ & $8$ & & & & $14$ \\
	$4$ & $1$ & $9$ & $31$ & $41$ & & & $82$ \\
	$5$ & $1$ & $14$ & $78$ & $213$ & $240$ & & $546$ \\
	$6$ & $1$ & $20$ & $160$ & $680$ & $1\,556$ & $1\,528$ & $3\,945$ 
	\end{tabular}
	}
	\caption{The number of Tamari interval preposets (left) and connected Tamari interval preposets (right) on~$[n]$ with $k$ nodes.}
	\label{table:TamariIntervalPreposets}
\end{table}


\subsection{Tamari interval preposets and interval hypergraphic polytopes}
\label{subsec:TIPPIHP}

Extending \cref{subsec:TIPIHP}, we connect Tamari interval preposets to interval hypergraphic polytopes.

\begin{proposition}
\label{prop:TamariIntervalPreposets}
The face preposets of an interval hypergraphic polytope are Tamari interval preposets.
Conversely, any Tamari interval preposet is a face poset of an interval hypergraphic polytope.
\end{proposition}

\begin{proof}
The following proof mimics the proof of \cref{lem:avoidingPatterns}\,\eqref{item:ap2} and \Cref{prop:TamariIntervalPosets2}.

By \Cref{prop:acyclicPreorientations}, the faces of an interval hypergraphic polytope $\simplex_\II$ are in bijection with acyclic preorientations $A$ of $\II$.
Let $\bless_A$ be the preposet of an acyclic preorientation of $\II$, fix integers~$1 \le a < b < c \le n$, and assume that~$a \bless_{A} c$.
By \Cref{prop:acyclicPreorientations}, there are~$I_1, \dots, I_k \in \II$ such that~$a \in A(I_1)$, $A(I_{i+1}) \cap I_i \neq \emptyset$ for all~$i \in [k-1]$, and~$c \in I_k$.
As $\bigcup_{i \in [k]} I_i$ is an interval containing~$a$ and~$c$ and~$a < b < c$, it also contains~$b$.
Hence, there is~$i \in [k]$ such that~$b \in I_i$, and the sequence~$I_1, \dots, I_i$ proves that~$a \bless_{A} b$.
Similarly, $a \bmore_{A} c$ implies $b \bmore_A c$.
We conclude that~$\bless_A$ is indeed a Tamari interval preposet by \Cref{def:TamariIntervalPreposet}.

Consider a Tamari interval proposet $\bless$. 
By \Cref{def:TamariIntervalPreposet}\,\eqref{item:TamIntPreposDef2}, the set $a^\bless$ is an interval for all $a \in [n]$.
Consider the orientation~$A_\bless$ of the interval hypergraph~$\II_\bless \eqdef \set{a^\bless}{a \in [n]}$ defined by~$A_\bless(a^\bless) = \set{b}{a \bless b \text{ and } b \bless a}$ for all~$a \in [n]$.
It is clear that~$A_\bless$ is acyclic, and that~$\bless_{A_\bless} =\, \bless$. 
Hence, $\bless$ is a face preposet of the interval hypergraphic polytope~$\simplex_{\II_\bless}$.
\end{proof}


\subsection{A finer description}
\label{subsec:finerDescriptionTamariIntervalPreposet}

\enlargethispage{.2cm}
Extending \cref{subsec:finerDescriptionTamariIntervalPoset}, we now refine \cref{prop:TamariIntervalPreposets} to describe which Tamari interval preposet appears as a face preposet in which interval hypergraphic polytope.

\begin{theorem}
\label{thm:TamariIntervalPreposetInIntervalHypergraphic}
A preposet~$\bless$ is a face preposet of an interval hypergraphic polytope~$\simplex_\II$ if and only if
\begin{enumerate}[(i)]
\item for any~$I \in \II$, there is~$a \in [n]$ such that~$a \in I \subseteq a^\bless$, \label{cond:belong1preposet}
\item for any equivalence class of~$\bless$ given as $\{a_1 < \dots < a_q\}$, and any~$1 \le p < q$, there is~$I \in \II$ such that~$\{a_p,a_{p+1}\} \subseteq I \subseteq a_p^\bless = a_{p+1}^\bless$, \label{cond:belong2preposet}
\item for any cover relation~$X \lessdot Y$, there is~$a \in X$, $b \in Y$, and~$I \in \II$ such that~${\{a,b\} \subseteq I \subseteq a^\bless}$. \label{cond:belong3preposet}
\end{enumerate}
\end{theorem}

\begin{proof}
Assume first that ${\bless} = {\bless_A}$ is the face preposet of the hypergraphic polytope $\simplex_\II$ corresponding to the acyclic preorientation $A$ of $\II$. Then, by \Cref{prop:acyclicPreorientations}:
\begin{enumerate}[(i)]
\item For $I \in \II$, and any~$a \in A(I) \subseteq I$, we have~$a \in I \subseteq a^\bless$.
\item For any $p\in [q-1]$, as $a_p \bless a_{p+1}$ is a cover relation, there exists $I \in \II$ such that $a_p \in A(I)$ and $a_{p+1} \in I$; similarly, as $a_p \bmore a_{p+1}$ is a cover relation, there exists $J \in \II$ such that $a_p \in J$ and $a_{p+1} \in A(J)$. As $A$ is acyclic, we deduce without loss of generality that $\{a_p, a_{p+1}\}\subseteq A(I)$, and by the previous~\eqref{cond:belong1preposet} that $A(I)\subseteq a_p^{\bless} = a_{p+1}^{\bless}$.
\item For any cover relation $X\lessdot Y$, there is $a \in X$, $b \in Y$ and $I\in \II$ such that $a\in A(I)$ and $b\in I$. In particular, $\{a,b\}\subseteq I\subseteq a^\bless$
\end{enumerate}
Hence, $\bless$ and $\II$ satisfy~\eqref{cond:belong1preposet}, \eqref{cond:belong2preposet} and~\eqref{cond:belong3preposet}.

Conversely, consider a preposet $\bless$ and an interval hypergraph $\II$ that satisfy~\eqref{cond:belong1preposet}, \eqref{cond:belong2preposet} and~\eqref{cond:belong3preposet}. We now construct an acyclic preorientation $A$ of $\II$ such that ${\bless} = {\bless_A}$.
For all~$I\in \II$, we define $A(I) = \set{a \in [n]}{a\in I \subseteq a^\bless}$.
Note that~$A(I)$ is non-empty by~\eqref{cond:belong1preposet}.
We claim that
\begin{itemize}
\item $A$ is acyclic. Otherwise, there are $I,J\in \II$ with $a\in A(I)\cap (J \ssm A(J))$ and $b\in A(J)\cap (I\ssm A(I)) $ by \Cref{prop:prelength2cycles}. By definition of~$A$, this would imply that~$a \in J \subseteq b^\bless$ and~$b \in I \subseteq a^\bless$, hence~$a^\bless = b^\bless$, so that~$a \in A(J)$ and~$b \in A(I)$: a contradiction.
\item ${\bless} = {\bless_A}$.
First, for any $a,b$ consecutive in an equivalence class of $\bless$, there exists $I\in \II$ such that $\{a, b\}\subseteq I \subseteq a^\bless = b^\bless$ by $\eqref{cond:belong2preposet}$. 
Hence, we get that $\{a, b\} \in A(I)$, and thus~$a \bless_A b$ and~$b \bless_A a$.
We conclude by transitivity that the equivalence classes of~$\bless$ and~$\bless_A$ coincide.
Second, for any cover relation~$X \lessdot Y$, there exists $a$ in $X$, $b$ in $Y$ and $I\in \II$ such that $\{a,b\} \subseteq I \subseteq a^\bless$ by~\eqref{cond:belong3preposet}. 
Hence, $a \in A(I)$ and $b \in I$ so that we get~$X \lessdot_A Y$.
\end{itemize}
We conclude that~${\bless} = {\bless_A}$ is indeed the face poset of the hypergraphic polytope~$\simplex_\II$ corresponding to the acyclic preorientation~$A$ of~$\II$.   
\end{proof}

\begin{example}
\label{exm:cappedUnitIntervalHypergraphicPolytopePreposets}
For~$S \subseteq [n-1]$, the face preposets of the hypergraphic polytope~$\simplex_{\II_S}$ of \linebreak ${\II_S \eqdef \set{[i,i+1]}{i \in S} \cup \{[n]\}}$ are precisely the capped $S$-unit interval preposets of~\cref{def:cappedSUnitIntervalPreposets}.
\end{example}


\begin{example}
\label{exm:uniformIntervalHypergraphicPolytopePreposets}
For the (complete) $k$-uniform interval hypergraph~$\II = \set{[i,i+k-1]}{i \in [n-k]}$, the face preposets of the hypergraphic polytope~$\simplex_\II$ are precisely the $k$-uniform interval preposets of~\cref{def:uniformIntervalPreposets}.
\end{example}

Finally, \cref{prop:inclusionPosetIntervalHypergraphsContainingTamariIntervalPoset} extends to face preposets as follows.
For a Tamari interval preposet~$\bless$, we denote by~\defn{$C(\bless)$} the relations of~$\bless$ given by:
\begin{itemize}
\item all equivalence relations of~$\bless$, \ie all~$a \bless b$ such that~$a \bmore b$,
\item the canonical relations~$\canoArc$ (see \cref{def:canonicalRelation}) for all cover relations~$A \lessdot B$ of the poset~$\less$ of~$\bless$.
\end{itemize}
We denote by~\defn{$U(\bless)$} the maximal relations of $C(\bless)$ for the order given by~$(a \bless b) \prec (a \bless c)$ for~$1 \le a < b < c \le n$ or~$1 \le c < b < a \le n$.

\begin{proposition}
\label{prop:inclusionPosetIntervalHypergraphsContainingTamariIntervalPreposet}
We still denote by~$\mathdefn{\c{I}}$ the inclusion poset of interval hypergraphs~$\II$ on~$[n]$ containing all singletons.
For a given Tamari interval preposet~$\bless$, denote by~\defn{$\c{I}_\bless$} the subposet of~$\c{I}$ induced by interval hypergraphs~$\II$ on~$[n]$ such that~$\bless$ is a face preposet of~$\simplex_\II$.
Then
\begin{enumerate}[(i)]
\item $\c{I}_\bless$ is order convex in~$\c{I}$.
\item The interval hypergraph~$\mathdefn{\II_\bless} \eqdef \set{a^\bless}{a \in [n]}$ is a minimal element of~$\c{I}_\bless$. Moreover, it is the unique minimal element of~$\c{I}_\bless$ if and only if there are no~$1 \le a < b < c \le n$ such that~$a \bless b \bless c$, or $a \bmore b \bmore c$, or $a \bmore b \bless c$ in~$U(\bless)$.
\item The interval hypergraph~$\mathdefn{\JJ_\bless} \eqdef \bigcup_{a \in [n]} \set{I}{a \in I \subseteq a^\bless}$ is the unique maximal element of~$\c{I}_\bless$.
\end{enumerate}
\end{proposition}

\begin{proof}
The proof is similar to that of \cref{prop:inclusionPosetIntervalHypergraphsContainingTamariIntervalPoset} and left to the reader.
\end{proof}


\pagebreak
\section{Schröder weeping willows and faces of simple interval hypergraphic polytopes}
\label{sec:SchroderWeepingWillows}

In this section, we consider the face trees of the simple interval hypergraphic polytopes.


\subsection{Schröder weeping willows}

Following \cref{def:weepingWillow,prop:WWTIP}, we now consider the following trees.
A typical example is illustrated in \cref{fig:SchroderWeepingWillow}, and four interesting families will be presented in \cref{subsec:examplesSchroderWeepingWillows}.

\begin{figure}[H]
	\centering
    \includegraphics[scale=1]{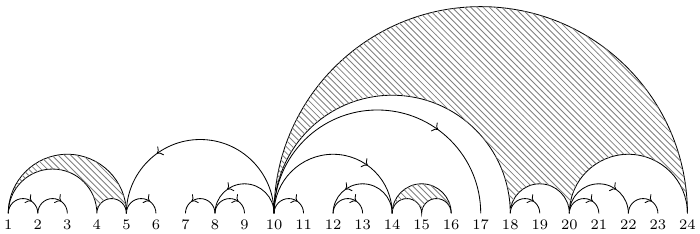}
	\caption{A typical Schröder weeping willow.}
	\label{fig:SchroderWeepingWillow}
\end{figure}

\begin{definition}
\label{def:SchroderWeepingWillow}
A \defn{Schröder weeping willow} on $[n]$ is a directed tree~$\sww$ on a partition of~$[n]$ whose associated preposet is a Tamari interval preposet.
\end{definition}

\begin{remark}
The weeping willows of \cref{def:weepingWillow} are essentially the Schröder weeping willows of \cref{def:SchroderWeepingWillow} where each node is a singleton.
\end{remark}

\begin{remark}
\label{rem:drawingSWW}
We represent the Schröder weeping willows as in \cref{rem:drawingTIPP}, except that we now gather all increasing and decreasing arcs above the horizontal axis.
It follows from \cref{prop:noncrossing} that the resulting drawing is non-crossing.
See \cref{fig:SchroderWeepingWillow}.
\end{remark}

%
%


\subsection{Inclusion versus refinement}
\label{subsec:inclusionRefinementSchroderWeepingWillows}

Recall from \cref{subsec:inclusionRefinementTamariIntervalPreposets} that one can order Tamari interval preposets either by inclusion or by refinement.
Restricting to weeping willows, the refinement poset is also the \defn{contraction poset}, \ie the transitive closure of the edge contraction on Schröder weeping willows.

Note that in contrast to \cref{prop:inclusionLatticeTamariIntervalPreposets}, the inclusion poset on Schröder weeping willow is not a lattice, as illustrated in our next example.

\begin{example}
Consider the four Schröder weeping willows
\[
{\sww_1} \eqdef \;\weepingWillow{4}{1/2,1/4,2/3}{1} \quad\;
{\sww_2} \eqdef \;\weepingWillow{4}{1/2,1/4,4/3}{1} \quad\;
{\sww_3} \eqdef \;\weepingWillow{4}{1/2,2/4,4/3}{1} \quad\;
{\sww_4} \eqdef \;\weepingWillow{4}{1/4,4/2,2/3}{1}
\]
and their corresponding Tamari interval preposets~$\bless_1, \bless_2, \bless_3$ and~$\bless_4$.
Then~$\bless_1$ and~$\bless_2$ are not included in each other, and $\bless_3$ and~$\bless_4$ are not included in each other.
Moreover, 
\[
{\bless_1} \cup {\bless_2} = \TamariIntervalPoset{4}{1/2,1/4,4/3,2/3}{1} = {\bless_3} \cap {\bless_4}
\]
which is not a Schröder weeping willow.
Hence, in the inclusion poset on Schröder weeping willows, $\sww_1$ and~$\sww_2$ have no join, while $\sww_3$ and~$\sww_4$ have no meet.
\end{example}

But now, contrarily to \cref{exm:inclusionLatticeTamariIntervalPreposets}, the contraction poset on Schröder weeping willows behaves nicely.
See \Cref{fig:SchroderWWLattice}.

\begin{figure}
    \centerline{\includegraphics[scale=1]{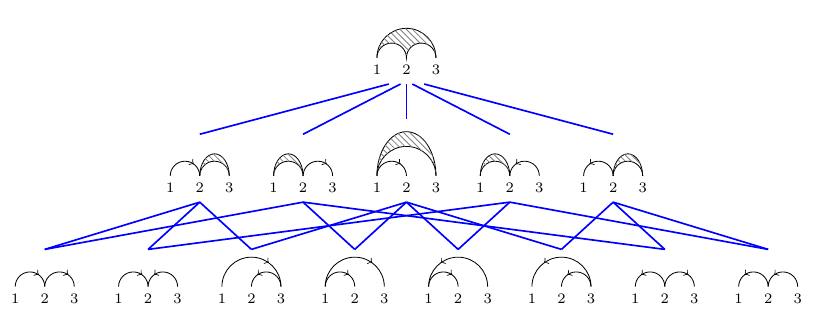}}
    \caption{The lattice of Schröder weeping willows on $3$ nodes (trimmed of its artificial minimum) admits 14 elements.}
    \label{fig:SchroderWWLattice}
\end{figure}

\begin{proposition}
The contraction poset on Schröder weeping willows, augmented by an artificial minimal element, is a lattice (that we call \defn{contraction lattice}).
\end{proposition}

\begin{proof}
As a Schr\"oder weeping willow is a tree, the cone $\poly{C}_{\bless}$ associated to some Schr\"oder weeping willow $\sww$ is a simplicial cone.
Consider two such simplicial cones $\poly{C}, \poly{C}'$, and let $\poly{D}_1, \dots, \poly{D}_s$ be the faces of both $\poly{C}$ and $\poly{C}'$ contained in the intersection $\poly{C}\cap\poly{C}'$.
As $\poly{C}$ is simplicial, the convex hull $\conv(\poly{D}_1, \dots, \poly{D}_s)$ is also a face of both $\poly{C}$ and $\poly{C}'$, and is the unique inclusion-wise maximal common face.
In particular $\conv(\poly{D}_1, \dots, \poly{D}_s)$ is the cone associated to a certain Schr\"oder weeping willow.
Consequently, any two elements of the contraction poset on Schr\"oder weeping willows admit a meet.
As this poset has a unique minimum and a unique maximum, it is a lattice.
\end{proof}

\begin{remark}
In the join semilattice on Schröder weeping willows, the atoms are the weeping willows, and the maximal element is the Schröder weeping willow with a single node~$[n]$.
\end{remark}


\subsection{Four families of Schröder weeping willows}
\label{subsec:examplesSchroderWeepingWillows}

Similarly as in \cref{subsec:examplesWeepingWillows}, we now describe four interesting families of Schröder weeping willows, which we illustrate when~$n = 3$.

\begin{example}
\label{exm:faceLatticeCube}
The weeping willows of \cref{exm:orientationPath} are the minimal elements of the contraction poset on partial orientations of a line.
For instance, when $n = 3$, we get the following (which is the face lattice of a square, trimmed of its minimum):

\centerline{\includegraphics[scale=1]{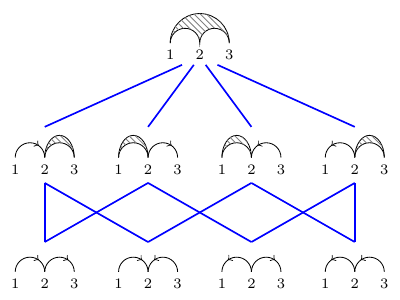}}
\end{example}

\begin{example}
\label{exm:faceLatticeAssociahedron}
The weeping willows of \cref{exm:binarySearchTrees} are the minimal elements of the contraction poset on the classical Schr\"oder trees.
For instance, when $n = 3$, we get the following (which is the face lattice of a pentagon, trimmed of its minimum):

\centerline{\includegraphics[scale=1]{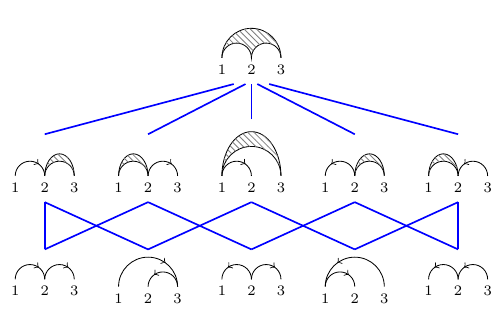}}
\end{example}

\begin{example}
\label{exm:faceLatticePitmanStanley}
The weeping willows of \cref{exm:PitmanStanleyTrees} are the minimal elements of the contraction poset on Schröder Pitman-Stanley trees.
For instance, when $n = 3$, we get the following (which is the face lattice of a square, trimmed of its minimum):

\centerline{\includegraphics[scale=1]{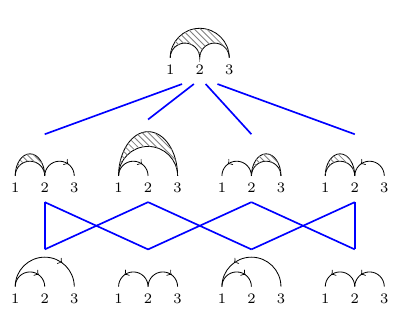}}
\end{example}

\begin{example}
\label{exm:faceLatticeFreehedron}
The weeping willows of \cref{exm:freehedronTrees} are the minimal elements of the contraction poset on Schröder freehedron trees.
For instance, when $n = 3$, we get the same contraction poset as in \cref{exm:faceLatticeAssociahedron}.
\end{example}


\subsection{Counting Schröder weeping willows}
\label{subsec:countingSchroderWeepingWillows}

\enlargethispage{.3cm}
We now exploit \cref{rem:drawingSWW} to count (rooted) Schröder weeping willows.

%
%





\begin{proposition}
\label{prop:generatingFunctionsSchroderWeepingWillows}
Let
\begin{align*}
E & = 1 + x + 4x^2 + 20x^3 + 113x^4 + 688x^5 + 4404x^6 + 29219x^7 + 199140x^8 + 1385904x^9 + \cdots \\
F & = 1 + 2x + 7x^2 + 33x^3 + 181x^4 + 1083x^5 + 6854x^6 + 45111x^7 + 305629x^8 + 2117283x^9 + \cdots \\
G & = 1 + 3x + 14x^2 + 79x^3 + 489x^4 + 3195x^5 + 21635x^6 + 150296x^7 + 1064427x^8 + 7653367x^9 + \cdots \\
H & = 1 + 3x + 13x^2 + 68x^3 + 395x^4 + 2450x^5 + 15892x^6 + 106489x^7 + 731379x^8 + 5121392x^9 + \cdots
\end{align*}
denote the ordinary generating functions of unoriented non-crossing partition trees where $n$ is in a singleton block \OEIS{A108447}, of unoriented non-crossing partition trees \OEIS{A054727}, of Schröder weeping willows, and of rooted Schröder weeping willows \OEIS{A200757} respectively, where~$x$ counts the size of the ground set minus one.
Then,
\begin{align*}
E & = 1 + x F E^2,
&&&
F & = \frac{E}{1 - x E},
\\
G & = \frac{1 + 2 x G E^2}{1 - x E},
& \text{and} &&
H & = \frac{1 + x (H + F) E^2}{1 - x E}.
\end{align*}
Hence, $E$, $F$, $G$ and $H$ are solutions of
\begin{align*}
x E^3 + x E^2 - (x+1) E + 1 & = 0,
&&&
x F^3 + (x^2-x) F^2 + (2x-1) F + 1 & = 0,
\\
(10 x - 1) G^3 + (5 x^2 - 3 x - 1) G^2 + (5x + 1) G + 1 & = 0,
&&&
x^2 H^3 + x (x - 2) H^2 + (1 - x) H - 1 & = 0.
\end{align*}
\end{proposition}

\begin{proof}
The four formulas immediately follow from the decompositions illustrated in \Cref{fig:GeneratingFunctionsSchroderEFGH} (recall that if $A$ is the generating function for a combinatorial family, then $\frac{1}{1-A}$ is the generating function for sequences of objects in this family).
\begin{figure}[h]
	\centering
	\includegraphics[width=0.999\linewidth]{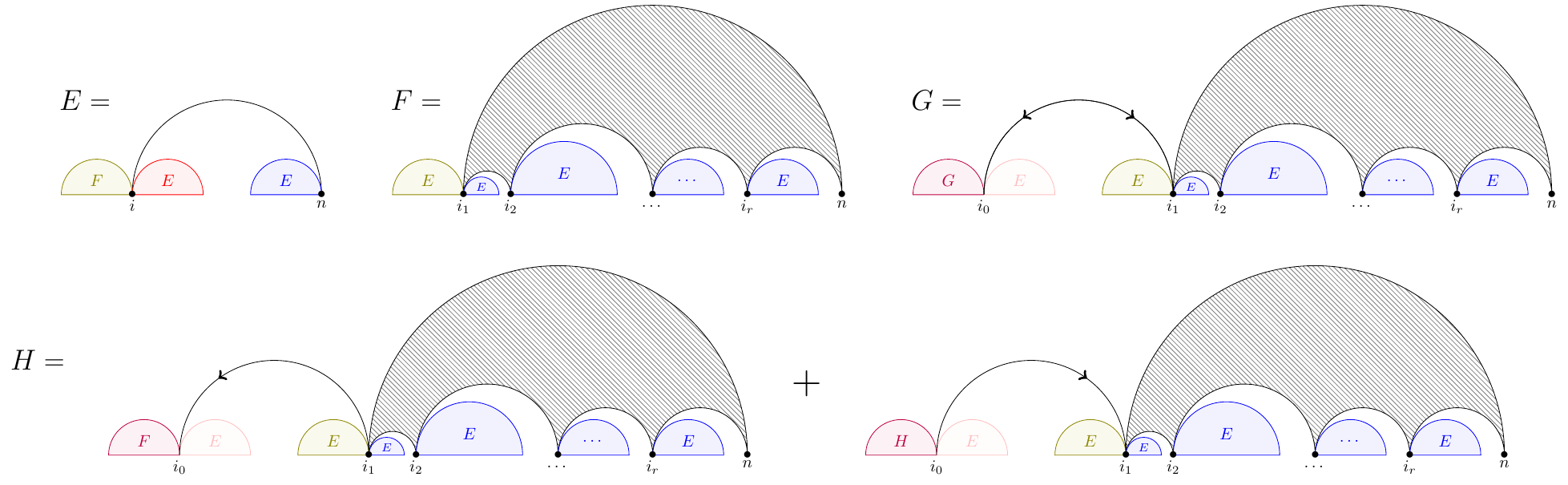}
	\caption{Decomposition for~$E$ (top left), for~$F$ (top middle), for~$G$ (top right), and for~$H$ (bottom).}
	\label{fig:GeneratingFunctionsSchroderEFGH}
\end{figure}

Substituting the formula expressing $F$ in the formula for $E$ gives the claimed annihilating polynomial for $E$.
Once this is established, one can express $F$, $G$ and $H$ as functions of $E$, then check that their respective annihilating polynomials yields multiples of the one of $E$ (alternatively, each annihilating polynomial can be obtained with the resultant method).
\end{proof}

\begin{proposition}
\label{prop:bivariateGeneratingFunctionsSchroderWeepingWillows}
Let
\begin{align*}
E & = 1 + xz + x^2z + x^3z + 3x^2z^2 + x^4z + 7x^3z^2 + x^5z + 12x^4z^2 + 12x^3z^3 + \cdots \\
F & = z + xz + x^2z + xz^2 + x^3z + 3x^2z^2 + x^4z + 6x^3z^2 + 3x^2z^3 + x^5z + 10x^4z^2 + 14x^3z^3 + \cdots \\
G & = z + x z + x^2 z + 2 x z^2 + x^3 z + 5 x^2 z^2 + 9 x^3 z^2 + 8 x^2 z^3 + x^4 z + \cdots \\
H & = z + x z + x^2 z + 2 x z^2 + x^3 z + 5 x^2 z^2 + 9 x^3 z^2 + 7 x^2 z^3 + x^4 z + \cdots
\end{align*}
denote the ordinary generating functions of unoriented non-crossing Schröder trees where $n$ is in a singleton block \OEIS{A108447}, of unoriented non-crossing Schröder trees \OEIS{A054727}, of Schröder weeping willows, and of rooted Schröder weeping willows \OEIS{A200757} respectively, where~$x$ counts the size of the ground set minus one, and $z$ counts the number of nodes (\ie the number of parts in the partition of $[x+1]$).
Then,
\begin{align*}
E & = 1 + x F E^2,
&&&
F & = \frac{z E}{1 - x E}, \\
G & = \frac{z \bigl(1 + 2 x G E^2 \bigr)}{1 - x E},
& \text{and} &&
H & = \frac{z \bigl(1 + x (H + F) E^2 \bigr)}{1 - x E}.
\end{align*}
\end{proposition}

\begin{proof}
The proof is similar as the one of \Cref{prop:generatingFunctionsSchroderWeepingWillows}, using again \Cref{fig:GeneratingFunctionsSchroderEFGH}.
\end{proof}

\begin{remark}
In particular, the $f$-vector of the contraction poset on Schröder weeping willow, \ie the number of Schröder weeping willows with ground set~$[n]$ counted according to their number of nodes, is formed by the coefficients of the polynomial (in the variable $z$) which is in factor of $x^{n-1}$ inside $G$.
For instance, for $n = 3$, the $f$-vector is $(8, 5, 1)$, as can be read off from the size of each row in \Cref{fig:SchroderWWLattice}; meanwhile, the polynomial in factor of $x^{3-1}$ inside $G$ is $8z^3 + 5z^2 + z$.
\end{remark}

\begin{remark}
One can recover $f$ (resp. $g$, resp. $h$) from \Cref{prop:generatingFunctionsWeepingWillows} by looking at the coefficients in front of $x^k z^{k+1}$ in $F$ (resp. $G$, resp. $H$) in \Cref{prop:bivariateGeneratingFunctionsSchroderWeepingWillows}.
Indeed, a weeping willow is a Schöder weeping willow with the same number of nodes and parts.
\end{remark}


\subsection{Schröder weeping willows and simple interval hypergraphic polytopes} 
\label{subsec:SWWSIHP}

Following \cref{subsec:WWSIHP}, we now specialize \cref{subsec:TIPPIHP} to connect Schröder weeping willows with face trees of simple interval hypergraphic polytopes.
The next statement extends \cref{prop:weepingWillows1,prop:weepingWillows2}.

\begin{proposition}
\label{prop:SchroderWeepingWillows}
The face preposets of a simple interval hypergraphic polytope are Schröder weeping willows.
In fact, the face lattice of a simple interval hypergraphic polytope is a graded sublattice of the contraction lattice on Schröder weeping willows.
Conversely, any Schröder weeping willow is a face poset of a simple interval hypergraphic polytope.
\end{proposition}

\begin{proof}
This immediately follows from three immediate facts:
\begin{itemize}
\item the face forests of a simple deformed permutahedron are precisely all contractions of its vertex forests, and its face lattice is isomorphic to the contraction lattice on its face forests,
\item the Schröder weeping willows are precisely the contractions of the weeping willows.
\qedhere
\end{itemize}
\end{proof}

The same arguments yield the following extension of \cref{prop:rootedWeepingWillows1,prop:rootedWeepingWillows2}.

\begin{proposition}
\label{prop:RootedSchroderWeepingWillows}
The face preposets of an interval nestohedron are rooted Schröder weeping willows.
Conversely, any rooted Schröder weeping willow is a face poset of an interval nestohedron.
\end{proposition}

\begin{example}
For instance, the contraction posets of \cref{exm:faceLatticeCube,exm:faceLatticeAssociahedron,exm:faceLatticePitmanStanley,exm:faceLatticeFreehedron} are isomorphic to the face lattices of the cube, associahedron, Pitman--Stanley polytope, and freehedron respectively.
\end{example}


\subsection{A finer description}
\label{subsec:finerDescriptionSchroderWeeping}

Finally, we extend \cref{prop:inclusionPosetIntervalHypergraphsContainingWeepingWillow,prop:inclusionPosetIntervalBuildingSetsContainingRootedWeepingWillow} to all Schröder weeping willows.
Recall from \cref{prop:inclusionPosetIntervalHypergraphsContainingTamariIntervalPreposet} that, for a Tamari interval preposet~$\bless$, we considered the interval hypergraphs $\mathdefn{\II_\bless} \eqdef \set{a^\bless}{a \in [n]}$ and~$\mathdefn{\JJ_\bless} \eqdef \bigcup_{a \in [n]} \set{I}{a \in I \subseteq a^\bless}$.
For a preposet~$\bless$, recall that we denote by~\defn{$C(\bless)$} the relations of~$\bless$ given by:
\begin{itemize}
\item all equivalence relations of~$\bless$, \ie all~$a \bless b$ such that~$a \bmore b$,
\item the canonical relations~$\canoArc$ (see \cref{def:canonicalRelation}) for all cover relations~$A \lessdot B$ of the poset~$\less$ of~$\bless$.
\end{itemize}
We denote by~\defn{$K(\bless)$} the maximal relations of~$C(\bless)$ for the order given by~$(a, b) \prec (c, d)$ if~$a,b,c,d$ are not all equivalent for~$\bless$ and~$\min(c,d) \le \min(a,b)$ and~$\max(a,b) \le \max(c,d)$,

\begin{proposition}
\label{prop:inclusionPosetIntervalHypergraphsContainingSchroderWeepingWillow}
We denote by~$\mathdefn{\c{S}}$ the inclusion poset of interval hypergraphs~$\II$ on~$[n]$ such that~$\simplex_\II$ is simple.
For a given Schröder weeping willow~$\sww$ with transitive closure~$\bless$, denote by~$\mathdefn{\c{S}_\bless} \eqdef \c I_\bless \cap \c S$ the subposet of~$\c{S}$ induced by interval hypergraphs~$\II$ on~$[n]$ such that~$\sww$ is a face tree of~$\simplex_\II$.
Then
\begin{enumerate}[(i)]
\item $\c{S}_\bless$ is order convex in~$\c{S}$.
\item The interval hypergraph~$\II_\bless$ is a minimal element of~$\c{S}_\bless$. Moreover, it is the unique minimal element of~$\c{S}_\bless$ if and only if there are no~$1 \le a < b < c \le n$ such that~$a \bless b \bless c$, or~$a \bmore b \bmore c$, or $a \bmore b \bless c$, where both relations are in~$C(\bless)$ and at least one of them is in~$K(\bless)$.
\item The interval hypergraph~$\JJ_\bless$ is the unique maximal element of~$\c{S}_\bless$.
\end{enumerate}
\end{proposition}

\begin{proof}
The proof is similar to that of \cref{prop:inclusionPosetIntervalHypergraphsContainingWeepingWillow} and left to the reader.\end{proof}

\begin{proposition}
\label{prop:inclusionPosetIntervalBuildingSetsContainingRootedSchroderWeepingWillow}
We denote by~$\mathdefn{\c{N}}$ the inclusion poset of interval building sets on~$[n]$ (\ie interval hypergraphs~$\II$ containing all singletons, and such that~$I, J \in \II$ and~$I \cap J \ne \varnothing$ implies~${I \cup J \in \II}$).
For a given rooted Schröder weeping willow~$\sww$ with transitive closure~$\bless$, denote by ${\mathdefn{\c{N}_\bless} \eqdef \c I_\bless \cap \c N}$ the subposet of~$\c{N}$ induced by interval building sets~$\II$ on~$[n]$ such that~$\sww$ is a face tree of~$\simplex_\II$.
Then~$\c{N}_\bless$ is the interval between~$\II_\bless$ and~$\JJ_\bless$.
\end{proposition}

\begin{proof}
The proof is similar to that of \cref{prop:inclusionPosetIntervalBuildingSetsContainingRootedWeepingWillow} and left to the reader.
\end{proof}


\section*{Acknowledgements}

Most of our working group came together during the January 2025 Banff workshop ``Lattice Theory''.
We are grateful to the organizers (Emily Barnard, Cesar Ceballos, Colin Defant, Osamu Iyama, and Nathan Williams) and to all participants for the friendly and inspiring atmosphere.
We also thank Gabe Udell for the suggestion to consider hypergraphs containing~$[n]$.

\bibliographystyle{alpha}
\bibliography{simpleIntervalHypergraphicPolytopes}

@article {Postnikov,
	AUTHOR = {Postnikov, Alexander},
	TITLE = {Permutohedra, associahedra, and beyond},
	JOURNAL = {Int. Math. Res. Not. IMRN},
	FJOURNAL = {International Mathematics Research Notices. IMRN},
	YEAR = {2009},
	NUMBER = {6},
	PAGES = {1026--1106},
}

@article {PostnikovReinerWilliams,
	AUTHOR = {Postnikov, Alexander and Reiner, Victor and Williams, Lauren K.},
	TITLE = {Faces of generalized permutohedra},
	JOURNAL = {Doc.~Math.},
	FJOURNAL = {Documenta Mathematica},
	VOLUME = {13},
	YEAR = {2008},
	PAGES = {207--273},
}

@inproceedings {Edmonds,
	AUTHOR = {Edmonds, Jack},
	TITLE = {Submodular functions, matroids, and certain polyhedra},
	BOOKTITLE = {Combinatorial {S}tructures and their {A}pplications ({P}roc. {C}algary {I}nternat. {C}onf., {C}algary, {A}lta., 1969)},
	PAGES = {69--87},
	PUBLISHER = {Gordon and Breach, New York},
	YEAR = {1970},
}

@article {FeichtnerSturmfels,
	AUTHOR = {Feichtner, Eva~Maria and Sturmfels, Bernd},
	TITLE = {Matroid polytopes, nested sets and {B}ergman fans},
	JOURNAL = {Port. Math. (N.S.)},
	FJOURNAL = {Portugaliae Mathematica. Nova S\'erie},
	VOLUME = {62},
	YEAR = {2005},
	NUMBER = {4},
	PAGES = {437--468},
}

@article {BenedettiBergeronMachacek,
   	 AUTHOR = {Benedetti, Carolina and Bergeron, Nantel and Machacek, John},
     	TITLE = {Hypergraphic polytopes: combinatorial properties and antipode},
   	JOURNAL = {J. Comb.},
  	FJOURNAL = {Journal of Combinatorics},
   	 VOLUME = {10},
      	YEAR = {2019},
    	NUMBER = {3},
     	PAGES = {515--544},
}

@article {BergeronPilaud,
	AUTHOR = {Bergeron, Nantel and Pilaud, Vincent},
	TITLE = {Interval hypergraphic lattices},
	JOURNAL = {European J. Combin.},
	FJOURNAL = {European Journal of Combinatorics},
	VOLUME = {132},
	YEAR = {2026},
	PAGES = {Paper No. 104285},
}

@article {PilaudPoullot2025PivotPolytope,
	AUTHOR = {Pilaud, Vincent and Poullot, Germain},
	TITLE = {Pivot Polytopes of Products of Simplices and Shuffles of Associahedra},
	JOURNAL = {Discrete Comput. Geom.},
	FJOURNAL = {Discrete \& Computational Geometry},
	YEAR = {2025},
}

@article {CarrDevadoss,
	AUTHOR = {Carr, Michael~P. and Devadoss, Satyan~L.},
	TITLE = {Coxeter complexes and graph-associahedra},
	JOURNAL = {Topology Appl.},
	FJOURNAL = {Topology and its Applications},
	VOLUME = {153},
	YEAR = {2006},
	NUMBER = {12},
	PAGES = {2155--2168},
}

@article {PitmanStanley,
	AUTHOR = {Stanley, Richard P. and Pitman, Jim},
	TITLE = {A polytope related to empirical distributions, plane trees, parking functions, and the associahedron},
	JOURNAL = {Discrete Comput. Geom.},
	FJOURNAL = {Discrete \& Computational Geometry. An International Journal of Mathematics and Computer Science},
	VOLUME = {27},
	YEAR = {2002},
	NUMBER = {4},
	PAGES = {\mbox{pp. 603--634}},
}

@article {Saneblidze-freehedron,
	AUTHOR = {Saneblidze, Samson},
	TITLE = {The bitwisted {C}artesian model for the free loop fibration},
	JOURNAL = {Topology Appl.},
	FJOURNAL = {Topology and its Applications},
	VOLUME = {156},
	YEAR = {2009},
	NUMBER = {5},
	PAGES = {897--910},
}

@article {Defant-fertilitopes,
	AUTHOR = {Defant, Colin},
	TITLE = {Fertilitopes},
	JOURNAL = {Discrete Comput. Geom.},
	FJOURNAL = {Discrete \& Computational Geometry. An International Journal of Mathematics and Computer Science},
	VOLUME = {70},
	YEAR = {2023},
	NUMBER = {3},
	PAGES = {713--752},
}

@article {Pilaud-acyclicReorientationLattices,
	AUTHOR = {Pilaud, Vincent},
	TITLE = {Acyclic reorientation lattices and their lattice quotients},
	JOURNAL = {Ann. Comb.},
	FJOURNAL = {Annals of Combinatorics},
	NOTE = {Online first},
	YEAR = {2024},
}

@article {Kim,
	AUTHOR = {Kim, Sangwook},
	TITLE = {Shellable complexes and topology of diagonal arrangements},
	JOURNAL = {Discrete Comput. Geom.},
	FJOURNAL = {Discrete \& Computational Geometry. An International Journal of Mathematics and Computer Science},
	VOLUME = {40},
	YEAR = {2008},
	NUMBER = {2},
	PAGES = {190--213},
}

@article{BjornerWachs,
	AUTHOR = {Bj{\"o}rner, Anders and Wachs, Michelle L.},
	FJOURNAL = {Journal of Combinatorial Theory. Series A},
	JOURNAL = {J. Combin. Theory Ser. A},
	NUMBER = {1},
	PAGES = {85--114},
	TITLE = {Permutation statistics and linear extensions of posets},
	VOLUME = {58},
	YEAR = {1991}
}

@article {ChatelPilaudPons,
	AUTHOR = {Chatel, Gr\'egory and Pilaud, Vincent and Pons, Viviane},
	TITLE = {The weak order on integer posets},
	JOURNAL = {Algebraic Combinatorics},
	PUBLISHER = {MathOA foundation},
	VOLUME = {2},
	NUMBER = {1},
	YEAR = {2019},
	PAGES = {1--48},
}

@article{ChatelPons,
	Author = {Ch{\^a}tel, Gr{\'e}gory and Pons, Viviane},
	Fjournal = {Journal of Combinatorial Theory. Series A},
	Journal = {J. Combin. Theory Ser. A},
	Pages = {58--97},
	Title = {Counting smaller elements in the {T}amari and {$m$}-{T}amari lattices},
	Volume = {134},
	Year = {2015}}

@article {Chapoton1,
	AUTHOR = {Chapoton, Fr\'{e}d\'{e}ric},
	TITLE = {Sur le nombre d'intervalles dans les treillis de {T}amari},
	JOURNAL = {S\'{e}m. Lothar. Combin.},
	FJOURNAL = {S\'{e}minaire Lotharingien de Combinatoire},
	VOLUME = {55},
	YEAR = {2005/07},
	PAGES = {Art. B55f, 18},
}

@article {Chapoton2,
	AUTHOR = {Chapoton, Fr\'{e}d\'{e}ric},
	TITLE = {Une note sur les intervalles de {T}amari},
	JOURNAL = {Ann. Math. Blaise Pascal},
	FJOURNAL = {Annales Math\'{e}matiques Blaise Pascal},
	VOLUME = {25},
	YEAR = {2018},
	NUMBER = {2},
	PAGES = {299--314},
}

@article {BernardiBonichon,
	AUTHOR = {Bernardi, Olivier and Bonichon, Nicolas},
	TITLE = {Intervals in {C}atalan lattices and realizers of triangulations},
	JOURNAL = {J. Combin. Theory Ser. A},
	FJOURNAL = {Journal of Combinatorial Theory. Series A},
	VOLUME = {116},
	YEAR = {2009},
	NUMBER = {1},
	PAGES = {55--75},
}

@article {FangFusyNadeau,
	AUTHOR = {Fang, Wenjie and Fusy, \'{E}ric and Nadeau, Philippe},
	TITLE = {Tamari intervals and blossoming trees},
	JOURNAL = {Comb. Theory},
	FJOURNAL = {Combinatorial Theory},
	VOLUME = {5},
	YEAR = {2025},
	NUMBER = {1},
	PAGES = {Paper No. 4, 41},
}

@book {Schoute,
 	AUTHOR = {Schoute, Pieter Hendrik},
	TITLE = {Analytical treatment of the polytopes regularly derived from the regular polytopes. {S}ection {I}: {T}he simplex.},
 	JOURNAL = {{Amst. Akad. Verh. Sect. I}},
 	FJOURNAL = {{Verhandelingen der Koninklijke (Nederlandse) Akademie van Wetenschappen, Afdeeling Natuurkunde, Sectie 1}},
 	VOLUME = {11},
	NUMBER = {3},
 	PAGES = {82 s.},
 	YEAR = {1911},
}

@article {Rado,
	AUTHOR = {Rado, Richard},
	TITLE = {An inequality},
	JOURNAL = {J. London Math. Soc.},
	FJOURNAL = {The Journal of the London Mathematical Society},
	VOLUME = {27},
	YEAR = {1952},
	PAGES = {1--6},
}

@phdthesis{Tamari,
	AUTHOR = {Tamari, Dov},
	TITLE = {Monoides pr\'eordonn\'es et cha\^ines de {M}alcev},
	SCHOOL = {Universit\'e Paris Sorbonne},
	YEAR = {1951},
}

@article {Loday,
	AUTHOR = {Loday, Jean-Louis},
	TITLE = {Realization of the {S}tasheff polytope},
	JOURNAL = {Arch.~Math.~(Basel)},
	FJOURNAL = {Archiv der Mathematik},
	VOLUME = {83},
	YEAR = {2004},
	NUMBER = {3},
	PAGES = {267--278},
}

@book {ShniderSternberg,
	AUTHOR = {Shnider, Steve and Sternberg, Shlomo},
	TITLE = {Quantum groups: From coalgebras to {D}rinfeld algebras},
	SERIES = {Series in Mathematical Physics},
	PUBLISHER = {International Press},
	ADDRESS = {Cambridge, MA},
	YEAR = {1993},
	PAGES = {592},
}

@article {PilaudSantos-quotientopes,
	AUTHOR = {Pilaud, Vincent and Santos, Francisco},
	TITLE = {Quotientopes},
	JOURNAL = {Bull. Lond. Math. Soc.},
	FJOURNAL = {Bulletin of the London Mathematical Society},
	YEAR = {2019},
	VOLUME = {51},
	NUMBER = {3},
	PAGES = {406--420}
}

@article {PadrolPilaudRitter,
	AUTHOR = {Padrol, Arnau and Pilaud, Vincent and Ritter, Julian},
	TITLE = {Shard polytopes},
	JOURNAL = {Int. Math. Res. Not. IMRN},
	FJOURNAL = {International Mathematics Research Notices. IMRN},
	YEAR = {2023},
	NUMBER = {9},
	PAGES = {7686--7796},
}

@article {PilaudSantos-brickPolytope,
	AUTHOR = {Pilaud, Vincent and Santos, Francisco},
	TITLE = {The brick polytope of a sorting network},
	JOURNAL = {European~J.~Combin.},
	FJOURNAL = {European Journal of Combinatorics},
	VOLUME = {33},
	YEAR = {2012},
	NUMBER = {4},
	PAGES = {632--662},
}

@article {PadrolPaluPilaudPlamondon,
	AUTHOR = {Padrol, Arnau and Palu, Yann and Pilaud, Vincent and Plamondon, Pierre-Guy},
	TITLE = {Associahedra for finite type cluster algebras and minimal relations between $\b{g}$-vectors},
	JOURNAL = {Proc. London Math. Soc.},
	FJOURNAL = {Proceedings of the London Mathematical Society},
	VOLUME = {127},
	YEAR = {2023},
	NUMBER = {3},
	PAGES = {513--588},
}

@article {PadrolPilaudPoullot-deformedNestohedra,
	AUTHOR = {Padrol, Arnau and Pilaud, Vincent and Poullot, Germain},
	TITLE = {Deformation cones of graph associahedra and nestohedra},
	JOURNAL = {European J. Combin.},
	FJOURNAL = {European Journal of Combinatorics},
	VOLUME = {107},
	YEAR = {2023},
	PAGES = {103594},
}

@article {BazierMatteChapelierLaguetDouvilleMousavandThomasYildirim,
	AUTHOR = {Bazier-Matte, V\'eronique and Chapelier-Laguet, Nathan and Douville, Guillaume and Mousavand, Kaveh and Thomas, Hugh and Y\i{}ld\i{}r\i{}m, Emine},
	TITLE = {{ABHY} {A}ssociahedra and {N}ewton polytopes of ${F}$-polynomials for finite type cluster algebras of simply laced finite type},
	JOURNAL = {J. Lond. Math. Soc. (2)},
	FJOURNAL = {Journal of the London Mathematical Society. Second Series},
	YEAR = {2023},
	DOI = {10.1112/jlms.12817},
}

@book {Stasheff-HSpaces,
	AUTHOR = {Stasheff, James},
	TITLE = {{$H$}-spaces from a homotopy point of view},
	SERIES = {Lecture Notes in Mathematics, Vol. 161},
	PUBLISHER = {Springer-Verlag, Berlin-New York},
	YEAR = {1970},
	PAGES = {v+95},
}

@article {Forcey-multiplihedra,
	AUTHOR = {Forcey, Stefan},
	TITLE = {Convex hull realizations of the multiplihedra},
	JOURNAL = {Topology Appl.},
	FJOURNAL = {Topology and its Applications},
	VOLUME = {156},
	YEAR = {2008},
	NUMBER = {2},
	PAGES = {326--347},
}

@article {SaneblidzeUmble-diagonals,
	AUTHOR = {Saneblidze, Samson and Umble, Ronald},
	TITLE = {Diagonals on the permutahedra, multiplihedra and associahedra},
	JOURNAL = {Homology Homotopy Appl.},
	FJOURNAL = {Homology, Homotopy and Applications},
	VOLUME = {6},
	YEAR = {2004},
	NUMBER = {1},
	PAGES = {363--411},
}

@article {ArdilaDoker,
	AUTHOR = {Ardila, Federico and Doker, Jeffrey},
	TITLE = {Lifted generalized permutahedra and composition polynomials},
	JOURNAL = {Adv. in Appl. Math.},
	FJOURNAL = {Advances in Applied Mathematics},
	VOLUME = {50},
	YEAR = {2013},
	NUMBER = {4},
	PAGES = {607--633},
}

@article {ChapotonPilaud,
	AUTHOR = {Chapoton, Fr\'{e}d\'{e}ric and Pilaud, Vincent},
	TITLE = {Shuffles of deformed permutahedra, multiplihedra, constrainahedra, and biassociahedra},
	JOURNAL = {Ann. H. Lebesgue},
	FJOURNAL = {Annales Henri Lebesgue},
	VOLUME = {7},
	YEAR = {2024},
	PAGES = {1535--1601},
}

@unpublished {BottmanPoliakova,
	AUTHOR = {Bottman, Nathaniel and Poliakova, Daria},
	TITLE = {Constrainahedra},
	NOTE = {Preprint, \href{https://arxiv.org/abs/2208.14529}{\texttt{arXiv:2208.14529}}},
	YEAR = {2022},
}

@article {AguiarArdila,
	AUTHOR = {Aguiar, Marcelo and Ardila, Federico},
	TITLE = {Hopf monoids and generalized permutahedra},
	JOURNAL = {Mem. Amer. Math. Soc.},
	FJOURNAL = {Memoirs of the American Mathematical Society},
	VOLUME = {289},
	YEAR = {2023},
	NUMBER = {1437},
}

@article {PadrolPilaudPoullot-deformedGraphicalZonotopes,
	AUTHOR = {Padrol, Arnau and Pilaud, Vincent and Poullot, Germain},
	TITLE = {Deformed graphical zonotopes},
	JOURNAL = {Discrete Comput. Geom.},
	FJOURNAL = {Discrete \& Computational Geometry. An International Journal of Mathematics and Computer Science},
	VOLUME = {73},
	YEAR = {2025},
	NUMBER = {2},
	PAGES = {447--465},
}

@article {PadrolPilaudPoullot-deformationConesHypergraphicPolytopes,
	AUTHOR = {Padrol, Arnau and Pilaud, Vincent and Poullot, Germain},
	TITLE = {Deformation cones of hypergraphic polytopes},
	JOURNAL = {S\'{e}m. Lothar. Combin.},
	FJOURNAL = {S\'{e}minaire Lotharingien de Combinatoire},
	VOLUME = {86B},
	YEAR = {2022},
	PAGES = {Art. \#72, 12 pp.},
}

@article {Rehberg,
	AUTHOR = {Rehberg, Sophie},
	TITLE = {Pruned inside-out polytopes, combinatorial reciprocity theorems and generalized permutahedra},
	JOURNAL = {Electron. J. Combin.},
	FJOURNAL = {Electronic Journal of Combinatorics},
	VOLUME = {29},
	YEAR = {2022},
	NUMBER = {4},
	PAGES = {Paper No. 4.36, 31},
}

@article {CardinalHoangMerinoMickaMutze,
	AUTHOR = {Cardinal, Jean and Hoang, Hung P. and Merino, Arturo and Mi\v{c}ka, Ond\v{r}ej and M\"{u}tze, Torsten},
	TITLE = {Combinatorial generation via permutation languages. {V}. {A}cyclic orientations},
	JOURNAL = {SIAM J. Discrete Math.},
	FJOURNAL = {SIAM Journal on Discrete Mathematics},
	VOLUME = {37},
	YEAR = {2023},
	NUMBER = {3},
	PAGES = {1509--1547},
}

@unpublished {CardinalSteiner,
	AUTHOR = {Cardinal, Jean and Steiner, Raphael},
	TITLE = {Shortest paths on polymatroids and hypergraphic polytopes},
	JOURNAL = {Comb. Theory},
	FJOURNAL = {Combinatorial Theory},
	VOLUME = {5},
	YEAR = {2025},
	NUMBER = {3},
	PAGES = {Paper No. 3, 29},
}

@article {AgnarssonMorris,
	AUTHOR = {Agnarsson, Geir and Morris, Walter D.},
	TITLE = {On {M}inkowski sums of simplices},
	JOURNAL = {Ann. Comb.},
	FJOURNAL = {Annals of Combinatorics},
	VOLUME = {13},
	YEAR = {2009},
	NUMBER = {3},
	PAGES = {271--287},
}

@article {Agnarsson,
	AUTHOR = {Agnarsson, Geir},
	TITLE = {On a special class of hyper-permutahedra},
	JOURNAL = {Electron. J. Combin.},
	FJOURNAL = {Electronic Journal of Combinatorics},
	VOLUME = {24},
	YEAR = {2017},
	NUMBER = {3},
	PAGES = {Paper No. 3.46, 25},
}

@unpublished {ABGPS,
	AUTHOR = {Abram, Antoine and Bastidas, Jose and G{\'e}linas, F{\'e}lix and Pilaud, Vincent and Sack, Andrew},
	TITLE = {Ornamentation lattices and intreeval hypergraphic lattices},
 	NOTE = {Preprint, \href{http://arxiv.org/abs/2508.01606}{\texttt{arXiv:2508.01606}}}, 
	YEAR = {2025},
}

@inproceedings {Greene,
	AUTHOR = {Greene, Curtis},
	TITLE = {Acyclic orientations},
	BOOKTITLE = {Proceedings of the NATO Advanced Study Institute held in Berlin (West Germany)},
	SERIES = {Nato Science Series C:},
	VOLUME = {31},
	PAGES = {65--68},
	PUBLISHER = {Springer Netherlands},
	YEAR = {1977},
}

@article {GreeneZaslavsky,
	AUTHOR = {Greene, Curtis and Zaslavsky, Thomas},
	TITLE = {On the interpretation of {W}hitney numbers through arrangements of hyperplanes, zonotopes, non-{R}adon partitions, and orientations of graphs},
	JOURNAL = {Trans. Amer. Math. Soc.},
	FJOURNAL = {Transactions of the American Mathematical Society},
	VOLUME = {280},
	YEAR = {1983},
	NUMBER = {1},
	PAGES = {97--126},
}

@article {Stanley-acyclicOrientations,
	AUTHOR = {Stanley, Richard P.},
	TITLE = {Acyclic orientations of graphs},
	JOURNAL = {Discrete Math.},
	FJOURNAL = {Discrete Mathematics},
	VOLUME = {5},
	YEAR = {1973},
	PAGES = {171--178},
}

@article {DeConciniProcesi,
	AUTHOR = {De Concini, Conrado and Procesi, Claudio},
	TITLE = {Wonderful models of subspace arrangements},
	JOURNAL = {Selecta Math. (N.S.)},
	FJOURNAL = {Selecta Mathematica. New Series},
	VOLUME = {1},
	YEAR = {1995},
	NUMBER = {3},
	PAGES = {459--494},
}

@article {DosenPetric,
	AUTHOR = {Do\v{s}en, Kosta and Petri\'{c}, Zoran},
	TITLE = {Hypergraph polytopes},
	JOURNAL = {Topology Appl.},
	FJOURNAL = {Topology and its Applications},
	VOLUME = {158},
	YEAR = {2011},
	NUMBER = {12},
	PAGES = {1405--1444},
}

@unpublished {Pilaud-MarioLuigi,
	AUTHOR = {Pilaud, Vincent},
	TITLE = {Plumbing bijections},
 	NOTE = {Preprint, \href{http://arxiv.org/abs/2512.05001}{\texttt{arXiv:2512.05001}}}, 
	YEAR = {2025},
}

@article {KnutsonMiller-subwordComplex,
	AUTHOR = {Knutson, Allen and Miller, Ezra},
	TITLE = {Subword complexes in {C}oxeter groups},
	JOURNAL = {Adv.~Math.},
	FJOURNAL = {Advances in Mathematics},
	VOLUME = {184},
	YEAR = {2004},
	NUMBER = {1},
	PAGES = {161--176}
}

@article {KnutsonMiller-GroebnerGeometry,
	AUTHOR = {Knutson, Allen and Miller, Ezra},
	TITLE = {Gr\"obner geometry of {S}chubert polynomials},
	JOURNAL = {Ann. of Math. (2)},
	FJOURNAL = {Annals of Mathematics. Second Series},
	VOLUME = {161},
	YEAR = {2005},
	NUMBER = {3},
	PAGES = {1245--1318}
}

@incollection {PilaudStump-ELlabeling,
	AUTHOR = {Pilaud, Vincent and Stump, Christian},
	TITLE = {{EL}-labelings and canonical spanning trees for subword complexes},
	BOOKTITLE = {Discrete Geometry and Optimization},
	PUBLISHER = {Springer},
	SERIES = {Fields Inst. Comm. Series},
	YEAR = {2013},
	PAGES = {213--248},
}

@unpublished {Pilaud-greedyFlipTree,
	AUTHOR = {Pilaud, Vincent},
	TITLE = {The greedy flip tree of a subword complex},
 	NOTE = {Preprint, \href{http://arxiv.org/abs/1203.2323}{\texttt{arXiv:1203.2323}}}, 
	YEAR = {2012},
}

@article {FroeseRenken,
	AUTHOR = {Froese, Vincent and Renken, Malte},
	TITLE = {Terrain-like graphs and the median {G}enocchi numbers},
	JOURNAL = {European J. Combin.},
	FJOURNAL = {European Journal of Combinatorics},
	VOLUME = {115},
	YEAR = {2024},
	PAGES = {Paper No. 103780, 8},
}

@article {DumontRandrianarivony,
	AUTHOR = {Dumont, Dominique and Randrianarivony, Arthur},
	TITLE = {D\'{e}rangements et nombres de {G}enocchi},
	JOURNAL = {Discrete Math.},
	FJOURNAL = {Discrete Mathematics},
	VOLUME = {132},
	YEAR = {1994},
	NUMBER = {1-3},
	PAGES = {37--49},
}

@article {JelinekTopfer,
	AUTHOR = {Jel\'{\i}nek, V\'{\i}t and T\"{o}pfer, Martin},
	TITLE = {On grounded {L}-graphs and their relatives},
	JOURNAL = {Electron. J. Combin.},
	FJOURNAL = {Electronic Journal of Combinatorics},
	VOLUME = {26},
	YEAR = {2019},
	NUMBER = {3},
	PAGES = {Paper No. 3.17, 13},
}

@inproceedings {AshurFiltserSababn,
	AUTHOR = {Ashur, Stav and Filtser, Omrit and Sababn, Rachel},
	TITLE = {Terrain-like and non-jumping graphs},
	BOOKTITLE = {Proceedings of the 35th European Workshop on Computational Geometry (EuroCG)},
	PAGES = {51},
	YEAR = {2019},
}

@article {FroeseRenken-algo,
	AUTHOR = {Froese, Vincent and Renken, Malte},
	TITLE = {A fast shortest path algorithm on terrain-like graphs},
	JOURNAL = {Discrete Comput. Geom.},
	FJOURNAL = {Discrete \& Computational Geometry. An International Journal of Mathematics and Computer Science},
	VOLUME = {66},
	YEAR = {2021},
	NUMBER = {2},
	PAGES = {737--750},
}

@mastersthesis {Hixon,
	AUTHOR = {Hixon, Thomas Stuart},
	TITLE = {Hook graphs and more : Some contributions to geometric graph theory},
	SCHOOL = {Technische Universitat Berlin},
	YEAR = {2013},
}

@article {SotoCaro,
	AUTHOR = {Soto, Mauricio and Thraves Caro, Christopher},
	TITLE = {{$p$}-box: a new graph model},
	JOURNAL = {Discrete Math. Theor. Comput. Sci.},
	FJOURNAL = {Discrete Mathematics \& Theoretical Computer Science. DMTCS.},
	VOLUME = {17},
	YEAR = {2015},
	NUMBER = {1},
	PAGES = {169--186},
}

@article {CatanzaroChaplickFelsnerHalldorssonHalldorssonHixonStacho,
	AUTHOR = {Catanzaro, Daniele and Chaplick, Steven and Felsner, Stefan and Halld\'{o}rsson, Bjarni V. and Halld\'{o}rsson, Magn\'{u}s M. and Hixon, Thomas and Stacho, Juraj},
	TITLE = {Max point-tolerance graphs},
	JOURNAL = {Discrete Appl. Math.},
	FJOURNAL = {Discrete Applied Mathematics. The Journal of Combinatorial Algorithms, Informatics and Computational Sciences},
	VOLUME = {216},
	YEAR = {2017},
	NUMBER = {part 1},
	PAGES = {84--97},
}

@unpublished {OEIS,
  	KEY = {{OEIS}},
	TITLE = {The {O}n-{L}ine {E}ncyclopedia of {I}nteger {S}equences},
	NOTE = {Published electronically at \url{http://oeis.org}},
	YEAR = {2010},
}

@misc{PadrolPoullot2025IndecomposabilityAndBeyond,
title={Indecomposability and beyond via the graph of edge dependencies},
author={Arnau Padrol and Germain Poullot},
year={2026},
eprint={2512.05307},
archivePrefix={arXiv},
primaryClass={math.CO},
url={https://arxiv.org/abs/2512.05307},
}

@misc{Poullot2025RaysDeformationConesGraphical,
title={Rays of the deformation cones of graphical zonotopes},
author={Germain Poullot},
year={2025},
eprint={2404.02669},
archivePrefix={arXiv},
primaryClass={math.CO},
url={https://arxiv.org/abs/2404.02669},
}

@article {bostan2023refined,
    AUTHOR = {Bostan, Alin and Chyzak, Fr\'ed\'eric and Pilaud, Vincent},
     TITLE = {Refined {P}roduct {F}ormulas for {T}amari {I}ntervals},
   JOURNAL = {Electron. J. Combin.},
  FJOURNAL = {Electronic Journal of Combinatorics},
    VOLUME = {33},
      YEAR = {2026},
    NUMBER = {1},
     PAGES = {P1.62},
      ISSN = {1077-8926},
   MRCLASS = {05A15 (05A19 06)},
  MRNUMBER = {5051152},
       DOI = {10.37236/14666},
       URL = {https://doi.org/10.37236/14666},
}

@article {DermenjianHohlwegPilaud,
    AUTHOR = {Dermenjian, Aram and Hohlweg, Christophe and Pilaud, Vincent},
     TITLE = {The facial weak order and its lattice quotients},
   JOURNAL = {Trans. Amer. Math. Soc.},
  FJOURNAL = {Transactions of the American Mathematical Society},
    VOLUME = {370},
      YEAR = {2018},
    NUMBER = {2},
     PAGES = {1469--1507},
      ISSN = {0002-9947,1088-6850},
   MRCLASS = {05E10 (03G10 06B99 20F55)},
  MRNUMBER = {3729508},
MRREVIEWER = {Sa\'ul\ A.\ Blanco},
       DOI = {10.1090/tran/7307},
       URL = {https://doi.org/10.1090/tran/7307},
}

@misc{CortesFrost2026DyckPathsConfigurationSpaces,
      title={Dyck Paths, Configuration Spaces and Polytopes For Linear Nakayama algebras}, 
      author={Veronica Calvo Cortes and Hadleigh Frost},
      year={2026},
      eprint={2602.04571},
      archivePrefix={arXiv},
      primaryClass={math.CO},
      url={https://arxiv.org/abs/2602.04571}, 
}
\label{sec:biblio}


\end{document}